\documentclass[12pt,leqno]{scrartcl}

\usepackage{a4wide} %% maybe not so good for the
\usepackage{amsthm}

\usepackage{amsmath}

\usepackage{amssymb}

\usepackage{amscd}

\usepackage[all]{xy}

\usepackage{times}

\usepackage{mathrsfs}

\usepackage{fancyheadings}
\usepackage{hyperref}

\theoremstyle{plain}
\newtheorem{prop}[subsection]{Proposition}
\newtheorem{lemma}[subsection]{Lemma}

\newtheorem{cor}[subsection]{Corollary}

\newtheorem{sublemma}[subsubsection]{Lemma}
\newtheorem{subthm}[subsubsection]{Theorem}
\newtheorem{subprop}[subsubsection]{Proposition}
\newtheorem{subcor}[subsubsection]{Corollary}

\theoremstyle{definition}
\newtheorem{defi}[subsection]{Definition}

\newtheorem{nota}[subsection]{Notation}

\newtheorem{subdef}[subsubsection]{Definition}
\newtheorem{subrmk}[subsubsection]{Remark}
\newtheorem{subex}[subsubsection]{Example}

\newtheorem{subnotation}[subsubsection]{Notation}

\numberwithin{equation}{subsubsection}

\renewcommand{\mathcal}{\mathscr}

\newcommand{\Nat}{\mathbb{N}}

\newcommand{\Q}{\mathbb{Q}}
\newcommand{\Z}{\mathbb{Z}}

\newcommand{\Aff}{\mathbb{A}}

\newcommand{\Hp}{{}^p\mathrm{H}}
\renewcommand{\H}{\mathrm{H}}

\DeclareMathOperator{\Hom}{Hom}
\DeclareMathOperator{\Ext}{Ext}
\DeclareMathOperator{\Homf}{\underline{Hom}}

\newcommand{\GD}{\mathbb{G}}

\newcommand{\F}{\mathcal{F}}
\newcommand{\Gf}{\mathcal{G}}
\newcommand{\Of}{\mathcal{O}}

\newcommand{\Mf}{\mathcal{M}}
\newcommand{\Bf}{\mathcal{B}}
\newcommand{\Uf}{\mathcal{U}}

\newcommand{\adj}{\mathrm{adj}} 
\newcommand{\Af}{\mathcal{A}}

\DeclareMathOperator{\can}{can} 
\DeclareMathOperator{\Cb}{\mathrm{C}^b}
\DeclareMathOperator{\Co}{\mathrm{C}}
\DeclareMathOperator{\Cpl}{\mathrm{C}^+}
\DeclareMathOperator{\Cmn}{\mathrm{C}^-}

\DeclareMathOperator{\Cech}{\mathrm{\check{C}}}
\newcommand{\Calt}[2]{\mathrm{\check{C}_{alt,#2}^{#1}}} 
\newcommand{\Cf}{\mathcal{C}}

\DeclareMathOperator{\Coker}{Coker}

\DeclareMathOperator{\D}{D} %% catégorie dérivée
\newcommand{\Df}{\mathcal{D}}

\newcommand{\Dbete}{{\mathrm{DF}_{\text{b\^ete}}}}
\DeclareMathOperator{\Kb}{\mathrm{K}^b}
\DeclareMathOperator{\Ko}{\mathrm{K}}
\DeclareMathOperator{\Kpl}{\mathrm{K}^+}
\DeclareMathOperator{\Kmn}{\mathrm{K}^-}
\DeclareMathOperator{\Nb}{\mathrm{N}^b}
\DeclareMathOperator{\Db}{\mathrm{D}^b}
\DeclareMathOperator{\Dpl}{\mathrm{D}^+}
\DeclareMathOperator{\Dpos}{\mathrm{D}^{\geq 0}}
\DeclareMathOperator{\DFb}{\mathrm{DF}^b} 
\DeclareMathOperator{\DFpl}{\mathrm{DF}^+} 
\DeclareMathOperator{\Dbc}{\mathrm{D}^b_c} 
\DeclareMathOperator{\Dbh}{\mathrm{D}^b_h} 
\DeclareMathOperator{\Dbm}{\mathrm{D}^b_m}
\DeclareMathOperator{\DFbc}{\mathrm{DF}^b_c} 
\DeclareMathOperator{\DFbh}{\mathrm{DF}^b_h} 
\DeclareMathOperator{\DFbm}{\mathrm{DF}^b_m} 
\DeclareMathOperator{\Dba}{\mathrm{D}^b_a} 
\DeclareMathOperator{\DF}{DF} %% catégorie dérivée filtrée
\DeclareMathOperator{\DFc}{DF_c} 
\DeclareMathOperator{\Fil}{Fil} %% catégorie filtrée
\DeclareMathOperator{\Filf}{Fil^f} %% catégorie filtrée
\newcommand{\DP}{{}^{w}\D}
\newcommand{\Dp}{{}^{p}\D}
\newcommand\geom{\text{geom}}
\newcommand{\If}{\mathcal{I}}
\newcommand{\Jf}{\mathcal{J}}

\DeclareMathOperator{\diag}{diag}
\newcommand{\et}{\mathrm{\acute{e}t}} %% plus accent
\newcommand{\proet}{\mathrm{pro\acute{e}t}} %% plus accent
\newcommand{\naive}{\mathrm{naive}} %% plus accent
\newcommand{\fl}{\longrightarrow}
\newcommand{\ra}{\longrightarrow}
\def\flnom#1{\stackrel{#1}{\fl}}
\newcommand{\fle}{\longmapsto}

\DeclareMathOperator{\Gr}{Gr} %% graded
\DeclareMathOperator{\Gra}{Gr} %% graded
\DeclareMathOperator{\Gal}{Gal} %% for Galois group
\newcommand{\id}{\mathrm{id}} %% morphisme identité 
\renewcommand{\Im}{\mathrm{Im}} %% replaces old \Im (Imaginary part
\DeclareMathOperator{\Ind}{Ind}
\newcommand{\iso}{\stackrel{\sim}{\fl}}
\DeclareMathOperator{\Ker}{Ker}
 
\newcommand{\Lf}{\mathcal{L}} %% système local
\newcommand{\Loc}{\mathcal{Loc}} %% système local
\newcommand{\Ltimes}{\otimes}
\DeclareMathOperator{\Ob}{\mathrm{Ob}}
\newcommand{\op}{\mathrm{op}} 

\DeclareMathOperator{\Perv}{Perv}

\newcommand\quash[1]{}
\newcommand{\real}{\mathrm{real}} %% foncteur réalisation
\newcommand{\Sch}{\mathbf{Sch}}
\newcommand{\Sets}{\mathrm{Sets}}
\newcommand{\Sgoth}{\mathfrak{S}}
\DeclareMathOperator{\Sh}{Sh}

\DeclareMathOperator{\Spec}{Spec}
\DeclareMathOperator{\Tot}{Tot}
\newcommand{\TR}{\mathfrak{TR}}

\newcommand{\ungras}{1\!\!\mkern -1mu1}
\newcommand{\var}{\mathrm{var}}
\newcommand{\X}{\mathcal{X}}
\newcommand{\Y}{\mathcal{Y}}

\newcommand\footnotevar[1]{}

\title{Mixed $\ell$-adic complexes
for schemes over number fields}
\author{Sophie Morel}

\begin{document}

\maketitle

\begin{abstract}
If $X$ is a variety over a number field, Annette Huber has defined
in \cite{Hu} a category of ``horizontal'' (or ``almost everywhere unramified'')
$\ell$-adic complexes and $\ell$-adic perverse sheaves on $X$. For such
objects, the notion of weights makes sense
(in the sense of Deligne, see \cite{De}), just as in the case of
varieties over finite
fields. However, contrary to what happens in that last case, mixed
perverse sheaves (or mixed locally constant sheaves) on $X$
do not have a weight
filtration in general, even when $X$ is a point.
The goal of this paper is to show how to avoid
this problem by working directly in the derived category of the abelian
category of perverse
sheaves that do admit a weight filtration. As an application, the methods of \cite{M}
to calculate the intermediate extension of a pure perverse sheaf apply over any finitely
generated field, and not just over a finite field.

\end{abstract}


\section{Introduction}

Let $k$ be a field of finite type over its prime subfield,
let $X$ be a separated scheme of finite type over $k$, 
and let $\ell$ be a prime number invertible in $k$. In her article
\cite{Hu} Annette Huber introduced a category $\Dbm(X)=D^b_m(X,E)$
of mixed horizontal $\ell$-adic sheaves on $X$, where $E$ is an
algebraic extension of $\Q_\ell$.
The idea of \cite{Hu}
is to consider the category of $\ell$-adic complexes on $X$ that
extend to a constructible $\ell$-adic complex on a model $\X$
of $X$ over a normal scheme $\Uf$ of finite type over $\Z$ and with field
of fractions $k$;
we also want the
morphisms between complexes to extend to $\X$. There is a
natural definition of weights (in the sense of Deligne's \cite{De}) on such
complexes, by considering their restriction to the fibers of $\X$ over
closed points of $\Uf$. So we have a notion of pure sheaves, and mixed
complexes are defined (as in \cite{De}) as those complexes whose cohomology
sheaves have a filtration with pure quotients.

By sections 2 and 3 of \cite{Hu}, the 6 operations (usual and exceptional
direct and inverse images, tensor products and internal Homs) exist on these
categories of complexes. Moreover, it is shown in 2.5 and 3.2 of \cite{Hu}
that the category $D^b_m(X)$ has a (self-dual) perverse t-structure,
whose heart $\Perv_m(X)$ is called the category of horizontal mixed perverse
sheaves on $X$.

Also, the results of chapters 4 and 5 of
Beilinson-Bernstein-Deligne's book \cite{BBD} about
the t-exactness (or perverse cohomological amplitude) of the 6 functors,
and the way these 6 functors affect weights, can be extended to our
situation thanks to
Deligne's generic base change theorem (SGA 4 1/2 [Th. finitude]
section 2), see for example 3.4 and 3.5 of \cite{Hu}.

Finally, there is a notion of weight filtration on an object of $\Perv_m(X)$
(see \cite{Hu} 3.7); it is an increasing filtration whose quotients are
pure perverse sheaves
of increasing weights. This filtration is unique if it exists (\cite{Hu}
3.8), but unfortunately
it doesn't always exist, unless $k$ is a finite field. As noted
in the remark below \cite{Hu} 3.8, the category of horizontal mixed
perverse sheaves on $X$ admitting a weight filtration is a full abelian
subcategory $\Perv_{mf}(X)$ of $\Perv_m(X)$ which is stable by subquotients,
but it is not stable by extensions.

As a consequence, if we start from a horizontal mixed perverse sheaf that
does have a weight filtration and apply some sheaf operations, then it
is not clear that the perverse cohomology sheaves of the resulting
mixed complex will still have weight filtrations. (Although we
would certainly expect that to be the case.)
For example,
this is a problem if we want to generalize the arguments of \cite{M}, that
gives among other things a formula for
the intersection complex of $X$.

The goal of this paper is to give a solution to this problem, inspired
by Beilinson's theorem that, if $k$ is a finite field, then the derived
category of $\Perv_m(X)$ is canonically equivalent to $D^b_m(X)$
(see \cite{Be1}, \cite{Be2}; note that Beilinson's result is more general).
Beilinson also gives a way to reconstruct the derived direct image
functors from their perverse versions, and formulas adapted to perverse
sheaves for the unipotent
nearby
and vanishing cycles functors. Building on this, Morihiko Saito has
shown in \cite{S0} and \cite{S} how to recover the other operations
(inverse images, tensor products and internal $\Hom$s) using only
perverse sheaves.

In this paper, we will follow the ideas of Beilinson and M. Saito to construct
all the sheaf operations on the bounded derived categories of the
categories $\Perv_{mf}(X)$. The main point, which is taken as an axiom
in \cite{S}, is the fact that these categories are stable by
perverse direct images; in section \ref{stab_Pmf}, we show how to deduce it
from Deligne's weight-monodromy theorem. Another difficulty is to state
all the compatibilities that the sheaf operations should satisfy. We have
chosen to use the formalism of crossed functors (``foncteurs crois\'es''),
originally due to Deligne and developed by Voevodsky and Ayoub. In
order to check that the constructions of Beilinson and M. Saito do fit
into this formalism, we have had to rewrite some of them. (Another
reason is that the categories $\Perv_{mf}(X)$ satisfy assumptions that are
slightly different from the axioms of \cite{S}, and so certain proofs
become simpler, and at least one proof has to be totally changed. However,
most of the constructions are very similar to the ones in \cite{S}.)

Here is a quick description of the different parts of the paper.
Section 1 is the introduction, and section 2 contains reminders about
$\ell$-adic perverse sheaves, the realization functor and a quick
summary of the beginning of
Huber's article \cite{Hu}, in particular the definition of the main
object of study $\Perv_{mf}(X)$. In section 3, we state the main results
of the paper, first informally and then using the language of crossed
functors. Section 4 gives a list of functors that obviously preserve the
categories $\Perv_{mf}(X)$. 
Section 5 contains reminders about Beilinson's construction of unipotent
nearby and vanishing cycles.
In section 6, we state the form of Deligne's
weight-monodromy theorem that we will use, and deduce the crucial fact that
perverse direct images also preserve the categories $\Perv_{mf}(X)$; we also
give an application to complexes with
support in a closed subscheme, that was already noted in 2.2.1 of
\cite{Be1} and Theorem 5.6 of \cite{S}. 
Section 7 gives the proof
of the first main theorem (Theorem \ref{thm_main1}, concerning the
existence of the four operations $f^*,f_*,f_!,f^!$), and section 8 gives
the proof of the second main theorem (Theorem \ref{thm_tensor_product},
about the existence of tensor products and internal $\Hom$s).
Section 9 shows how the results of this article imply that we can extend
the formalism of weight truncation functors defined in \cite{M}.
Finally, the appendix gives a review of filtered derived categories and
f-categories (i.e. filtered triangulated categories), with a focus on
the compatibility between the realization functor and derived functors.

Here are some conventions that will be used throughout the paper :
\begin{itemize}
\item[-] As we are considering sheaves for the \'etale topology or
pro\'etale topology, we
only care about schemes up to universal homeomorphism. So we will allow
ourselves to specify a closed subscheme of a scheme $X$ by giving only the
underlying closed subset.

\item[-] We are mostly interested in the triangulated versions of the
sheaf operations, so we will denote them without the usual ``R'''s or
``L'''s. For example, the derived direct image functors will simply be denoted
by $f_*$, and we will similarly write $f^*$, $f_!$ and $f^!$ for the
other direct and image inverse functors, seen as functors between the
triangulated categories of complexes of sheaves. The only exception we will
make is for the functor $R\Hom$ (in an abelian category), in order to distinguish it
from $\Hom$.

\item[-] All the schemes will be assumed to be excellent and separated,
and all the morphisms will be assumed to be of finite type. (We are only
interested in schemes that are of finite type over $\Z$ or over a field,
and these schemes are automatically excellent.) If we write ``scheme
over $k$'', where $k$ is a field, we will mean ``separated field of
finite type over $k$''.
Also,
the letter $\ell$ always stands for a prime number invertible over
all the schemes considered.

\item[-] If $\Cf$ is an additive category, we will denote by $\Co(\Cf)$ (resp. $\Cpl(\Cf)$, $\Cmn(\Cf)$, $\Cb(\Cf)$) the additive
category of (cohomological) complexes of objects of $\Cf$ (resp. of bounded below complexes, bounded above complexes, bounded
complexes); we also denote by $\Ko(\Cf)$ (resp. $\Kpl(\Cf)$, $\Kmn(\Cf)$, $\Kb(\Cf)$) the corresponding triangulated categories
of complexes modulo homotopy.

\item[-] If $\Af$ is a quasi-abelian category (for example of abelian category), we denote by $\D(\Af)$ (resp. $\Dpl(\Af)$, $\D^-(\Af)$,
$\Db(\Af)$) the derived category (resp. bounded below derived category, bounded above derived category, bounded derived) of $\Af$,
defined as the quotient of $\Ko(\Af)$ (resp. $\Kpl(\Af)$, $\Kmn(\Af)$, $\Kb(\Af)$) by the triangulated subcategory of exact complexes.

\end{itemize}

\tableofcontents

\section{Horizontal perverse sheaves}
\label{section_def}

In this section, we recall the definition of $\ell$-adic constructible
complexes and $\ell$-adic perverse sheaves, 
the construction of the realization
functor from the bounded derived category of perverse sheaves to
the category of constructible complexes, and finally
the definition by A. Huber of horizontal
$\ell$-adic complexes and horizontal perverse sheaves, as well as the
mixed versions.

\subsection{$\ell$-adic complexes}
\label{l-adic_complexes}

Let $X$ be a scheme and $E$ be an algebraic extension of $\Q_\ell$.
If we want to stay in the familiar framework of triangulated categories
(and avoid $\infty$-categories), there are two approaches to the
category of bounded constructible \'etale $E$-complexes on $X$ that
work at the level of generality that we need : Ekedahl's approach in \cite{Ek}
(see also Fargues's paper \cite{Far} for some complements)
and the Bhatt-Scholze definition via the pro\'etale site in \cite{BS}.
The second works in a more general setting, and it is known to be
equivalent to the first when they both apply.
While we could make the constructions that we need work with both
approaches, we
will mostly stick to the Bhatt-Scholze approach, because it makes the
homological algebra simpler.

\begin{subrmk} In his article \cite{Ek}, Ekedahl makes the assumption that
the scheme $X$ is of finite type over a regular scheme of dimension $\leq 1$.
The reason for this is that the necessary theorems for torsion \'etale
sheaves were only available in this setting at the time. Since then, Gabber
has proved the finiteness theorem (see Expos\'e XIII of \cite{Ga}), the
absolute purity theorem (see \cite{Fuji} or III.3 of \cite{Ga}) 
and the existence of a
dualizing complex (see Expos\'e XVII of \cite{Ga}) in the more general
setting considered here, so the results of \cite{Ek} extend to this setting.

\end{subrmk}

Let us review quickly the construction of
Bhatt and Scholze via the pro\'etale
site $X_\proet$ of $X$ (see \cite{BS}) :
The category $\Dbc(X,E)$ is defined as
a the full subcategory of the category $\D(X_\proet,E)$ of sheaves of
$E$-modules on $X_\proet$ (definition 4.1.1 of \cite{BS})
whose objects are bounded complexes with constructible cohomology sheaves,
where a pro\'etale sheaf $\F$ of $E$-vector spaces is called \emph{constructible}
if $X$ has a finite stratification $(Z_i)_{i\in I}$ by locally closed subschemes
such that each
$\F_{|Z_i}$ is lisse, i.e. locally (in the pro\'etale topology)
free of finite rank;
see definitions 6.8.6 and 6.8.8 of \cite{BS}. By propositions
5.5.4, 6.6.11, 6.8.11 and 6.8.14 of \cite{BS},
this category is canonically equivalent to the one defined by Ekedahl
if $X$ satisfies condition (A) or (B) of \cite{BS} 5.5.1.
\footnote{Technically, Ekedahl only defines the category $\Dbc(X,\Of_E)$ for
a finite extension $E$ of $\Q_\ell$, so to be precise, we should say
that the category $\Dbc(X,E)$ of Bhatt-Scholze is canonically
equivalent to the inverse $2$-limit over all finite subextensions of
$E$ of the tensor product over $\Of_E$ of $E$ and of the category of
constructible $\Of_E$-complexes defined
by Ekedahl.}
This point of view is conceptually simpler 
and has the advantage that direct images, tensor
products and internal Homs are the restriction of actual derived functors
on the categories $\D(X_\proet,E)$.

The six operations on the categories $\Dbc(X,E)$
(direct and inverses images, direct images with
proper support, exceptional inverse images, derived tensor products and
derived internal Homs) are constructed in sections 6.7 and 6.8 of
\cite{BS}. Suppose that we are given a dimension function $\delta$ on $X$
(see D\'efinition XVII.2.1.1 of \cite{Ga}), and let $K_X$ be the corresponding
dualizing complex on $X$. 
By this, we mean a potential dualizing complex on $X$ for the dimension
function $\delta$
(see D\'efinition XVII.2.1.2 of \cite{Ga}); this is known to be unique
up to unique isomorphism (Th\'eor\`eme XVII.5.1.1 of \cite{Ga}) and
to be a dualizing complex (Th\'eor\`eme XVII.6.1.1 of \cite{Ga}).
We then denote by $D_X=\Homf_X(.,K_X)$ the duality functor defined
by $K_X$; it satisfies all the usual properties, see Lemma 6.7.20 of \cite{BS}.

The category $\Dbc(X,E)$ has a canonical t-structure,
whose heart is the category $\Sh_c(X,E)$ of
constructible sheaves (this is automatic if we use definition
6.8.8 of \cite{BS} for $\Dbc(X,E)$).
This category has 
a full abelian subcategory stable by extensions $\Loc(X,E)$,
the category of lisse sheaves (or locally constant sheaves, or local
systems), see definition 6.8.3 of \cite{BS}.
We will only use
the category $\Loc(X,E)$ if $X$ is connected regular; 
in that case (and more generally if
$X$ is geometrically unibranch), this category is equivalent
to the category of continuous representations of the \'etale
fundamental group $\pi^\et(X)$
of $X$ on finite-dimensional $E$-vector spaces (see
lemmas 7.4.7 and 7.4.10 and remark 7.4.8 of \cite{BS}; the equivalence
is given by taking stalks at a geometric point of $X$). 
In particular, if $X$ is smooth of relative dimension $d$ over
a field $k$ and if we use the dimension function $\delta:x\fle\dim(\overline{\{x\}})$
(see the beginning of section \ref{def_perverse}), then
for every $\Lf\in\Ob(\Loc(X,E))$ corresponding
to a representation of $\pi^\et(X)$, if we denote
by $\Lf^\vee$ the lisse sheaf corresponding to the dual representation,
then $D_X(\Lf)\simeq\Lf^\vee(-d)[-2d]$. Indeed, we have $K_X=\underline{E}_X
(-d)[-2d]$, hence $\H^0\Homf_X(\Lf,K_X)=\Lf^\vee(-d)[-2d]$, and all
the $\H^i\Homf_X(\Lf,K_X)$ for $i\geq 1$ vanish by
exercise III.1.31 in \cite{Milne} (and lemma 6.7.13 of \cite{BS}).

\subsection{Perverse sheaves}
\label{def_perverse}

In this section, we assume that $X$ satisfies the conditions
of Corollaire XIV.2.4.4 of \cite{Ga} (for example, $X$ is of finite
type over $\Z$ or over a field)
and we
fix the dimension function $\delta$ on $X$
defined by $\delta(x)=\dim(\overline{\{x\}})$.
As explained
in \ref{l-adic_complexes}, it determines a dualizing complex $K_X$ and a duality
functor $D_X$. 
We define two full subcategories $\Dp^{\leq 0}$ and $Dp^{\geq 0}$ 
of $\Dbc(X,E)$ by the
following formulas:
\[\Dp^{\leq 0}=\{K\in\Ob\Dbc(X,E)|\forall x\in X,\ \forall i>-\delta(x),
\ \H^i(i_x^* K)=0\}\]
and
\[\Dp^{\geq 0}=\{K\in\Ob\Dbc(X,E)|\forall x\in X,\ \forall i<-\delta(x)
,\ \H^i(i_x^! K)=0\},\]
where, for every point $x$ of $X$ (not necessarily closed),
we denote the inclusion $x\fl X$ by $i_x$.
This is a t-structure on $\Dbc(X,E)$ for the same reasons as in \cite{BBD}
2.2.9-2.2.19: We consider couples $(\mathcal{S},\mathcal{L})$, where
where $\mathcal{S}$ is a finite stratification
of $X$ by locally closed connected regular subschemes,
and $\mathcal{L}$ is the data, for each stratum $Z$ of
$\mathcal{S}$, of
a finite set $\mathcal{L}(Z)$
of lisse sheaves on $Z$, such that condition (c) of \cite{BBD}
2.2.9 is satisfied. 
We denote by $\D_{(\mathcal{S},\mathcal{L})}(X,E)$ the full
subcatgeory of $\Dbc(X,E)$ whose objects are the complexes
$K$ such that, for each stratum $Z$ of $\mathcal{S}$ and each
$i\in\Z$, $\H^iK_{|Z}$ is isomorphic to an element of $\mathcal{L}(Z)$.
The categories $(\Dp^{\leq 0},\Dp^{\geq 0})$ induces a t-structure on 
$\D_{(\mathcal{S},\mathcal{L})}(X,E)$ by gluing, as in \cite{BBD} 1.4.
Then we note that the category $\Dbc(X,E)$ is the filtered inductive limit of
its subcategories $\D_{(\mathcal{S},\mathcal{L})}(X,E)$, and that the t-structures
are compatible thanks to the purity theorem (\cite{Fuji} or
XVI.3 of \cite{Ga}).

We will call the t-structure $(\Dp^{\leq 0},\Dp^{\geq 0})$ the
\emph{perverse t-structure} on $\Dbc(X,E)$, and denote its heart
by $\Perv(X,E)$. This is the category of \emph{perverse sheaves} on $X$
(with coefficients in $E$). We denote the associated cohomology
functor by $\Hp^i:\Dbc(X,E)\fl\Perv(X,E)$.

Let us list the exactness properties of the (derived) sheaf operations
for this t-structure. 

Suppose that we have a flat morphism of finite type $X\fl S$.
The following
proposition is an immediate consequence of the definitions.

\begin{subprop} Let $u:T\fl S$ be an \'etale map
(resp. the inclusion of the generic point of $S$),
and consider the functor
$u^*:\Dbc(X,E)\fl\Dbc(X\times_S T,E)$.

Then $u^*$ (resp. $u^*[-\dim S]$)
is t-exact.

\label{prop_restr_fg}
\end{subprop}

Then we recall the properties proved in sections 4.1 and 4.2 of
\cite{BBD}.

\begin{subprop} 
Let $f:X\fl Y$ be a finite type morphism.
Then:

\begin{itemize}
\item[(i)] The functors $D_X$ and $D_Y$ are t-exact.
\item[(ii)] If $f$ is affine, then $f_*$ is right t-exact and $f_!$
is left t-exact.
\item[(iii)] If the dimension of the fibers of $f$ is $\leq d$, then
$f_*$ (resp. $f_!$, resp. $f^*$, resp. $f^!$) is of perverse
cohomological amplitude $\geq -d$ (resp. $\leq d$, resp.
$\leq d$, resp. $\geq -d$).
\item[(iv)] If $f$ is quasi-finite and affine, then $f_*$ and $f_!$
are t-exact.
\item[(v)] If $f$ is smooth of relative dimension $d$, then $f^!\simeq
f^*[2d](d)$, and $f^*[d]$ and $f^![-d]$ are t-exact.
In particular, if $f$ is \'etale, then $f^*=f^!$ is t-exact.
\item[(vi)] The external tensor product $\boxtimes:\Dbc(X,E)\times
\Dbc(Y,E)\fl\Dbc(X\times Y,E)$ is t-exact.
\item[(vii)] The Tate twist functor $K\fle K(1)$ is t-exact.

\end{itemize}
\label{prop_exactness}
\end{subprop}

Remember that the external tensor product
of $K\in\Ob\Dbc(X,E)$ and $L\in\Dbc(Y,E)$ is the object
$K\boxtimes L$ of $\Dbc(X\times Y,E)$ defined by
\[K\boxtimes L=(pr_X^* K)\otimes(pr_Y^* L),\]
where $pr_X:X\times Y\fl X$ and $pr_Y:X\times Y\fl Y$ are the two projections
(see SGA 5 III 1.5, where $K\boxtimes L$ is denoted
$K\otimes_{\Spec k}L$).

\begin{proof} For (i), note that, by Theorem 6.3(iii) of \cite{Ek},
for all $K,L\in\Ob\Dbc(X,E)$ and every $x\in X$, we have a canonical isomorphism
\[i_x^!\Homf_X(K,L)\simeq \Homf_x(i_x^*K,i_x^!L).\]
Applying this to $L=K_X$ and using the isomorphisms
$i_x^! K_X\simeq E(\delta(x))[2\delta(x)]$ that are part of the
definition of a potential dualizing complex (see D\'efinition
XVII.2.1.2 of \cite{Ga}), we see that 
$K\in\Ob(Dp^{\leq 0})$ if and only if $D_X(K)\in\Ob(\Dp^{\geq 0})$. 
As $D_X^2\simeq\id_{\Dbc(X,E)}$, this also
implies that $K\in\Ob(\Dp^{\geq 0})$ if and only if $D_X(K)\in\Ob(\Dp^{\leq 0})$,
and we are done.

Point (ii) is
proved exactly as Th\'eor\`eme 4.1.1 of \cite{BBD}, as soon as we have
the analogue of Artin's vanishing theorem, which is proved in
Expos\'e XV of \cite{Ga}. Point (iii) is proved exactly as 4.2.4 of
\cite{BBD}, and (iv) follows from (ii) and (iii). To prove point
(v), it suffices to prove the isomorphism $f^!\simeq f^*[2d](d)$; but
this is SGA 4 XVIII 3.2.5. Point (vi) is proved as in Proposition 4.2.8
of \cite{BBD}. Finally, point(vii) follows from (vi), because $K(1)$ is the exterior
tensor product of the complex $K$ on $X$ and of the perverse sheaf $\Q_\ell(1)$ on
$\Spec k$.
\end{proof}

We define the intermediate extension functor as in \cite{BBD}: if
$j:U\fl X$ is a locally closed
immersion and $K$ is an object of $\Perv(U,E)$, then
\[j_{!*}K=\Im(\Hp^0 j_! K\fl\Hp^0 j_* K).\]

The methods of section 4.3 of \cite{BBD} adapt immediately to our case and
give the following result:

\begin{subthm} The category
$\Perv(X,E)$ is Artinian and Noetherian, that is, all its objects have
finite length. Moreover, the simple objects are of the form
$j_{!*}L[d]$, where $j:Z\fl X$ is a locally closed immersion, the
subscheme $Z$ is connected regular of dimension $d$, and $L$ is a lisse
sheaf on $Z$ corresponding to an irreducible representation
of $\pi_1^\et(Z)$ (we call such a $L$ a \emph{simple} lisse sheaf).

\end{subthm}

\subsection{Perverse t-structures and the realization functor}
\label{section_perverse}

Fix schemes $X,Y$ as in section \ref{def_perverse}.
By Theorem \ref{thm_real} and Section~\ref{f-horizontal}, we have a triangulated \emph{realization functor}
$\real:\Db\Perv(X,E)\fl\Dbc(X,E)$ extending the inclusion $\Perv(X,E)\subset\Dbc(X,E)$, and similarly for
$\Dbc(Y,E)$.

The proof of the following proposition is explained 
on page~\pageref{perverse_test} in Section~\ref{f-horizontal}.

\begin{subprop}
Let $f:X\fl Y$ be a finite type morphism.
\begin{itemize}
\item[(i)] If $f$ is quasi-finite and affine, then we have
commutative diagrams 
(up to
natural isomorphism)
\[\xymatrix{\Db(\Perv(X,E))\ar[r]^-{f_*}\ar[d]_{\real} & 
\Db(\Perv(Y,E))\ar[d]^{\real} \\
\Dbc(X,E)\ar[r]_-{f_*} & \Dbc(Y,E)}\quad\mbox{and}\quad
\xymatrix{\Db(\Perv(X,E))\ar[r]^-{f_!}\ar[d]_{\real} & 
\Db(\Perv(Y,E))\ar[d]^{\real} \\
\Dbc(X,E)\ar[r]_-{f_!} & \Dbc(Y,E)}\]
\item[(ii)] If $f$ is smooth and of relative dimension $d$, then we have
a commutative diagram
(up to
natural isomorphism)
\[\xymatrix{\Db(\Perv(Y,E))\ar[r]^-{f^*[d]}\ar[d]_{\real} & 
\Db(\Perv(X,E))\ar[d]^{\real} \\
\Dbc(Y,E)\ar[r]_-{f^*[d]} & \Dbc(X,E)}\]
\item[(iii)] We have
a commutative diagram
(up to
natural isomorphism)
\[\xymatrix{\Db(\Perv(X,E))^\op\ar[r]^-{D_X}\ar[d]_{\real} & 
\Db(\Perv(X,E))\ar[d]^{\real} \\
\Dbc(X,E)^\op\ar[r]_-{D_X} & \Dbc(X,E)}\]

\end{itemize}
\label{prop_comp_real1}
\end{subprop}

We need one last compatibility. 
Suppose that we have a flat morphism of finite type $X\fl S$.
If $T\fl S$ is \'etale
(resp. the inclusion of the generic point of $S$),
$u:X_T\fl X$ is its base change to $X$ and $u^*:\Dbc(X,E)\fl\Dbc(X_T,E)$
is the restriction functor, then $u^*$ (resp. $u^*[-\dim S]$) is
t-exact by Proposition \ref{prop_restr_fg}. The following result
is proved exactly as Proposition \ref{prop_comp_real1} (see page~\pageref{perverse_test} in Section~\ref{f-horizontal}).

\begin{subprop}
In the situation above, we have
a commutative
diagram (up to
natural isomorphism)
\[\xymatrix{\Db(\Perv(X,E))\ar[r]^-{u^*[a]}\ar[d]_{\real} & 
\Db(\Perv(X_T,E))\ar[d]^{\real} \\
\Dbc(X,E)\ar[r]_-{u^*[a]} & \Dbc(X_T,E)}\]
where $a=0$ is $T\fl S$ is \'etale and $a=-\dim S$ if
$T\fl S$ is the inclusion of the generic point of $S$.

\label{prop_comp_real2}
\end{subprop}

\subsection{Horizontal constructible complexes}
\label{section_def_hor}

From now on, we fix
a field $k$ of finite type over its prime field (in other words,
$k$ is the field of fractions of an integral scheme of finite type over
$\Z$) and an algebraic extension $E$ of $\Q_\ell$.
We will consider separated schemes of finite type over $k$ and denote
by them by capital Roman letters such as $X$, $Y$, $U$ etc.

We will recall some constructions and results of sections 1-3 of
\cite{Hu}. In this article, Huber assumes that $k$ is a number field, but,
as she notes herself in the remark after proposition 2.3, this is not
really necessary and all her constructions extend to the more general
situation considered here, either by Deligne's generic constructibility
theorem (SGA 4 1/2 [Th. finitude]) or by Gabber's finiteness results
(\cite{Ga}).

Let $\Uf$ be the set (ordered by inclusion) of $\Z$-subalgebras $A\subset k$
that are regular and of finite type over $\Z$ and such that $k$ is the
field of fractions of $A$.
By a theorem of Nagata
(see EGA IV 6.12.6), if $\Bf$ is an integral scheme of finite type over
$\Z$,
then the regular
locus of $\Bf$ is open in $\Bf$. Hence $k=\varinjlim_{A\in\Uf}A$. So we are
in the situation of EGA IV 8 and can use the results of this reference.

If $A\in\Uf$, we say that a scheme over $\Spec A$ is \emph{horizontal}
if it is flat and of finite type over $A$. Let $X$ be a scheme over $k$.
We denote by $\Uf X$ the category of triples $(A,\X,u)$, where $A\in\Uf$,
$\X$ is a horizontal scheme over $A$ and $u$ is an isomorphism of $k$-schemes
$X\iso \X\otimes_A k$; we will often omit $u$ from the notation. A morphism
$(A,\X,u)\fl(A',\X',u')$ is an inclusion $A\subset A'$ and an open embedding
$f:\X'\fl\X\otimes_A A'$ such that $u'=u\circ f$. Then we have a canonical
isomorphism (given by the entry $u$ of the triples)
\[X\iso\varinjlim_{(A,\X)\in\Ob\Uf X}\X\otimes_A k.\]

If $f:(A,\X)\fl (A',\X')$ is a morphism
in $\Uf X$, then it induces an exact functor
\[\Dbc(\X,E)\fl \Dbc(\X\otimes_A A',E)\flnom{f^*}\Dbc(\X',
E),\]
where the first functor is the restriction functor along the open embedding
$\X\otimes_A A'\fl\X$.

\begin{subdef} (See \cite{Hu}, Definition 1.2.)
Let $X$ be a scheme over $k$. 
We define the category $\Dbh(X,E)$ by
\[\Dbh(X,E)=2-\varinjlim_{(A,\X)\in\Ob\Uf X}\Dbc(\X,E).\]

We call this category the \emph{category
of bounded constructible horizontal  $E$-complexes
of sheaves on $X$}.

\end{subdef}

Note that we could also define versions of these categories with coefficients
in $\Of_E$ (if $E$ is a finite extension of $\Q_\ell$), as Huber does. But
we will only be interested in this article in the category $\Dbh(X,E)$.

As in the remark following Definition 1.2 of \cite{Hu}, we see that
these categories are triangulated and have a tautological
t-structure (induced by the tautological
t-structure on the categories $\Dbc(\X,E)$), and that all the
properties of Theorem 6.3 of \cite{Ek} carry over.
The heart of the canonical t-structure will be denoted by $\Sh_h(X,E)$, and we will call
its objects 
\emph{horizontal constructible sheaves} on $X$.

We denote by $\eta^*:\Dbh(X,E)\fl \Dbc(X,E)$ the exact
functor induced
by the restriction functors $\Dbc(\X,E)\fl
\Dbc(\X\otimes_Ak,E)\flnom{u^*}
\Dbc(X,E)$, for $(A,\X,u)\in\Ob\Uf X$.

\begin{subprop}
\begin{itemize}
\item[(i)] The functor $\eta^*$ is fully faithful on the heart of the
tautological t-structure.
\item[(ii)] If 
$\F,\Gf\in\Ob(\Sh_h(X,E))$, then
\[\eta^*:\Ext^1_{\Dbh(X,E)}(\F,\Gf)\fl\Ext^1_{\Dbc(X,E)}(\eta^*\F,\eta^*\Gf)\]
is injective.

\end{itemize}
\label{prop_1.3}
\end{subprop}

\begin{proof} The first point is proposition 1.3 of \cite{Hu}. 
We prove the second point.
Let $(A,\X,u)\in\Ob(\Uf X)$ such
that $\F$ and $\Gf$ come from objects $K$ and $L$ of $\Dbc(\X,E)$. 
We use
$u$ to identify $X$ and $\X\otimes_A k$. 
The constructible sheaves $\bigoplus_{i\not=0}\H^i K$ and
$\bigoplus_{i\not=0}\H^i L$ on $\X$ are supported on a closed subset disjoint
from $X$, so, after shrinking $\X$, we may assume that
$K$ and $L$ are constructible sheaves on $\X$. By definition of
$\Dbh(X,E)$, we have
\[\Ext^1_{\Dbh(X,E)}(\F,\Gf)=\varinjlim_{A\subset A'\in\Uf}
\Ext^1_{\Dbc(\X\otimes_A A',E)}(K_{|\X\otimes_A A'},L_{|\X\otimes_A A'}).\]
Let $A'\supset A$ be an element of $\Uf$. The category $\Dbc(\X\otimes_A A',
E)$ is a full subcategory of $\D((\X\otimes_A A')_\proet,E)$, the
derived category of the category of $E$-modules on the pro\'etale
site of $\X\otimes_A A'$, so the groups $\Ext^1_{\Dbc(\X\otimes_A A',E)}$
parametrize extensions in this category of $E$-modules (see section
3.2 of chapter III of Verdier's book \cite{Ve}). But $\Sh_c(\X\otimes_A A',
E)$ is a Serre subcategory of the category of all sheaves of $E$-modules
(see proposition 6.8.11 of \cite{BS}), so
$\Ext^1_{\Dbc(\X\otimes_A A',E)}(K_{|\X\otimes_A A'},L_{|\X\otimes_A A'})$
is the group of equivalence classes of extensions of 
$K_{|\X\otimes_A A'}$
by $L_{|\X\otimes_A A'}$ in $\Sh_c(\X\otimes_A A',E)$. We have a similar
statement for $\Ext^1_{\Dbc(X,E)}(\eta^*\F,\eta^*\Gf)$. 

Now let $c\in\Ext^1_{\Dbh(X,E)}(\F,\Gf)$, and suppose that its
image in $\Ext^1_{\Dbc(X,E)}(\eta^*\F,\eta^*\Gf)$ is $0$. There exists
an element $A'\supset A$ of $\Uf$ and an extension
\[0\fl L_{|\X\otimes_A A'}\fl M\fl K_{|\X\otimes_A A'}\fl 0\]
in $\Sh_c(\X\otimes_A A',E)$ whose class is $c$. The hypothesis
on $c$ says that the restriction of this extension to $X$
is split. But, by point (i), this implies that there exists an element
$A''\supset A'$ of $\Uf$ such that the restriction of the extension
to $\X\otimes_A A''$ already splits, which means that $c=0$.
\end{proof}

\subsection{Horizontal perverse sheaves}
\label{section_hor}

In this section, we define the perverse t-structure on $\Dbh(X,E)$.
Note that, by Proposition 1.4 of Giral's article \cite{Gi},
all the rings $A$ in $\Uf$ have the same Krull dimension $c$, which is
the transcendence degree of $k$ over its prime field if $k$ is of positive characteristic, and
$1$ plus this transcendence degree if $k$ is of characteristic $0$.

If $(A,\X,u)\in\Ob\Uf X$, then we consider the perverse t-structure
on $\Dbc(\X,E)$ defined in section \ref{def_perverse}.
Then the functor $u^*[-c]:\Dbc(\X,E)\fl\Dbc(X,E)$ is t-exact 
by Proposition \ref{prop_restr_fg}.
Also, for every morphism
$f:(A,\X,u)\fl(A',\X',u')$ in $\Uf X$, the restriction functor
$f^*:\Dbc(\X',E)\fl\Dbc(\X,E)$ is t-exact by the same proposition.
By taking the limit of the shift by $c$ of the t-structures on the
$\Dbc(\X,E)$,
we get a t-structure on $\Dbh(X,E)$ such that $\eta^*:\Dbh(X,E)\fl
\Dbc(X,E)$ is t-exact. We denote the heart of this t-structure by
$\Perv_h(X,E)$ and call it the category of \emph{horizontal perverse sheaves} 
on $X$. We still denote the perverse cohomology functors by $\Hp^i:
\Dbh(X,E)\fl\Perv_h(X,E)$.
We also have a realization functor $\real:\Db\Perv_h(X,E)\fl\Dbh(X,E)$
extending the inclusion $\Perv_h(X,E)\subset\Dbh(X,E)$, 
by Theorem \ref{thm_real} and Section~\ref{f-horizontal}.

\begin{subrmk}
This is not Huber's construction.
Let us recall her construction 
and compare it with ours.
Let $A\in\Uf$ and let $\X$ be a horizontal scheme over $A$. As in \cite{Hu} 2.1,
we say that a stratification of $\X$ is \emph{horizontal} 
if all its strata are
smooth over $A$. 
Suppose that $E/\Q_\ell$ is finite.
If $S$ is a horizontal stratification of $\X$ and
$L$ is the data of a set of irreducible lisse $\Of_E$-sheaves on every stratum
of $S$ satisfying condition (c) of \cite{Hu} 2.2, we get as in Definition 2.2
and Lemma 2.4 of \cite{Hu} a full subcategory $D^b_{(S,L)}(\X,\Of_E)$ of
$(S,L)$-constructible objects in $\Dbc(\X,\Of_E)$, and it has a self-dual
perverse t-structure, whose heart we will denote by $\Perv_{(S,L)}(\X,\Of_E)$.
Because the strata $S$ are smooth over $\Spec A$, lisse sheaves on them are
perverse for our t-structure on $\Dbc(\X,\Of_E)$ when placed
in degree $-\dim(A)=-c$, so Huber's t-structure is the shift
by $c$ of the one
induced by our perverse t-structure on $\Dbc(\X,\Of_E)$.

By Proposition 2.3 of \cite{Hu}, 
the category $\Dbh(X,\Of_E)$ is the 2-colimit
of the categories $D^b_{(S,L)}(\X,\Of_E)$ over all $(A,\X)\in\Ob\Uf X$ and
couples $(S,L)$ as before, and by Theorem 2.5 of \cite{Hu}, the perverse
t-structure goes to the limit and induces a t-structure on $\Dbh(X,\Of_E)$.
This is the t-structure that is used in \cite{Hu}, and, by the observation
of the previous paragraph, it coincides with the one that we defined
at the beginning of this section.

\end{subrmk}

Huber's definition has the advantage that we can apply Deligne's generic
base change theorem (from SGA 4 1/2 [Th. finitude]) to deduce statements
for horizontal perverse sheaves from statements for perverse
sheaves on schemes over finite fields proved in sections 4 and 5 of
\cite{BBD}.

For example, we see as in \cite{Hu} 2.7 that
the six operations have the usual exactness properties
with respect to the perverse t-structure (which means the properties of
\cite{BBD} 4.1 and 4.2), that the category $\Perv_h(X,E)$ is
Artinian and Noetherian, and that we have the same description of its simple
objects as in Theorem 4.3.1 of \cite{BBD}.
\footnote{But we could also have proved all these statements from our
definition.}

The following result is a slight generalization of the first part of
Lemma 2.12 of \cite{Hu}. 

\begin{subprop} 
\begin{enumerate}
\item The functor $\eta^*:\Perv_h(X,E)\fl\Perv(X,E)$ is fully
faithful, and its essential image is the full category of perverse sheaves
on $X$ that extend to a constructible complex on some $\X$, for
$(A,\X)\in\Ob \Uf X$.
\item For every $K,L\in\Ob\Perv(X,E)$, the morphism
\[\Ext^1_{\Dbh(X,E)}(K,L)\fl\Ext^1_{\Dbc(X,E)}(\eta^* K,\eta^* L)\]
induced by $\eta$ is injective.

\end{enumerate}
\label{prop_eta_fully_faithful}
\end{subprop}

\begin{proof}
If the category where we take the $\Ext^i$ is
clear from context, we omit it in this proof. Also, we omit the
coefficients $E$ in the notation.

We prove both points by Noetherian induction on $X$. If $\dim X=0$,
then the perverse $t$-structure on $\Dbc(X,
E)$ is the usual t-structure, so both points follow from Proposition
\ref{prop_1.3} (which is an easy consequence of Proposition
1.3 of \cite{Hu}).

Suppose that $\dim X>0$, and let $K,L\in\Ob\Perv_h(X)$. Lemma 2.12
of \cite{Hu} says that the map $\Hom(K,L)\fl
\Hom(\eta^* K,\eta^* L)$ is injective, and we want to show
that it is also surjective. We show this by induction on the sum of the
lengths of $K$ and $L$.

Suppose first that $K$ and $L$ are both simple. Then we have smooth connected
locally closed subschemes $k_1:Y_1\fl X$ and $k_2:Y_2\fl X$ and
horizontal
locally constant sheaves $\Lf_1$ on $Y_1$ and $\Lf_2$ on $Y_2$
such that $K=k_{1!*}\Lf_1[\dim Y_1]$ and $L=k_{2!*}\Lf_2[\dim Y_2]$.
We have $\Hom(K,L)=0$ if $K\not\simeq L$, and
$\Hom(K,K)=\Hom_{\Dbh(X)}(\Lf_1,\Lf_1)$.
In particular, by
Proposition 1.3 of \cite{Hu}, $\eta^*K$ and $\eta^* L$ are also simple, and
$\Hom(K,L)\iso\Hom(\eta^* K,\eta^* L)$, proving the first point.

We prove the second point.
Let $Z=\overline{Y}_1\cap\overline{Y}_2$, and denote by $i:Z\fl X$ and
$j:X-Z\fl X$ the inclusions. We have an exact triangle
\[R\Hom(i^*K,i^!L)\fl R\Hom(K,L)\fl R\Hom(j^* K,j^*L)\flnom{+1}.\]
As $j^*K$ and $j^*L$ are perverse with disjoint supports on $X-Z$,
$R\Hom(j^*K,j^*L)=0$, so we get isomorphisms $\Ext^i(i^*K,i^!L)\iso
\Ext^i(K,L)$ for every $i\in\Z$. We have a similar result for $\eta^*K$
and $\eta^*L$. 

If $\overline{Y}_1$ is not contained is $\overline{Y}_2$, then $Z$ is a proper
closed subset of $\overline{Y}_1$, so $i^*K$ and $i^*\eta^*K$ are concentrated
in perverse degree $\leq -1$. As $i^!L$ and $i^!\eta^*L$ are concentrated in
perverse degree $\geq 0$, we get
\[\Ext^1(K,L)=\Hom(\Hp^{-1}i^*K,\Hp^0 i^!L)\]
and
\[\Ext^1(\eta^*K,\eta^*L)=\Hom(\Hp^{-1}i^*\eta^*K,\Hp^0 i^!\eta^*L),\]
so the second point follows from the induction hypothesis applied to
$Z$.

If $\overline{Y}_2$ is not contained is $\overline{Y}_1$, then $Z$ is a proper
closed subset of $\overline{Y}_2$, so $i^!L$ and $i^!\eta^*L$ are concentrated
in perverse degree $\geq 1$. As $i^*K$ and $i^*\eta^*K$ are concentrated in
perverse degree $\leq 0$, we get
\[\Ext^1(K,L)=\Hom(\Hp^{0}i^*K,\Hp^1 i^!L)\]
and
\[\Ext^1(\eta^*K,\eta^*L)=\Hom(\Hp^{0}i^*\eta^*K,\Hp^1 i^!\eta^*L),\]
so the second point again follows from the induction hypothesis applied to
$Z$.

Finally, suppose that $\overline{Y}_1=\overline{Y}_2$. Then $i^*K$ and
$i^!L$ are perverse and simple, and we may assume that $Y_1=Y_2$. Let $b$ be the
inclusion of the open subscheme $Y_1$ of $Z$, and $a$ be the inclusion
of its complement. As before, we have an exact triangle
\[R\Hom(a^*i^*K,a^!i^!L)\fl R\Hom(i^*K,i^!L)\fl R\Hom(b^*i^* K,b^*i^!L)
\flnom{+1}.\]
As $i^*K$ and $i^!L$ are simple of support $Z$, we know that
$a^*i^*K$ is concentrated in perverse degree $\leq -1$ and that
$a^!i^!L$ is concentrated in perverse degree $\geq 1$, so we get an
injective map
\[\Ext^1(i^*K,i^!L)\fl\Ext^1(b^*i^*K,b^*i^!L)=\Ext^1_{\Dbh(X)}(\Lf_1,\Lf_2).\]
We have a similar calculation for $i^*\eta^*K$ and $i^!\eta^*L$, and
so the second point follows from Proposition \ref{prop_1.3}.
This finishes the proof in the case where $K$ and $L$ are both simple.

Now suppose that we have an exact sequence $0\fl K_1\fl K\fl K_2\fl 0$,
and that we know the result for the pairs $(K_1,L)$ and $(K_2,L)$. We show it
for $(K,L)$. Write $K'=\eta^*K$, $L'=\eta^*L$ etc.
We have a commutative diagram with exact rows
\[\xymatrix@C=5pt{
0\ar[r] & \Hom(K_2,L)\ar[r]\ar[d]^\wr & \Hom(K,L)\ar[r]\ar[d] &\Hom(K_1,L)
\ar[r]\ar[d]^\wr & \Ext^1(K_2,L)\ar[r]\ar@{^{(}->}[d] &
\Ext^1(K,L)\ar[r]\ar[d] & \Ext^1(K_1,L)\ar@{^{(}->}[d]\\
0\ar[r] & \Hom(K'_2,L')\ar[r] & \Hom(K',L')\ar[r] &\Hom(K'_1,L')\ar[r] & 
\Ext^1(K'_2,L')\ar[r] &
\Ext^1(K',L')\ar[r] & \Ext^1(K'_1,L')}\]
so both points follow from the five lemma.

The case where we have an exact sequence $0\fl L_1\fl L\fl L_2\fl 0$ such
that the result is known for $(K,L_1)$ and $(K,L_2)$ is treated in the
same way.
\end{proof}

\subsection{Mixed perverse sheaves}
\label{def_mixed}

The key observation in order to define weights in our setting, if $A\in\Uf$, then, as $A$ is a $\Z$-algebra of
finite type, the residue fields of closed points of $\Spec A$ are finite, so
we can use the theory of \cite{BBD} chapter 5 as in section 3 of \cite{Hu}
to define categories $\Dbm(X,E)$ of mixed
horizontal complexes. Once we have defined what it means for a horizontal constructible
sheaf to be punctually pure of a certain weight, the definition
proceeds as in \cite{BBD} 5.1.5. If $\F\in\Ob\Sh_h(X,E)$ 
and $w\in\Z$,
we say that $\F$ is \emph{punctually pure of weight $w$} if there exists
$(A,\X)\in\Ob\Uf X$ and $\F'\in\Ob\Sh_c(\X,E)$ a constructible sheaf extending $\F$
such that, for every closed point $x$ of $\Spec A$, $\F'_{|\X_x}$ is punctually
pure of weight $w$ in the sense of \cite{BBD} 5.1.5 (that is, of
Deligne's \cite{De}). We say that $\F$ is \emph{mixed} if it has a filtration
whose graded pieces are punctually pure of some weight.

We denote by $\Dbm(X,E)$ the full subcategory of \emph{mixed complexes}
in $\Dbh(X,E)$; the objects of $\Dbm(X,E)$ are complexes
$K$ such that all the (usual) cohomology sheaves of $K$ are
mixed.

By Proposition 3.2 of \cite{Hu}, these subcategories
are stable by the 6 operations and inherit a perverse t-structure from
$\Dbh(X,E)$. We denote the heart of this t-structure by
$\Perv_m(X,E)$; it is a full subcategory of $\Perv_h(X,E)$,
stable by subquotients and extensions. All the compatibilities between
the six operations (and the intermediate extension functor) and weights
that are proved in \cite{De} and \cite{BBD} chapter 5 remain true, see
\cite{Hu} 3.3-3.6. Also, the realization functor $\real:D^b\Perv_h(X,E)\fl
\Dbh(X,E)$ restricts to a functor $\real:D^b\Perv_m(X,E)\fl
\Dbh(X,E)$, whose essential image is contained in $\Dbm(X,E)$ by
definition of $\Dbm(X,E)$.

Let us introduce weight filtrations, following Definition 3.7 of \cite{Hu}.

\begin{subdef} Let $K\in\Ob\Perv_m(X,E)$. A \emph{weight filtration} on
$K$ is a separated exhaustive ascending filtration $W$ on $K$ (in the abelian category
$\Perv_m(X,E)$) such that $\Gra_k^W K$ is pure of weight $k$ for every
$k\in\Z$.

\end{subdef}

As the abelian category $\Perv_m(X,E)$ is Artinian and Noetherian, such
a filtration is automatically finite. Note also that morphisms in
$\Perv_m(X,E)$ are strictly compatible with weight filtrations (Lemma
of 3.8 \cite{Hu}), so in particular a weight filtration is unique if it exists.

\begin{subdef} Let $\Perv_{mf}(X,E)$ be the full subcategory of
$\Perv_m(X,E)$ whose objects are mixed horizontal perverse sheaves
admitting a weight filtration.

\end{subdef}

This subcategory is clearly stable by subquotients in $\Perv_m(X,E)$, but
it is not stable by extensions (even if $X=\Spec k$),
see the warning before Proposition 3.4 of \cite{Hu}.

Finally, the following conservativity result will be very useful.

\begin{subprop}
\begin{itemize}
\item[(i)] The functor $\eta^*:\Dbh(X,E)\fl\Dbc(X,E)$ is conservative.
\item[(ii)] The realization functor $\Db\Perv_h(X,E)\fl\Dbh(X,E)$ is conservative.
\item[(iii)] The obvious functor $\Db\Perv_{mf}(X,E)\fl\Db\Perv_h(X,E)$ is conservative.

\end{itemize}
\label{prop_real_cons}
\end{subprop}

\begin{proof} 
In all three cases, we have t-structures on the source and target for which the functors
are t-exact and such that the family of cohomology functors for the t-structure is conservative
(the perverse t-structure on $\Dbh(X,E)$ and $\Dbc(X,E)$, and the canonical t-structure on the
derived categories). So it suffices to check that the functors on the hearts are conservative. 
But these functors are all faithful and exact (in fact, they are all fully faithful), so they are
conservative.
\end{proof}

\section{Main theorems}

From now on, we will fix the algebraic extension $E$ of $\Q_\ell$ and
omit it in the notation.

\subsection{Informal statement}

Informally, the main theorems say that the sheaves operations
($f_*$, $f^*$, $f_!$ and $f^!$, $\Homf$, $\Ltimes$, Poincar\'e-Verdier
duality, unipotent nearby and
vanishing cycles) lift to the categories $\Db\Perv_{mf}(X)$ in
a way that is compatible with the realization functors $\Db\Perv_{mf}(X)\fl
\Dbm(X)$, and that all the relations between these functors that are
true in the categories $\Dbm(X)$ are still true in the categories
$\Db\Perv_{mf}(X)$.

A convenient way to say this is to use the formalism introduced
in Ayoub's thesis \cite{Ay1} (and in his article \cite{AyReal1}).
Then Theorem \ref{thm_main1} says that the four operations 
$f_*$, $f^*$, $f_!$ and $f^!$ exist and satisfy all the expected adjunctions
and compatibilities, and Theorem \ref{thm_tensor_product} asserts the existence
and properties of the derived internal $\Hom$s and derived tensor products.
The stability of the categories $\Perv_{mf}$ under the perverse direct image functors
is proved in section
\ref{stab_Pmf}, and the
unipotent vanishing cycles are constructed in section
\ref{section_Beilinson_cycles} (see Corollary \ref{cor_nearby_etc}).

\subsection{Formal statement}

We denote by $\Sch/k$ the category of schemes over $k$
(always assumed to be
separated of finite type, as before)
and by $\TR$ the $2$-category of triangulated
categories.

The notion of a formalism of the four operations ($f^*$, $f_*$, $f_!$, $f^!$)
has been axiomatized by Deligne, Voevodsky and Ayoub, under the name
of ``foncteur crois\'e''.
\footnote{There are other approaches, but this particular one seems better
suited to our situation. For example, using derivators is complicated by the
fact that it is difficult to make sense of the notion of ``perverse
sheaf over a diagram of schemes'', because inverse image functors typically
do not preserve perverse sheaves.}
We will follow Ayoub's presentation.

\begin{subdef} (See Definition 1.2.12 of \cite{Ay1}.)
\footnote{Note that
we take the two categories $\Cf_1$ and $\Cf_2$ of this reference to be equal
to $\Sch/k$.}
A {\em crossed functor} (``foncteur crois\'e'') on $\Sch/k$ with values
in $\TR$ (relatively to the class of cartesian squares)
is a quadruple of $2$-functors $H=(H^*,H_*,H_!,H^!):\Sch/k\fl\TR$, such
that :
\begin{itemize}
\item[(0)] for every $X\in\Ob(\Sch/k)$, we have $H_*(X)=H_!(X)=H^*(X)=H^!(X)$
(we denote this triangulated category by $H(X)$);
\item[(1)] the functors $H_*,H_!$ are covariant, and the functors $H^*,H^!$
are contravariant;
\item[(2)] the functor $H^*$ is a global left adjoint of $H_*$;
\item[(3)] the functor $H^!$ is a global right adjoint of $H_!$;
\end{itemize}
together with the data of exchange structures of type $\swarrow$ on the couples
$(H_*,H_!)$ and $(H^*,H^!)$ (see Definition 1.2.1 of \cite{Ay1}), i.e., for
every cartesian square
\[\xymatrix{X'\ar[r]^-{g'}\ar[d]_{f'} & X\ar[d]^f \\
X\ar[r]_-{g} & Y}\]
in $\Sch/k$, we have morphisms of functors $H_!(f)\circ H_*(g')\fl
H_*(g)\circ H_!(f')$ and $H^*(g')\circ H^!(f)\fl H^!(f')\circ H^*(g)$ compatible
with horizontal and vertical composition of squares.

This data is moreover required to satisfy the following condition :
For every cartesian square
\[\xymatrix{X'\ar[r]^-{g'}\ar[d]_{f'} & X\ar[d]^f \\
X\ar[r]_-{g} & Y}\]
in $\Sch/k$, the morphisms
$H^*(g)\circ H_!(f)\fl H_!(f')\circ H^*(g')$ and $H_!(f')\circ H^*(g')\fl
H^*(g)\circ H_!(f)$ formally constructed used the exchange structures and
adjunctions (see the beginning of \cite{Ay1} 1.2.4) are isomorphisms
and inverses of each other; equivalently, we could required that the
morphisms
$H^!(f)\circ H_*(g')\fl H_*(g)\circ H^!(f')$ and $H_*(g)\circ H^!(f')\fl
H^!(f)\circ H_*(g')$ are isomorphisms
and inverses of each other.

\end{subdef}

\begin{subdef} (See Definition 3.1 and Theorem 3.4 of \cite{AyReal1}.)
Suppose that we have two crossed functors $H_1,H_2:\Sch/k\fl\TR$.
A {\em morphism of crossed functors} $R:H_1\fl H_2$ is the following data :
\begin{itemize}
\item[(1)] For every $X\in\Ob(\Sch/k)$, a triangulated functor
$R_X:H_1(X)\fl H_2(X)$.
\item[(2)] For every $f:X\fl Y$ in $\Sch/k$, invertible natural
transformations
\[\theta_f:H_2^*(f)\circ R_Y\iso R_X\circ H_1^*(f)\]
\[\gamma_f:R_Y\circ H_{1,*}(f)\iso H_{2,*}(f)\circ R_X\]
\[\rho_f:H_{2,!}(f)\circ R_X\iso R_Y\circ H_{2,!}(f)\]
\[\xi_f:R_X\circ H_2^!(f)\iso H_2^!(f)\circ R_Y.\]

\end{itemize}

We require these transformations to satisfy the compatibility conditions
spelled out in section 3 of Ayoub's paper \cite{AyReal1}.

\label{def_hom_cf}
\end{subdef}

\begin{subex}
\label{example_crossed_functor}
For $a\in\{c,h,m\}$,
we have a crossed functor $H_a=(H_a^*,H_{a,*},H_{a,!},H_a^!):\Sch/k\fl\TR$
defined in the following way :
\begin{itemize}
\item For every $X\in\Ob(\Sch/k)$, 
\[H_a^*(X)=H_{a,*}(X)=H_{a,!}=(X)=H_a^!(X)=\Dba(X).\]
\item For every $f:X\fl Y$ in $\Sch/k$, we have
$H_a^*(f)=f^*$, $H_{a,*}(f)=f_*$, $H_{a,!}(f)=f_!$ and $H_a^!(f)=f^!$.

\end{itemize}

Moreover, we have morphisms of crossed functors $H_m\fl H_h\fl H_c$.

\end{subex}

Then our first main result is the following theorem.

\begin{subthm} There exists a crossed functor
$H_{mf}=(H_{mf}^*,H_{mf,*},H_{mf,!},H_{mf}^!):\Sch/k\fl\TR$ and a morphism
of crossed functors $R:H_{mf}\fl H_m$ such that, for every $X\in\Ob(\Sch/k)$,
$H_{mf}(X)=\Db\Perv_{mf}(X)$ and $R_X:H_{mf}(X)=\Db\Perv_{mf}(X)\fl H_m(X)
=\Dbm(X)$ is the composition of the obvious functor $\Db\Perv_{mf}(X)\fl\Db
\Perv_m(X)$ and of the realization functor of section \ref{section_hor}.

Moreover, the functor $R_X$ is conservative for every $k$-scheme $X$, and
we have for every morphism $f$ in $\Sch/k$
a natural transformation $H_{mf,!}(f)\fl H_{mf,*}(f)$, which is an isomorphism
if $f$ is proper.

\label{thm_main1}
\end{subthm}

To prove this, we will follow the same strategy as in chapter 1 of
\cite{Ay1} and section 3 of \cite{AyReal1}, and deduce the existence
of the crossed functor and of the natural transformation $f_!\fl f_*$
from that of a stable homotopic $2$-functor (see Definition \ref{def_sh2f}).

We note that the conservativity of $R_X$ follows immediately
from Proposition \ref{prop_real_cons}, and then
the fact that $f_!\fl f_*$ is
an isomorphism for $f$ proper follows from the conservativity
of the functors $R_X$.

\begin{subdef} (See \cite{Ay1} 1.4.1.)
Let $H^*:\Sch/k\fl\TR$ be a contravariant 2-functor. For $X\in\Ob(\Sch/k)$,
we write $H^*(X)=H(X)$, and for $f$ a morphism of $\Sch/k$, we also
denote the $1$-functor $H^*(f)$ by $f^*$. We assume that $H^*$ is
strictly unital, i.e., for every morphism $f:X\fl Y$ in $\Sch/k$,
the connection isomorphisms $(f\circ\id_X)^*\simeq f^*$ and $(\id_Y\circ f)^*
\simeq f^*$ are the identity.

We say that $H^*$ is a {\em stable homotopic $2$-functor} if it satisfies the
following conditions :
\begin{itemize}
\item[(1)] $H(\varnothing)=0$.
\item[(2)] For every $f:X\fl Y$ in $\Sch/k$, the functor $f^*:H(Y)\fl H(X)$
admits a right adjoint $f_*$. Moreover, if $f$ is a locally closed immersion,
then the counit $f^*f_*\fl\id_{H(X)}$ is an isomorphism.
\item[(3)] If $f:X\fl Y$ is a smooth morphism in $\Sch/k$, then the
functor $f^*$ admits a left adjoint $f_\sharp$. Moreover, if we have
a cartesian square :
\[\xymatrix{X'\ar[r]^-{g'}\ar[d]_{f'} & X\ar[d]^f \\
X\ar[r]_-{g} & Y}\]
with $f$ smooth, then the exchange morphism $f'_\sharp {g'}^*\fl g^*f_\sharp$
(defined formally using the adjonctions, see \cite{Ay1} 1.4.5) is an
isomorphism.
\item[(4)] If $j:U\fl X$ and $i:Z\fl X$ are complementary open and
closed immersions in $\Sch/k$, then the pair $(j^*,i^*)$ is conservative.
\item[(5)] If $X\in\Ob(\Sch/R)$ and $p:\Aff^1_X\fl X$ is the canonical
projection, then the unit $\id_X\fl p_*p^*$ is an isomorphism.
\item[(6)] With the notation of (5), if $s:X\fl\Aff^1_X$ is the zero section,
then $p_\sharp s_*:H(X)\fl H(X)$ is an equivalence of categories.

\end{itemize}
\label{def_sh2f}
\end{subdef}

\begin{subdef} (See Definition 3.1 of \cite{AyReal1}.)
Let $H_1^*,H_2^*:\Sch/k\fl\TR$ be two stable homotopic
$2$-functors. A \emph{morphism of stable homotopic $2$-functors}
$R:H_1^*\fl H_2^*$ is the data of :
\begin{itemize}
\item[(1)] For every $X\in\Ob(\Sch/k)$, a triangulated functor
$R_X:H_1(X)\fl H_2(X)$.
\item[(2)] For every $f:X\fl Y$ in $\Sch/k$, an invertible natural
transformation
\[\theta_f:f^*\circ R_Y\iso R_X\circ f^*.\]
\end{itemize}
We require that this data satisfy the following compatibility conditions :
\begin{itemize}
\item[(A)] The natural transformations are compatible with the composition
of morphisms in $\Sch/k$.
\item[(B)] If $f$ is smooth, then the natural transformation $f_\sharp\circ
R_X\fl R_Y\circ f_\sharp$ (obtained using the adjonction and $\theta_f^{-1}$)
is invertible.

\end{itemize}
\end{subdef}

\begin{subex} The crossed functors of Example \ref{example_crossed_functor}
define (by forgetting part of the data) three stable homotopic $2$-functors
$H_m^*$, $H_h^*$ and $H_c^*$, and morphisms $H_m^*\fl H_h^*\fl H_c^*$.

\end{subex}

Theorem \ref{thm_main1} now follows immediately from the following two results
(the first one is a consequence of several theorems of Ayoub and is also
used to construct the four operations on the triangulated categories
of Voevodsky motives, and the second one is the main technical result of
this paper).

\begin{subthm} 
\begin{itemize}
\item[(i)] (See Scholie 1.4.2 of \cite{Ay1}.)
Let $H^*:\Sch/k\fl\TR$ be a stable homotopic $2$-functor. Then $H^*$ extends
to a crossed functor $\Sch/k\fl\TR$.

\item[(ii)] (See Theorems 3.4 and 3.7 of \cite{AyReal1}.)
Let $H_1^*,H_2^*:\Sch/k\fl\TR$ be two stable homotopic $2$-functors
and $R:H_1^*\fl H_2^*$ be a morphism. Let $H_1,H_2:\Sch/k\fl\TR$ be crossed
functors extending $H_1^*,H_2^*$ as in (i). Then $R$ extends to a morphism
of crossed functors from $H_1$ to $H_2$.

\end{itemize}

\label{thm_crossed_functors}
\end{subthm}

\begin{subthm} There exists a stable homotopic $2$-functor
$H_{mf}^*:\Sch/k\fl\TR$ and a morphism of stable homotopic $2$-functors
$R:H_{mf}^*\fl H_m^*$ such that, for every $X\in\Ob(\Sch/k)$, $H_{mf}(X)=
\Db\Perv_{mf}(X)$ and $R_X:H_{mf}(X)\fl H_m(X)$ is the same functor as in
Theorem \ref{thm_main1}.

\label{thm_H^*}
\end{subthm}

\begin{proof} The construction of the $2$-functor $H_{mf}^*$ is given
in Corollary \ref{cor_inverse_images}. Let's check that it is a stable
homotopic $2$-functor. Property (1) is obvious. The fact that $H_{mf}^*(f)$
admits a right adjoint for every $f$ follows from the definition
of $H_{mf}^*$ as a global left adjoint, and the last part of property (2)
follows from Corollary \ref{cor_i^*} and Proposition \ref{prop_adjunction_open}.
The fact that $f^*$ admits a left adjoint for $f$ smooth is proved in
Proposition \ref{prop_duality}(iv), and the last part of property (3) as
well as properties (4) and (5)
follow from the conservativity of the realization functor. 
Finally, let $Y$ be a $k$-scheme, let $p:\Aff^1_Y\fl Y$ be the canonical
projection and $s:Y\fl\Aff^1_Y$ be the zero section. By Proposition
\ref{prop_duality}(v), we have a natural isomorphism $s_!\iso s_*$.
So we get a natural isomorphism
\[p_\sharp s_*=p_![2](1)s_*\iso p_!s_![2](1)\simeq(ps)_![2](1)=\id[2](1),\]
which shows that $p_\sharp s_*:\Db\Perv_{mf}(Y)\fl\Db\Perv_{mf}(Y)$ is an
equivalence of categories.
\end{proof}

Finally, we show the existence of tensor products and internal $\Hom$s
on the categories $\Db\Perv_{mf}(X)$.

\begin{subdef}
\begin{itemize}
\item[(i)] (See Definition 2.3.1 of \cite{Ay1}.)
A \emph{unitary symmetric mono\"idal} stable
homotopic $2$-functor is a stable homotopic
$2$-functor $H^*$ that takes its values in the $2$-category of symmetric
mono\"idal unitary triangulated categories, that is, that associates to
every $X\in\Ob\Sch/k$ a unitary symmetric mono\"idal
category $(H(X),\otimes_X,\ungras_X)$ and such that :
\begin{itemize}
\item[(a)] For every morphism $f:X\fl Y$ in $\Sch/k$, the functor
$f^*$ is unitary mono\"idal.
\item[(b)] (Projection formula.)
If $f:X\fl Y$ is smooth, $K\in\Ob H(Y)$ and $L\in\Ob H(X)$,
then the functorial map
\[p:f_\sharp(f^*(K)\otimes_Y L)\fl K\otimes_X f_\sharp(L)\]
constructed in Proposition 2.1.97 of \cite{Ay1} is an isomorphism.

\end{itemize}

\item[(ii)] (See Definition 3.2 of \cite{AyReal1}.)
Let $H_1^*$ and $H_2^*$ be two 
symmetric mono\"idal unitary stable homotopic $2$-functors. 
Then a morphism of 
symmetric mono\"idal unitary stable homotopic $2$-functors
from $H_1^*$ to $H_2^*$ is a morphism of stable homotopic $2$-functors
$R:H_1^*\fl H_2^*$ such that :
\begin{itemize}
\item[(a)] For every $X\in\Ob(\Sch/k)$, the functor $R_X$ is mono\"idal
unitary.
\item[(b)] For every morphism of $k$-schemes $f$, the natural
transformation $\theta_f$ is a morphism of mono\"idal unitary
functors.

\end{itemize}

\item[(iii)] (See Definition 2.3.50 of \cite{Ay1}.) If $H^*$ is as in
(i), we say that $H^*$ is \emph{closed} if, for every $X\in\Ob\Sch/k$,
the symmetric mono\"idal category $(H(X),\otimes_X)$ is closed; this
means that, for every object $K$ of $H(X)$, the endofunctor $K\otimes_X\cdot$
of $H(X)$ admits a right adjoint, that will be denoted by $\Homf_X(K,\cdot)$.

\end{itemize}
\label{def_monoidal}
\end{subdef}

\begin{subex} The stable homotopic $2$-functors $H_m^*$, $H_h^*$ and $H_c^*$
are all closed symmetric mono\"idal unitary (for the derived tensor product),
and the morphisms $H_m^*\fl H_h^*\fl H_c^*$ are 
morphisms of symmetric mono\"idal unitary stable homotopic $2$-functors.

\end{subex}

Our last result is the following :

\begin{subthm} There exists a structure of closed symmetric
mono\"idal unitary stable homotopic $2$-functor on $H_{mf}^*$ such that
$R:H_{mf}^*\fl H_m^*$ is a
morphism of symmetric mono\"idal unitary stable homotopic $2$-functors.

Moreover, for every $k$-scheme $X$, the functorial map
\[R_X\Homf_{\Db\Perv_{mf}(X)}(\cdot,\cdot)\fl \Homf_{\Dbm(X)}(R_X(\cdot),R_X(\cdot))\]
of \cite{AyReal1} (3.1) is an isomorphism.

\label{thm_tensor_product}
\end{subthm}

\begin{proof} This theorem is proved in section \ref{tensor}.
More precisely, the bifunctors $\otimes_X$ and $\Homf$ are constructed in
section \ref{tensor}, and all their properties are proved there except
for condition (i)(b) of Definition \ref{def_monoidal}. But this last
condition follows from the fact that the functor $R_X$ is conservative
(and that the analogous result is true in $\Dbm(X)$).
\end{proof}

\section{Easy stabilities}
\label{obvious}

The proof of Theorem \ref{thm_H^*} will require us to show that
the full subcategories $\Perv_{mf}(X)\subset\Perv_m(X)$ are preserved by
a certain number of sheaf operations. Here we list the easier such results.

\begin{prop} Let $f:X\fl Y$ be a morphism of $k$-schemes.
\begin{itemize}
\item[(i)] If $f$ is smooth of relative dimension $d$, then the exact
functor $f^*[d]:\Perv_m(Y)\fl\Perv_m(X)$ sends $\Perv_{mf}(Y)$ to $\Perv_{mf}(X)$.

\item[(ii)] If $f$ is proper, then, for every $k\in\Z$, the
functor $\Hp^k f_*:\Perv_m(X)\fl\Perv_m(Y)$ 
sends $\Perv_{mf}(X)$ to $\Perv_{mf}(Y)$.

\end{itemize}

\label{prop_obvious1}
\end{prop}

\begin{proof} Point (i) follows from the fact that
the functor $f^*[d]$ is exact (see Proposition \ref{prop_exactness})
and sends pure perverse sheaves to pure perverse sheaves (by \cite{BBD} 5.1.14).
Point (ii) is Proposition 3.9 of Huber's paper \cite{Hu}. (This proposition
is stated for $f$ smooth, but its proof doesn't use the smoothness of $f$.)
\end{proof}

\begin{prop} Let $X,Y\in\Ob(\Sch/k)$.
\begin{itemize}
\item[(i)] 
The Poincar\'e-Verdier duality 
functor $D_X:\Perv_m(X)^\op
\fl\Perv_m(X)$
sends $\Perv_{mf}(X)^\op$ to $\Perv_{mf}(X)$.
\item[(ii)] The external tensor product functor $\boxtimes:\Perv_m(X)\times
\Perv_m(Y)\fl\Perv_m(X\times Y)$
sends $\Perv_{mf}(X)\times\Perv_{mf}(Y)$ to $\Perv_{mf}(X\times Y)$.
\item[(iii)] The Tate twist 
functor $(1):\Perv_m(X)
\fl\Perv_m(X)$, $K\fle K(1)$
sends $\Perv_{mf}(X)$ to $\Perv_{mf}(X)$.

\end{itemize}

\label{prop_obvious2}
\end{prop}

\begin{proof} 
This follows from the fact all these functors are exact (see Proposition \ref{prop_exactness})
and send pure perverse sheaves to pure perverse sheaves (see \cite{BBD} 5.1.14).
\end{proof}

In particular, by deriving trivially the functors above, we get :
\begin{itemize}
\item[(i)] For every $X\in\Ob(\Sch/k)$, an exact functor
$D_X:\Db\Perv_{mf}^\op(X)\fl \Db\Perv_{mf}(X)$ and an isomorphism $D_X^2\simeq\id$,
and also an exact functor $\Db\Perv_{mf}(X)\fl \Db\Perv_{mf}(X)$, $K\fle K(1)$.
\item[(ii)] For every $X,Y\in\Ob(\Sch/k)$, an exact functor
$\boxtimes:\Db\Perv_{mf}(X)\times \Db\Perv_{mf}(Y)\fl \Db\Perv_{mf}(X\times Y)$,
satisfying the same properties of commutativity and associativity as the
external tensor product on the categories $\Dbc$.

\end{itemize}

Moreover, these functors
correspond to the usual ones on $\Dbm(X)$ by the realization functor
(by Proposition \ref{prop_comp_real1}).

Note that, by Proposition 3.2.2 and Theorem 3.2.4 of \cite{BBD}, the
$2$-functor $X\fle\Perv(X)$ is a stack for the {\'e}tale topology on $X$.
 We have the
following easy result :

\begin{prop} The categories $\Perv_h(U)$ (resp. $\Perv_m(U)$, resp. $\Perv_{mf}(U)$)
define a substack of $X\fle\Perv(X)$.

\label{prop_obvious3}
\end{prop}

\begin{proof} As $\Perv_h(U)$ (resp. $\Perv_m(U)$, resp. $\Perv_{mf}(U)$) is a full subcategory of $\Perv(U)$ for
every $U$, we only need to show the following fact : If $K$ is an
object of $\Perv(X)$ and if there exists an {\'e}tale cover
$(u_i:U_i\fl X)_{i\in I}$ of $X$ such that $u_i^* K$ is in 
$\Perv_h(U_i)$ (resp. $\Perv_m(U_i)$, resp.
$\Perv_{mf}(U_i)$)
for every $i\in I$, then $K$ is in $\Perv_h(X)$ (resp.
$\Perv_m(X)$, resp. $\Perv_{mf}(X)$).

We first treat the case of $\Perv_h$ and $\Perv_m$.
We may assume that $I$ is finite and that the $U_i$ are
affine. For all $i,j\in I$, we denote the fiber product of $u_i$ and $u_j$ by
$u_{ij}:U_i\times_X U_j\fl X$. Then, as $\Perv$ is a stack, we have an exact sequence in
$\Perv(X)$ :
\[0\fl K\fl\bigoplus_{i\in I}u_{i*}u_i^* K\fl\bigoplus_{i,j\in I}u_{ij*}u_{ij}^* K.\]
As the last two terms are in $\Perv_h(X)$ (resp. $\Perv_m(X)$) by assumption, and as
$\Perv_h(X)$ (resp. $\Perv_m(X)$) is a full abelian
subcategory of $\Perv(X)$ by Proposition
\ref{prop_eta_fully_faithful}, the perverse sheaf $K$ is also an object of $\Perv_h(X)$ (resp.
$\Perv_m(X)$).

We now treat the case of
$\Perv_{mf}(X)$. Let $a\in\Z$. We need to construct a subobject $L$
of $K$ such that $L$ is of weight $\leq a$ and $K/L$ is of weight $>a$.
For every $i\in I$, we set $L_i=W_a(u_i^* K)$, where $W$ is the weight
filtration on $K_i$. By the uniqueness of the weight filtration, the $L_i$
glue to a subobject $L$ of $K$. As we can test weights on an {\'e}tale
cover of $X$ (for example by Theorem 5.2.1 and 5.1.14(iii)
of \cite{BBD}), this $L$ satisfies the required conditions. 
\end{proof}

\begin{lemma} Let $i:Y\fl X$ be a closed immersion, and let $K\in\Ob\Perv_m(Y)$.
Then $K$ is in $\Perv_{mf}(Y)$ if and only if $i_* K$ is in $\Perv_{mf}(X)$.

\label{lemma_test_weight_on_i}
\end{lemma}

\begin{proof} If $K$ is in $\Perv_{mf}(Y)$, then $i_* K$ is in $\Perv_{mf}(X)$ by
Proposition \ref{prop_obvious1}(ii).

Conversely, assume that $i_*K$ is in $\Perv_{mf}(X)$. Let $a\in\Z$. We want to show that
there exists a subobject $K'$ of $K$ (in $\Perv_m(Y)$) such that $K'$ is of weight $\leq a$ and
$K/K'$ is of weight $\geq a+1$. By the assumption, there exists a subobject $L'\subset i_*K$
(in $\Perv_m(X)$) such that $L'$ is of weight $\leq a$ and $L'':=(i_*K)/L'$ is of weight $\geq a+1$.
Let $j$ be the inclusion of the complement of $Y$ in $X$. Then the functor $j^*$ is t-exact, so,
applying $j^*$ to the exact sequence $0\fl L'\fl i_* K\fl L''\fl 0$, we get an exact sequence
$0\fl j^*L'\fl 0\fl j^*L''\fl 0$ of mixed perverse sheaves on $X-Y$. This implies that
$j^*L'=j^*L''=0$, so the adjunction morphisms $i_*i^!L'\fl L'\fl i_*i^*L'$ and
$i_*i^!L''\fl L''\fl i_*i^*L''$ are isomorphisms. In particular, the mixed complexes
$i^*L'=i^!L'$ and $i^*L''=i^!L''$ are perverse. Let $K'=i^*L'$. We have just seen that $K'$ is perverse,
and the weights of $K'$ are $\leq a$ (see the remark after Definition 3.3 of \cite{Hu}). Also,
we have an exact triangle $K'=i^*L'\fl K\fl i^*L''=i^!L''\flnom{+1}$, which is actually an
exact sequence in $\Perv_m(Y)$, so the canonical map $K'\fl K$ is injective, and
$K/K'\simeq i^!L''$, which is of weight $\geq a+1$
(by the same remark in \cite{Hu}).
\end{proof}

\section{Beilinson's construction of unipotent nearby cycles}
\label{section_5}

In this section, we review Beilinson's construction of the unipotent nearby
and vanishing cycles functors from \cite{Be2}. There are two reasons
to do this:
\begin{itemize}
\item[(1)] We will want to define nearby cycles for horizontal
perverse sheaves, and to apply known theorems (about weights for example).
The easiest way to do this is to use Deligne's generic base change theorem,
but this might cause technical problems if we use the original
construction of nearby cycles (from SGA 7 I and XIII), which involves
direct images by morphisms that are not of finite type.
\item[(2)] We will need some of Beilinson's auxiliary functors anyway
to construct a left adjoint of $i_*$ for $i$ a closed immersion.

\end{itemize}

All the proofs of the results in this section can be found in \cite{Be2}
(see also \cite{gl}).

\subsection{Unipotent nearby cycles}

Fix a base field $k$, let $X$ be a $k$-scheme, and let
$f:X\fl\Aff^1_k$ be a morphism. 
We write
$\GD_m=\Aff^1-\{0\}$, 
$U=X\times_{\Aff^1}\GD_m\flnom{j}X$ and
$Y=X\times_{\Aff^1}\{0\}\flnom{i}X$.

We have an exact sequence
\[1\fl\pi_1^\geom(\GD_m,1)\fl\pi_1(\GD_m,1)\fl\Gal(\overline{k}/k)\fl 1,\]
which is split by the morphism coming from the unit section of
$\GD_m$. If $k$ is of characteristic $0$, then
$\pi_1^\geom(\GD_m,1)\simeq\widehat{\Z}(1)$; if $k$ is of characteristic
$p>0$, then $\pi_1^{\geom,(p)}(\GD_m,1)\simeq\widehat{\Z}^{(p)}(1)$. In both
cases, we get a projection $t_\ell:\pi_1^\geom(\GD_m,1)\fl\Z_\ell(1)$. We also
denote by $\chi:\Gal(\overline{k}/k)\fl\widehat{\Z}_\ell$
the $\ell$-adic cyclotomic character.

Let $\Psi_f:\Dbc(U)\fl\Dbc(Y_{\overline{k}})$ and 
$\Phi_f:\Dbc(X)\fl\Dbc(Y_{\overline{k}})$ be the nearby and vanishing cycles
functors defined in SGA 7 Expos{\'e} XVIII, shifted by $-1$ so that they
will be t-exact for the perverse t-structure. (See
Corollary 4.5 of Illusie's \cite{Il}, and
note that the dimension function we use on $U$ is shifted by
$+1$ when compared with Illusie's dimension function.)
We denote by $T$ a topological generator of $\pi_1^\geom(\GD_m,1)$ or
$\pi_1^{\geom,(p)}(\GD_m)$ (depending on the characteristic of $k$).
We have a functorial exact triangle
$\Psi_f\stackrel{T-1}{\fl}
\Psi_f\fl i^*j_*\stackrel{+1}{\fl}$.

\begin{subprop} There exists a functorial $T$-equivariant direct sum
decomposition $\Psi_f=\Psi_f^u\oplus\Psi_f^{nu}$ such that, for every
$K\in\Dbc(U)$, $T-1$ acts nilpotently on $\Psi_f^u(K)$ and invertibly
on $\Psi_f^{nu}(K)$.

\label{prop_Psi_u}
\end{subprop}

In particular, the functorial exact triangle $\Psi_f\stackrel{T-1}{\fl}
\Psi_f\fl i^*j_*\stackrel{+1}{\fl}$ induces a functorial exact triangle
$\Psi_f^u\stackrel{T-1}{\fl}
\Psi_f^u\fl i^*j_*\stackrel{+1}{\fl}$.

The functor $\Psi_f^u$ is called the \emph{unipotent nearby cycles functor}.

\begin{proof} It suffices to prove that, for every $K\in \Dbc(U)$, there
exists a nonzero polynomial $P$ (with coefficients in the coefficient field
$E$ that we are using for the categories $\Dbc$) such that $P(T)$ acts by $0$
on $\Psi_f(K)$. (The rest is standard linear algebra.)
As we know that $\Psi_f$ sends $\Dbc(X)$ to $\Dbc(Y_{\overline{k}})$ (i.e.
preserves constructibility), this follows from the fact that,
for every $L\in \Dbc(Y_{\overline{k}})$, the ring of endomorphisms of $L$
is finite-dimensional (over the same coefficient field $E$). To prove this
fact, we use induction on the dimension of $X$ to reduce to the case where
the cohomology sheaves of $L$ are local systems, and then it is trivial.
\end{proof}

Let $K\in\Dbc(U)$. Then $T:\Psi_f^u K\fl
\Psi_f^u K$ is unipotent, so there exists a unique nilpotent
$N:\Psi_f^u K\fl\Psi_f^u K(-1)$ such that $T=\exp(t_\ell(T)N)$ on $\Psi_f^u K$.
The operator $N$ is usually called the ``logarithm of the unipotent part of
the monodromy''. 
We get a functorial exact triangle $\Psi_f^u\stackrel{N}{\fl}\Psi_f^u(-1)\fl
i^*j_*\stackrel{+1}{\fl}$.

\subsection{Beilinson's construction}
\label{section_Beilinson_cycles}

Now we introduce the unipotent local systems that are used in Beilinson's
construction of $\Psi_f^u$.

\begin{subdef} For every $i\geq 0$, we define a $E$-local system
$\Lf_i$ on $\GD_m$ in the following way: the stalk
$\Lf_{i,1}$ of $\Lf_i$ at $1\in\GD_m(k)$ is the $E$-vector
space $E^{i+1}$, on which an element $u\rtimes\sigma$ of
$\pi_1(\GD_m,1)\simeq\widehat{\Z}(1)\rtimes\Gal(\overline
{k}/k)$ acts by $\exp(t_\ell(u)N)
\diag(1,\chi(\sigma)^{-1},\dots,\chi(\sigma)^{-i})$,
where
$\diag(x_0,\dots,x_{i})$ is the
diagonal matrix with diagonal entries $x_0,\dots,x_{i}$
and $N$ is the Jordan block
$\begin{pmatrix}0 & 1 & & 0 \\  & \ddots & \ddots & \\ &&\ddots & 1\\
0 &&  & 0\end{pmatrix}$.

If $i\leq j$, we have an obvious injection $\alpha_{i,j}:\Lf_i\fl\Lf_j$ and
an obvious surjection $\beta_{j,i}:\Lf_j\fl\Lf_i(i-j)$.
\label{def_L_a}
\end{subdef}

Note that $\Lf_i^\vee\simeq\Lf_i(i)$, so
(by the calculation at the end of section \ref{l-adic_complexes})
we have $D_U(\Lf_i)\simeq\Lf_i(i-1)[-2]$, and $D_U(\alpha_{i,j})$
corresponds by this isomorphism to $\beta_{j,i}(j-1)[-2]$.

\begin{subnotation} If $\Lf$ is a lisse sheaf on $\GD_m$ and $K$ is a perverse
sheaf on $U$, then the complex $K\Ltimes f^*\Lf$ is also perverse.
We denote it by $K\otimes\Lf$.

\end{subnotation}

We start with the construction of $\Psi_f^u$.

\begin{subprop} Let $K\in\Ob\Perv(U)$. 
\begin{itemize}
\item[(i)] For every $a\in\Nat$, we
have a canonical isomorphism
\[i_*
\Ker(N^{a+1},\Psi_f^u K)\iso\Ker(j_!(K\otimes\Lf_a)
\fl
j_*(K\otimes\Lf_a))=i_*\Hp^{-1}i^*j_*(K\otimes\Lf_a).\]
In particular, if $a$ is big enough, we get an isomorphism
$i_*\Psi_f^u K\iso i_*\Hp^{-1}i^*j_*(K\otimes\Lf_a)$.

\item[(ii)] For every $a\in\Nat$ such that $N^{a+1}=0$ on
$\Psi_f^u K$,
the following diagram is commutative:
\[\xymatrix{0\ar[r] & i_*\Psi_f^u K\ar[r]\ar@{=}[d] & j_!(K\otimes f^*\Lf_a)
\ar[r]\ar[d]^-{\alpha_{a,a+1}} & j_*(K\otimes f^*\Lf_a)
\ar[d]^-{\alpha_{a,a+1}} \\
0\ar[r] & i_*\Psi_f^u K\ar[r]\ar[d]^-{N} & j_!(K\otimes f^*\Lf_{a+1})
\ar[r]\ar[d]^-{\beta_{a,a+1}} & j_*(K\otimes f^*\Lf_{a+1})
\ar[d]^-{\beta_{a,a+1}} \\
0\ar[r] & i_*\Psi_f^u K(-1)\ar[r] & j_!(K\otimes f^*\Lf_a)(-1)\ar[r] &
j_*(K\otimes f^*\Lf_a)(-1) 
}\]

\item[(iii)] Let $a,b\in\Nat$ be such that
$N^{a+1}=N^{b+1}=0$ on $\Psi_f^u K$. Then there is a
canonical isomorphism
\[\Ker(j_!(K\otimes f^*\Lf_b)\fl j_*(K\otimes f^*\Lf_b))(-a-1)\iso
\Coker(j_!(K\otimes f^*\Lf_a)\fl j_*(K\otimes f^*\Lf_a))\]
induced by the connecting map coming from the commutative
diagram with exact rows
\[\xymatrix@C=25pt
{0\ar[r] & j_!(K\otimes f^*\Lf_a)\ar[r]^-{\alpha_{a,a+b+1}}\ar[d] &
j_!(K\otimes f^*\Lf_{a+b+1})\ar[r]^-{\beta_{b,a+b+1}}\ar[d] &
j_!(K\otimes f^*\Lf_b)(-a-1)\ar[r]\ar[d] & 0 \\
0\ar[r] & j_*(K\otimes f^*\Lf_a)\ar[r]^-{\alpha_{a,a+b+1}} &
j_*(K\otimes f^*\Lf_{a+b+1})\ar[r]^-{\beta_{a,a+b+1}} &
j_*(K\otimes f^*\Lf_b)(-a-1)\ar[r] & 0
}\]
Moreover, the morphism
\[\Hp^0 i^*j_*(K\otimes\Lf_a)\fl\Hp^0 i^*j_*(K\otimes\Lf_{a+b+1})\]
induced by $\alpha_{a,a+b+1}$ is zero.

\end{itemize}
\label{prop_Psi_fu}
\end{subprop}

Note in particular that we can use this construction to see
$\Psi_f^u$ as a functor from $\Perv(U)$ to $\Perv(Y)$ (and
not just to the category of $\Gal(\overline{k}/k)$-equivariant
objects in $\Perv(Y)$).

\begin{subcor}  For every $K\in \Perv(U)$, we have a canonical
isomorphism $D(\Psi_f^u K)\simeq\Psi_f^u(DK)(-1)$.

\end{subcor}

\begin{subcor} For every $a\in\Nat$,
we define a functor $C_a^\bullet:\Perv(U)\fl C^{[0,1]}(\Perv(X))$
(where the second category is the category of complexes concentrated in
degrees $0$ and $1$) by
\[C_a^\bullet(K)=(j_!(K\otimes\Lf_a)\fl j_*(K\otimes\Lf_a)).\]
With the transition morphisms given by the $\alpha_{a,b}$, the family
$(C_a)_{a\geq 0}$ becomes an inductive system of functors. 

Then we have canonical isomorphisms
\[i_*\Psi_f^u\simeq\varinjlim_{a\in\Nat}\H^0(C_a^\bullet)\]
and
\[0=\varinjlim_{a\in\Nat}\H^1(C_a^\bullet).\]

\end{subcor}

\begin{subrmk}
If we use the Ind-category $\Ind(\Perv(X))$ of $\Perv(X)$ (see for example
Chapter~6 of Kashiwara and Schapira's book
\cite{KS1}, and Theorem 8.6.5 of the same book for the fact
that this category is abelian), 
then we can reformulate this corollary in the
following way: 
We have a canonical isomorphism 
\[i_*\Psi_f^u\iso\varinjlim_{a\in\Nat}C_a^\bullet\]
of functors $\Perv(X)\fl
\Db\Ind(\Perv(X))$. Note that, by Theorem 15.3.1 of \cite{KS1}, the obvious
functor $\Db\Perv(X)\fl\Db\Ind(\Perv(X))$ is fully faithful (and its
essential image is the full subcategory of complexes with all their cohomology objects
in $\Perv(X)$). So $\varinjlim_{a\in\Nat}C_a^\bullet$ factors through the category
$\Db\Perv(X)$.

\label{rmk_Psi}
\end{subrmk}

We now give the definition of the maximal extension functor.

Let $K\in\Ob\Perv(U)$.
For each $a\geq 1$, we have a commutative diagram:
\[\xymatrix{j_!(K\otimes f^*\Lf_a)\ar[r]\ar[d]_{\beta_{a,a+1}} & 
j_*(K\otimes f^*\Lf_a)\ar[d]^{\beta_{a,a+1}} \\
j_!(K\otimes f^*\Lf_{a-1})(-1)\ar[r] & j_*(K\otimes f^*\Lf_{a-1})(-1)}\]
We write $\gamma_{a,a-1}:j_!(K\otimes f^*\Lf_a)\fl j_*(K\otimes
f^*\Lf_{a-1})(-1)$
for the diagonal map in this diagram.

\begin{subprop}
\begin{itemize}
\item[(i)] 
For $a\in\Nat$ big enough, the (injective) map
$\Ker(\gamma_{a,a-1})\fl\Ker(\gamma_{a+1,a})$ induced by
$\alpha_{a,a+1}:j_!(K\otimes f^*\Lf_a)\fl j_!(K\otimes f^*\Lf_{a+1})$ is
an isomorphism.
We write $\Xi_f K$ for the direct limit of the $\Ker(\gamma_{a,a-1})$.
This defines a left exact functor from $\Perv(U)$ to $\Perv(X)$, 
called the \emph{maximal extension functor}.

Moreover, if $a$ and $b$ are big enough, then the map
$\Coker(\gamma_{a,a-1})\fl\Coker(\gamma_{a+b,a+b-1})$ induced by $\alpha_{a-1,
a+b-1}(-1)$ is zero. In particular, we have
\[\varinjlim_a\Coker(\gamma_{a-1,a})=0,\]
and so $\Xi_f$ is exact.

\item[(ii)] We have
a functorial isomorphism $D_X\circ\Xi_f\simeq\Xi_f\circ D_U$ and two
functorial exact sequences
\[0\fl j_!\flnom{\alpha}\Xi_f\fl i_*\Psi_f^u(-1)\fl 0\]
and
\[0\fl i_*\Psi_f^u\fl\Xi_f\flnom{\beta} j_*\fl 0,\]
dual of each other,
in which the maps are the obvious ones. For example, in the first sequence,
the map $j_!K\fl\Xi_f K$ is induced by the injection $\alpha_{0,a}:j_!K=j_!(K
\otimes f^*\Lf_0)\fl j_!(K\otimes f^*\Lf_a)$, and the map
$\Xi_f\fl i_*\Psi_f^u(-1)$ is induced by the commutative square
\[\xymatrix{j_!(\cdot\otimes f^*\Lf_a)\ar[r]^-{\gamma_{a,a-1}}
\ar[d]_{\beta_{a,a-1}} & 
j_*(\cdot\otimes f^*\Lf_{a-1})(-1)\ar[d]^{\id} \\
j_!(\cdot\otimes f^*\Lf_{a-1})(-1)\ar[r] &
j_*(\cdot\otimes f^*\Lf_{a-1})(-1)}\]

\end{itemize}
\label{prop_Xi_f}
\end{subprop}

\begin{subrmk} As in Remark \ref{rmk_Psi}, we can deduce from
(i) of the proposition a natural isomorphism
\[\Xi_f K\iso \varinjlim_{a\in\Nat}
(j_!(K\otimes f^*\Lf_a)\flnom{\gamma_{a,a-1}}j_*(K\otimes
\Lf_{a-1})(-1))\]
in $\Db\Ind(\Perv(X))$.

\end{subrmk}

The next functor that we construct is the \emph{unipotent vanishing
cycles functor} $\Phi_f^u$. It is not very hard to show that this functor
is isomorphic to the direct summand of the usual vanishing cycles
functor on which the monodromy acts unipotently, but we will not need
this, so we will just use the following proposition as the definition of
$\Phi_f^u$.

\begin{subprop}
\begin{itemize}
\item[(i)] 
The complex of exact endofunctors of $\Perv(X)$ defined by
\[j_!j^*\flnom{\alpha+\eta}\Xi_f j^*\oplus\id\flnom{\beta-\varepsilon}j_*j^*\]
in degrees $-1$, $0$ and $1$, where $\eta:j_!j^*\fl\id$ is the counit of the
adjunction $(j_!,j^*)$ and $\varepsilon:\id\fl j_*j^*$ is the unit of the
adjunction $(j^*,j_*)$, has its cohomology concentrated in degree $0$ and
with support in $Y$.

We define an exact functor $\Phi_f^u:\Perv(X)\fl\Perv(Y)$ by setting
$i_*\Phi_f^u$ to be the $\H^0$ of this complex. 

\item[(ii)]
We denote by
$\can:\Psi_f^u j^*K\fl\Phi_f^u K$ the functorial map defined
by $i_*\Psi_f^u j^*K\fl\Xi_f j^*K$, and by
$\var:\Phi_f^u K\fl
\Psi_f^u j^*K(-1)$ the functorial map defined by
$\Xi_f j^*K\fl\Psi_f^u K(-1)$.

Then $\var\circ \can=N$ and $\can(-1)\circ \var=N$.

\item[(iii)]  We
have a functorial isomorphism $D\circ\Phi_f^u\simeq\Phi_f^u\circ D$, and the
duality exchanges $\can$ and $\var$.

\item[(iv)] There are canonical isomorphisms $\Ker(\can)=\Hp^{-1}i^*K$
and $\Coker(\can)=\Hp^0 i^*K$.

Dually, we have canonical isomorphisms $\Ker(\var)=\Hp^0 i^!K$ and
$\Coker(\var)=\Hp^1 i^!K$.

\end{itemize}
\label{prop_Phi_u}
\end{subprop}

Finally, we will need the functor that M. Saito calls $\Omega_f$.

\begin{subprop}
The functor $\beta+\varepsilon:\Xi_f j^*\oplus\id\fl j_*j^*$ is
surjective. Its kernel $\Omega_f$ is an exact endofunctor of $\Perv(X)$,
and we have functorial exact sequences
\[0\fl j_!j^*\flnom{\alpha-\eta}\Omega_f\fl i_*\Phi_f^u\fl 0\]
and
\[0\fl i_*\Psi_f^u j^*\fl\Omega_f\fl \id\fl 0,\]
in which the unmarked maps are the obvious ones.

\label{prop_Omega_f}
\end{subprop}

\section{Nearby cycles and mixed perverse sheaves}
\label{cycles1}

The goal of this section is to show that the functor of unipotent nearby cycles
preserves the categories $\Perv_{mf}(X)$ and to deduce that these categories
are also preserved by the functors $\Hp^k f_*$, for every morphism
$f$ of $\Sch/k$. The main tool is Deligne's weight-monodromy theorem
from \cite{De}.

We also give an application to the direct image functor by
a closed immersion $i$, which then allows us to
construct the functor $i^*$ on the categories $\Db\Perv_{mf}$.

\subsection{Nearby cycles on horizontal perverse sheaves}

We assume again that $k$ is a field that is finitely generated over
its prime field. Let $X$ be a $k$-scheme and $f:X\fl\Aff^1$
be a morphism. We write
$\GD_m=\Aff^1-\{0\}$, 
$U=X\times_{\Aff^1}\GD_m\flnom{j}X$ and
$Y=X\times_{\Aff^1}\{0\}\flnom{i}X$.

%We will use the constructions of section
%\ref{section_Beilinson_cycles} to \emph{define} the functors
%$\Psi_f^u$, $\Phi_f^u$, $\Xi_f^u$ and $\Omega_f^u$ on the category
%$\Perv_h(U)$. As the lisse sheaves $\Lf_a$ on $\GD_m$ are clearly horizontal
%and as we have the six operations on the categories $\Dbh$, this
%makes perfect sense, and it is compatible with the usual constructions
%via the functor $\eta^*:\Dbh\fl\Dbc$. Note also that these functors respect
%the subcategories of mixed perverse sheaves, because all the functors
%used in their definition respect the categories of mixed complexes.

In Section~\ref{section_5}, we have constructed exact functors
$\Psi_f^u:\Perv(U)\ra\Perv(Y)$ (see Proposition~\ref{prop_Psi_u}),
$\Phi_f^u:\Perv(X)\ra\Perv(Y)$ (see Proposition~\ref{prop_Phi_u}(i)),
$\Xi_f:\Perv(U)\ra\Perv(X)$ (see Proposition~\ref{prop_Psi_fu}(i)) and
$\Omega_f:\Perv(X)\ra\Perv(X)$ (see Proposition~\ref{prop_Omega_f}).
The next proposition says that all these functors respect the full subcategories of
horizontal perverse sheaves (resp. mixed perverse sheaves).

\begin{subprop}
For $?\in\{h,m\}$,
the functor $\Psi_f^u$ (resp. $\Phi_f^u$, $\Xi_f$, $\Omega_f$) sends the full
subcategory $\Perv_?(U)$ (resp. $\Perv_?(X)$, $\Perv_?(U)$, $\Perv_?(X)$) of
$\Perv(U)$ (resp. $\Perv(X)$, $\Perv(U)$, $\Perv(X)$) to the full subcategory
$\Perv_?(Y)$ (resp. $\Perv_?(Y)$, $\Perv_?(X)$, $\Perv_?(X)$) of
$\Perv(Y)$ (resp. $\Perv(Y)$, $\Perv(X)$, $\Perv(X)$).

\label{prop_cycles_mixed}
\end{subprop}

\begin{proof}
Let $K\in\Perv_h(U)$ (resp. $\Perv_m(U)$). As the lisse sheaves $\Lf_a$ on $\Gr_m$ of Definition~\ref{def_L_a}
are horizontal and mixed, the perverse
sheaves $K\otimes f^*\Lf_a$ are horizontal (resp. mixed), so the perverse sheaves $\Hp^{-1}i^*j_*(K\otimes\Lf_a)$ are all horizontal
(resp. mixed).
As $\Psi_f^u(K)\iso\Hp^{-1}i^*j_*(K\otimes\Lf_a)$ for $a$ big enough (Proposition~\ref{prop_Psi_fu}(i)), we
deduce that $\Psi_f^u(K)$ is in $\Perv_h(Y)$ (resp. $\Perv_m(Y)$). The statement for $\Xi_f$, $\Phi_f^u$ and $\Omega_f$
follows immediately from their definition, once we know that the $\Lf_a$ are horizontal and mixed.
\end{proof}

As only finite type schemes and
constructible complexes are involved in the definition of $\Psi_f^u$
if we use the alternative definition given by Proposition~\ref{prop_Psi_fu}(i)), we
can use Deligne's generic base change to compare our situation with the situation
over closed points of some ring $A\in\Uf$. We will see an
example of this in the next section.

\subsection{The relative monodromy filtration}
\label{nearby}

We recall the definition of the relative monodromy filtration, due
to Deligne.

\begin{subprop}(See Propositions 1.6.1 and 1.6.13 of \cite{De}.)
Let $K$ be an object
in some abelian category, and suppose that we have a finite increasing
filtration $W$ on $K$ and a nilpotent endomorphism $N$ of $K$. Then there
exists at most one finite increasing filtration $M$ on $K$ such that
$N(M_i)\subset M_{i-2}$ for every $i\in\Z$ and that, for every $k\in\Nat$
and every $i\in\Z$,
the morphism $N^k$ induces isomorphisms
\[\Gra_{i+k}^M\Gra_i^W K\iso \Gra_{i-k}^M\Gra_i^W K.\]

Moreover, if $W$ is trivial (that is, if there exists $i\in\Z$ such that
$\Gra_i^W K=K$), then the filtration $M$ always exists.

\label{prop_monodromy_filtration}
\end{subprop}

The filtration $M$ is called the \emph{monodromy filtration on $K$ relative
to the filtration $W$}. If $W$ is trivial, it is simply called the
\emph{monodromy filtration on $K$}.

We will use the following theorem, which is a close relative
of Theorem 1.8.4 of Deligne's Weil II paper \cite{De}.

\begin{subthm} Let $K\in\Ob\Perv_{mf}(U)$, and let $W$ be the weight filtration
on $K$. Then the monodromy filtration $M$
on $\Psi^u_f K$ relative to the filtration
$\Psi^u_f W$ exists, and $\Gra_i^M\Psi^u_f K$ is pure of weight $i-1$ for every
$i\in\Z$. In particular, $\Psi^u_f K$ is an object of $\Perv_{mf}(Y)$.

\label{thm_weight_filtration_psi}
\end{subthm}

Before giving the proof of the theorem, we state and prove a lemma.

\begin{sublemma} In the situation of the theorem, suppose that $K$ is
pure. Then the monodromy filtration $M$ on $\Psi_f^u K$ (which always exists)
is such that $\Gr_i^M\Psi_f^u K$ is pure of weight $i-1$ for every
$i\in\Z$.

\label{lemma_weight_filtration_psi}
\end{sublemma}

\begin{proof} Let $w$ be the weight of $K$.
Let $(A,\X,u)$ be an object of $\Uf X$ such that
$K$ comes by restriction from a shifted perverse sheaf $\mathcal{K}[-d]$ on
$\X$, where $d=\dim\Spec A$.
Fix $a\in\Nat$ such that $N^{a+1}=0$ on $\Psi_f^u K$.
After shrinking $\Spec A$ and $\X$ if necessary, we may assume that :
\begin{itemize}
\item[-] The morphism $f:X\fl\Aff^1_k$ extends to a morphism
$F:\X\fl\Aff^1_A$.
We write
$\mathcal{U}=\X\times_{\Aff^1_A}\GD_{A,m}\flnom{J}\X$ and
$\Y=\X\times_{\Aff^1_A}\{0\}\flnom{I}\X$.
\item[-] The lisse sheaves $\Lf_0,\ldots,\Lf_{a+1}$ all extend to
$\GD_{m,A}$. (In fact we can get all the $\Lf_b$ as soon as we have
$\Lf_1$, because they are
the symmetric powers of $\Lf_1$.)
\item[-] For every closed point $x$ of $\Spec A$, the restriction of
$\mathcal{K}$ to $\X_x$ is still perverse, and it is pure of weight $w+d$.
\item[-] The formation of the complexes
$J_!(\mathcal{K}\otimes\Lf_b)$ and $J_*(\mathcal{K}\otimes\Lf_b)$, 
for $b\in\{a,a+1\}$,
is compatible with every base change $x\fl\Spec A$, where $x$ is a closed
point of $\Spec A$. Moreover, if $\Lf$ is any subquotient
of $\Hp^{-1}I^*J_*(\mathcal{K}\otimes\Lf_a)$ (in the category
$\Perv(\X,E)$), then its restrictions to the fibers of
$\X$ above all the closed points of $\Spec A$ are still perverse.

\end{itemize}
Indeed, the first two points are standard, and the last two follow
from Deligne's generic base change theorem (see SGA 4 1/2,
[Th. finitude], Th\'eor\`eme 1.9) and from the purity theorem.

Let $\mathcal{K}'=\Hp^{-1}I^*J_*(\mathcal{K}\otimes\Lf_a)$, and
let $M$ the monodromy filtration on $\mathcal{K}'$ induced by $N$.
By the conditions above (and (i) of Proposition \ref{prop_Psi_fu}),
for every closed point $x$ of $\Spec A$, the restriction of
$\mathcal{K'}$ to $\X_x$ is a subobject of $\Psi_f^u\mathcal{K}_x$,
and the restriction of $M$ is the monodromy filtration. The result
about the weights of the graded pieces of the monodromy filtration
over the spectrum of a finite field (such as $x$) is known
by Theorem 5.1.2 of Beilinson and Bernstein's paper \cite{BB} (where it is
attributed to Gabber). So we get the conclusion by definition of the weights
on horizontal sheaves.
\end{proof}

\begin{proof}[Proof of Theorem~\ref{thm_weight_filtration_psi}]
We reason by induction on the length of the filtration $W$.
If $K$ is pure (i.e., if $W$ is trivial), then the conclusion of
the theorem is proved in Lemma \ref{lemma_weight_filtration_psi}.

Now assume that $W$ is of length $\geq 2$, and that we know the result
for every object of $\Perv_{mf}(U)$ with a shorter weight filtration.
Let $a\in\Z$ be such that $W_a K=K$ and $\Gra_a^W K\not=0$. By the induction
hypothesis, we know the theorem for $W_{a-1}K$ and $\Gra_a^W K$.
Write $L=\Psi^u_f K$, and let $F$ be the filtration $\Psi^u_f W$ on $L$.
By Theorem 2.20 of Steenbrink and Zucker's paper \cite{SZ}, the filtration
$M$ exists if and only if, for every $i\geq 1$, we have :
\[N^i(L)\cap F_{a-1}L(-i)\subset N^i(F_{a-1}L)+M_{a-i-1}F_{a-1}L(-i).\]
This is equivalent to saying that
\[(N^i(L)\cap F_{a-1}L(-i))/N^i(F_{a-1}L)\subset [M_{a-i-1}F_{a-1}L/(M_{a-i-1}
F_{a-1}L\cap N^i(F_{a-1}L))](-i).\]
As the filtration $M$ on $F_{a-1}L$ is the weight filtration up to a shift,
the inclusion above is also equivalent to the fact that
$(N^i(L)\cap F_{a-1}L(-i))/N^i(F_{a-1}L)$ is of weight $\leq a+i-2$. Observe
that $(N^i(L)\cap F_{a-1}L(-i))/N^i(F_{a-1}L)$ is the kernel of the
map $F_{a-1}L(-i)/N^i(F_{a-1}L)\fl L(-i)/N^i(L)$, so applying the snake lemma to
the commutative diagram with exact rows :
\[\xymatrix{0\ar[r] & F_{a-1}L\ar[r]\ar[d]^{N^i} & L\ar[r]\ar[d]^{N^i} &
\Gra_a^F L\ar[r]\ar[d]^{N^i} & 0 \\
0 \ar[r] & F_{a-1}L(-i)\ar[r] & L(-i)\ar[r] & \Gra_a^F L(-i)\ar[r] & 0}\]
gives a surjection
\[\Ker(N^i:\Gra_a^F L\fl\Gra_a^F L(-i))\fl
(N^i(L)\cap F_{a-1}L(-i))/N^i(F_{a-1}L).\]
But as $\Gra_a^F L=\Psi^u_f\Gra_a^W K$, we know by (1.6.4) of \cite{De} that
$\Ker(N^i:\Gra_a^F L\fl\Gra_a^F L(-i))$ is of weight $\leq a+i-2$
(or more correctly, we can deduce
this from the result we cited and Deligne's generic base change theorem, as in the proof
of Lemma \ref{prop_monodromy_filtration}),
and hence
all its quotients are. This proves the existence of the filtration $M$
on $L$.

Finally, we prove that $\Gra_i^M L$ is pure of weight $i-1$ for every
$i\in\Z$. The two properties defining $M$ in Proposition
\ref{prop_monodromy_filtration} stay true if we intersect
$M$ with $F_{a-1}L$ or take the quotient filtration in $\Gra_a^F L$, so this
gives the relative monodromy filtration on $F_{a-1}L$ and $\Gra_a^F L$
(by the
uniqueness statement). Hence we get exact sequences
\[0\fl\Gra_i^M F_{a-1}L\fl\Gra_i^M L\fl\Gra_i^M\Gra_a^F L\fl 0,\]
and so the fact that $\Gra_i^M L$ is pure of weight $i-1$ follows from the
induction hypothesis.
\end{proof}

\subsection{Cohomological direct image functors and weights}
\label{stab_Pmf}

\begin{subcor} Let $f:X\fl Y$ be a morphism of $k$-schemes. Then the
functors $\Hp^i f_*,\Hp^i f_!
:\Perv_m(X)\fl\Perv_m(Y)$ send $\Perv_{mf}(X)$ to $\Perv_{mf}(Y)$.

\label{cor_f_*_preserves_Perv_mf}
\end{subcor}

\begin{proof} As Poincar\'e-Verdier duality exchanges $\Hp^i f_*$ and
$\Hp^{-i}f_!$ and preserves the categories $\Perv_{mf}$, it suffices to treat
the case of $\Hp^i f_*$.

By Nagata's compactification theorem
(see for example Conrad's paper \cite{Co}), we can write $f=gj$, with
$j:X\fl X'$ an open embedding and $g:X'\fl Y$ proper. After replacing $X'$ by
the blowup of
$X'-X$ in $X'$, we may assume that the ideal of $X'-j(X)$ is
invertible. Then $j$ is affine, so $j_*$ is t-exact, so
we have $\Hp^i f_*=(\Hp^i g_*)\circ j_*$ for every $i\in\Z$. By Proposition
\ref{prop_obvious1}(ii), it suffices to prove the corollary for $j$. 
By Proposition
\ref{prop_obvious3}, we may assume that $X'$ is affine, and hence that there
exists $h\in\Of(X')$ generating the ideal of $X'-j(X)$.

So we see that it is enough to prove the corollary in the following situation :
there exists $h:Y\fl\Aff^1_k$ such that $f=j$ is the inclusion of $X:=h^{-1}(\GD_m)$
in $Y$. Let $i:Y-X\fl X$ be the inclusion of the complement. Let $K$ be
an object of $\Perv_{mf}(X)$, and denote by $W$ its weight filtration.
Let $a\in\Z$. We want to find a subobject $L$ of $j_* K$ such that
$L$ is of weight $\leq a$ and $j_* K/L$ is of weight $>a$. (This clearly
implies that $j_* K$ has a weight filtration.) 

If $W_a K=0$, then $K$ is of
weight $>a$, so $j_*K$ is of weight $>a$, and we take $L=0$.

If $W_a K=K$,
then $K$ is of weight $\leq a$, so $j_{!*}K$ is of weight $\leq a$
by Corollary 5.4.3 of \cite{BBD}. So it is enough to find a subobject
$L'$ of weight $\leq a$ of $j_*K/j_{!*}K$ such that $(j_*K/j_{!*}K)/L'$ is
of weight $>a$. But we know that $j_*K/j_{!*}K=i_*\Hp^0 i^*j_*K$
(by (4.1.11.1) of \cite{BBD}), which
is a quotient of $i_*\Psi^u_f K(-1)$. As $\Psi^u_f K$ has a weight filtration
by Theorem \ref{thm_weight_filtration_psi}, so does $j_*K/j_{!*}K$, and
we can find a $L'$ with the desired properties.

Suppose that $0\not=W_a K\not=K$, and let $K'=W_a K$ and $K''=K/W_a K$.
By the previous paragraph, there exists a subobject $L'$ of weight $\leq a$
of $j_* K'$ such that $j_* K'/L'$ is of weight $>a$. As $K''$ is of weight
$>a$, so is $j_* K''$. Using the exact sequence
\[0\fl j_* K'\fl j_* K\fl j_* K''\fl 0,\]
we see that $j_* K/L'$ is also of weight $>a$, so we can take $L=L'$.
\end{proof}

\begin{subcor} Let $j:U\fl X$ be an affine open embedding. 
Denote by 
$j^*:D^b\Perv_{mf}(X)\fl D^b\Perv_{mf}(U)$ and
$j_*:D^b\Perv_{mf}(U)\fl D^b\Perv_{mf}(X)$ the
derived functors of the exact functors
$j^*:\Perv_{mf}(X)\fl \Perv_{mf}(U)$ and
$j_*:\Perv_{mf}(U)\fl\Perv_{mf}(X)$.

Then this derived functors $(j^*,j_*)$ form a pair of adjoint functors.

\label{cor_adjunction_j}
\end{subcor}

\begin{proof} By Corollary 8.12 of \cite{Shul}, it suffices to prove that
the underived functors form a pair of adjoint functors. But, once we know
that both functors preserve the full subcategories $\Perv_{mf}\subset\Perv_m$,
this follows from the adjunction for the categories $\Perv_m$.
\end{proof}

\begin{subcor} The exact functors $\Psi_f^u$, $\Phi_f^u$, $\Xi_f$ and
$\Omega_f$ of section \ref{section_Beilinson_cycles}
preserve the full subcategories of mixed
perverse sheaves with weight filtrations.

\label{cor_nearby_etc}
\end{subcor}

\begin{proof} We already know the result for $\Psi_f^u$, by Theorem
\ref{thm_weight_filtration_psi}.

Suppose that $K\in\Perv_{mf}(U)$. Then $K\otimes f^*\Lf_i$
is in $\Perv_{mf}(U)$ for every $i\geq 0$. Indeed, if we denote by $W$ the
weight filtration on $K$, then we get a weight filtration on $K\otimes f^*\Lf_i$
by setting
\[W_a(K\otimes f^*\Lf_i)=\sum_{0\leq j\leq i}(W_{a-2j} K)\otimes f^*\Lf_j.\]
By Corollary \ref{cor_f_*_preserves_Perv_mf}, we see that
$j_!(K\otimes f^*\Lf_i)$ and $j_*(K\otimes f^*\Lf_i)$ are in $\Perv_{mf}(X)$
for every $i\geq 0$. By definition of $\Xi_f$, this implies that
$\Xi_f K\in\Perv_{mf}(X)$. The conclusion for  $\Omega_f$ then
follows from its construction in Proposition 
\ref{prop_Omega_f}.
Finally, by the construction in
Propositions \ref{prop_Xi_f}, the functor $i_*\Phi_f$ is a subquotient of
$\Xi_f j^*\oplus\id$. As $\Perv_{mf}(X)$ is stable by subquotients in $\Perv_m(X)$, the
functor $i_*\Phi_f$ sends $\Perv_{mf}(X)$ to itself. By Lemma \ref{lemma_test_weight_on_i},
this implies that $\Phi_f$ sends $\Perv_{mf}(X)$ to $\Perv_{mf}(Y)$.
\end{proof}

\subsection{Direct and inverse image by a closed immersion}

Let $X$ be a $k$-scheme and $Y\flnom{i}X$
be a closed subscheme of $X$.
We denote by $D^b_Y\Perv_{mf}(X)$ the full subcategory of $D^b\Perv_{mf}(X)$
whose objects are the complexes $K$ such that the support of
$\H^i K\in\Perv_{mf}(X)$
is contained in $Y$ for every $i\in\Z$. The exact functor $i_*:\Perv_{mf}(Y)
\fl\Perv_{mf}(X)$ induces a functor $i_*:D^b\Perv_{mf}(Y)\fl D^b\Perv_{mf}(X)$,
whose image is obviously in contained in $D^b_Y\Perv_{mf}(X)$.

\begin{subcor} With notation as above, the functor
$i_*:D^b\Perv_{mf}(Y)\fl D^b_Y\Perv_{mf}(X)$ is an equivalence of
categories.

We have a similar equivalence $\Dbm(Y)\iso\Dbm_{,Y}(X)$, where
$\Dbm_{,Y}(X)$ is the full subcategory of objects $K$ of $\Dbm(X)$ such that $\Hp^i K$ is in
$i_*\Perv_m(Y)$ for every $i\in\Z$.

Moreover, we can choose inverses $(i_*)^{-1}$
of these equivalences such that the
following diagram commutes :
\[\xymatrix{D^b\Perv_{mf}(Y)\ar[r]^-{i_*}\ar[d]_{R_Y} & 
D^b_Y\Perv_{mf}(X)\ar[r]^-{(i_*)^{-1}}\ar[d]_{R_X} & 
D^b\Perv_{mf}(Y)\ar[d]_{R_Y} \\
\Dbm(Y)\ar[r]_-{i_*} & \Dbm_{,Y}(X)\ar[r]_-{(i_*)^{-1}} & 
\Dbm(Y) \\
}\]
where the functors $R_X$ and $R_Y$ are defined in Theorem \ref{thm_main1}.

\label{cor_support}
\end{subcor}

\begin{proof} We prove the first statement.
It suffices to prove that, for all $K,L\in\Ob\Perv_{mf}(Y)$ and every $n\in\Z$, the functor
$i_*$ induces an isomorphism $\Hom_{\Db\Perv_{mf}(Y)}(K,L[n])\iso
\Hom_{\Db\Perv_{mf}(X)}(i_*K,i_*L[n])$. Note that both of these $\Hom$ groups are $0$ for
$n<0$, so we only need to consider the case $n\geq 0$.
Fix $K\in\Ob\Perv_{mf}(Y)$. The families of functors
$(L\fle\Hom_{\Db\Perv_{mf}(Y)}(K,L[n]))_{n\geq 0}$ 
and $(L\fle\Hom_{\Db\Perv_{mf}(X)}(i_*K,i_*L[n]))_{n\geq 0}$
are $\delta$-functors from $\Perv_{mf}(Y)$
to the category of abelian groups (in the sense of Definition 
\cite[\href{https://stacks.math.columbia.edu/tag/010Q}{Tag 010Q}]{SP}), and $i_*$ induces a
morphism between these $\delta$-functors
(see Definition
\cite[\href{https://stacks.math.columbia.edu/tag/010R}{Tag 010R}]{SP}). We want to show that
this morphism is an isomorphism.
We know that $i_*:\Hom_{\Perv_{mf}(Y)}(K,L)\iso\Hom_{\Perv_{mf}(X)}(i_*K,i_*L)$ is
an isomorphism for every $L\in\Ob\Perv_{mf}(Y)$ (because this is true in the categories
$\Perv_m$).
 Moreover, it follows easily from the Yoneda description of the extension groups in the
 derived category (see Section 3.2 of Chapter III of Verdier's book \cite{Ve} or
 Lemma \cite[\href{https://stacks.math.columbia.edu/tag/06XU}{Tag 06XU}]{SP}) that
the first of the two $\delta$-functors introduced above is effacable, i.e. satisfies the
hypothesis of Lemma \cite[\href{https://stacks.math.columbia.edu/tag/010T}{Tag 010T}]{SP}), and
hence is a universal $\delta$-functor
(see Definition
 \cite[\href{https://stacks.math.columbia.edu/tag/010S}{Tag 010S}]{SP}).
 By Lemma \cite[\href{https://stacks.math.columbia.edu/tag/010U}{Tag 010U}]{SP} 
(and Lemma \cite[\href{https://stacks.math.columbia.edu/tag/010T}{Tag 010T}]{SP} again), it
suffices to prove that the second $\delta$-functor is also effacable. So we want to prove that,
for all $K,L\in\Ob\Perv_{mf}(Y)$, every $n\geq 1$
and every $u\in\Hom_{\Db\Perv_{mf}(Y)}(K,L[n])$, there exists an injective
morphism
$L\fl L'$ in $\Perv_{mf}(Y)$ such that the image of $u$ in
$\Hom_{D^b\Perv_{mf}(X)}(i_* K,i_*L'[n])$ is $0$.

Let $(U_a)_{a\in A}$ be a finite affine cover of $X$.
For every $a\in A$, we have a cartesian diagram of immersions
\[\xymatrix{Y\ar[r]^-i & X \\ Y\cap U_\alpha\ar[r]_-{i_a}\ar[u]^{j'_a} & 
U_\alpha\ar[u]_{j_a}}\]
As $j'_\alpha$ and $j_\alpha$ are affine, the functors $j'_{\alpha*}$ and $j_{\alpha*}$ are
t-exact.
Let $L\in\Ob\Perv_{mf}(Y)$. By Corollary \ref{cor_f_*_preserves_Perv_mf},
the isomorphisms $i_* j'_{a*}{j'_a}^*L\simeq j_{a*}j_a^* i_*L$ and
$j_a^* i_*L\simeq i_{a*}{j'_a}^*L$ in
$\Perv_m(X)$ and $\Perv_m(U_\alpha)$
are isomorphisms of objects of $\Perv_{mf}(X)$ and $\Perv_{mf}(U_\alpha)$.
Using this and Corollary \ref{cor_adjunction_j}
we get for
$K,L\in\Ob\Perv_{mf}(Y)$ and $n\in\Z$ a canonical isomorphism
\[\Hom_{D^b\Perv_{mf}(X)}(i_* K,i_*j'_{a*}{j'_a}^*L[n])=
\Hom_{D^b\Perv_{mf}(U_a)}(i_{a*}{j'_a}^*K,i_{a*}{j'_a}^*L[n]).\]
As we have an injective morphism $L\fl\bigoplus_{a\in A}
j_{a*}j_a^* L$ in $\Perv_{mf}(Y)$ (by Corollary 
\ref{cor_f_*_preserves_Perv_mf} again), this reduces the corollary to the case
where $X$ is affine.

Now suppose that $X$ is affine. By an easy induction on the number of generators
of the ideal of $Y$, we may assume that this ideal only has one generator, i.e.,
that there exists a function $f:X\fl\Aff^1$ such that $Y=X\times_{\Aff^1}\{0\}$.
The exact functor $\Phi_f^u:\Perv_{mf}(X)\fl\Perv_{mf}(Y)$ induces a functor
$\Phi_f^u:D^b_Y\Perv_{mf}(X)\fl D^b\Perv_{mf}(Y)$, and we have $\Phi_f^u\circ
i_*\simeq\id_{D^b\Perv_{mf}(Y)}$. Let's show that $i_*\circ\Phi_f^u\simeq
\id_{D^b_Y\Perv_{mf}(X)}$, which will finish the proof.
By Proposition \ref{prop_Omega_f}
and Corollary \ref{cor_nearby_etc}
we have two exact sequences
of exact endofunctors of $\Perv_{mf}(X)$ :
\[0\fl j_!j^*\fl\Omega_f\fl i_*\Phi_f^u\fl 0\]
and
\[0\fl i_*\Psi_f^u j^*\fl\Omega_f\fl\id\fl 0,\]
where $j:X-Y\fl X$ is the inclusion.
Note that the restriction of the functor $j^*:D^b\Perv_{mf}(X)\fl
D^b\Perv_{mf}(U)$ to the full subcategory $D^b_Y\Perv_{mf}(X)$ is zero.
Hence the exact sequences above induces isomorphisms of endofunctors
of $D^b_Y\Perv_{mf}(X)$ :
\[i_*\Phi_f^u\stackrel{\sim}{\longleftarrow}\Omega_f\iso\id.\]

The proof of the second equivalence of categories is similar, except that we don't need
to use the Yoneda description  to show that the $\Ext$ groups in
$\Dbm(Y)$ define a $\delta$-functor.

The last statement of the Corollary
follows from the fact that we have isomorphisms
\[R_Y\circ\Phi_f^u\simeq\Phi_f^u\circ R_X\]
and
\[R_X\circ\Omega_f\simeq\Omega_f\circ R_X.\]
\end{proof}

% \subsection{Inverse image by a closed immersion}
%\label{closed}

\begin{subcor} Let $i:X\fl Y$ be a closed immersion. Denote by
$i_*:D^b\Perv_{mf}(X)\fl D^b\Perv_{mf}(Y)$ the derived functor of
the exact functor $i_*:\Perv_{mf}(X)\fl\Perv_{mf}(Y)$.

Then this functor $i_*$ admits a left adjoint
$i^*:D^b\Perv_{mf}(Y)\fl D^b\Perv_{mf}(X)$, and the counit $i^*i_*\fl\id$ of this
adjunction is an isomorphism.
Moreover, we have
an invertible natural transformation $\theta_i:i^*\circ R_Y\iso R_X\circ
i^*$. 

Finally, if $i':Y\fl Z$ is another closed immersion, then the following
diagram is commutative :
\[\xymatrix{R_X\circ i^*{i'}^*\ar[d]^\wr\ar[r]^-{\theta_i} & i^*\circ R_Y\circ{i'}
\ar[r]^-{\theta_{i'}} & i^*{i'}^*\circ R_Z\ar[d]^\wr \\
R_X\circ(i'i)^*\ar[rr]_-{\theta_{i'i}} & & (i'i)^*\circ R_Z}\]
where the vertical maps come from the composition isomorphisms $i'_*i_*\simeq
(i'i)_*$ and the uniqueness of the adjoint.

\label{cor_i^*}
\end{subcor}

\begin{proof} By Corollary \ref{cor_support}, we have an equivalence of
categories $i_*:D^b\Perv_{mf}(X)\iso D^b_X\Perv_{mf}(Y)$, where
$D^b_X \Perv_{mf}(Y)$ is the full subcategory of $D^b\Perv_{mf}(Y)$ whose
objects are complexes $K$ such that the support of
$H^i K\in\Perv_{mf}(Y)$ is contained in $X$
for every $i\in\Z$. So, to show that $i_*:D^b\Perv_{mf}(X)\fl D^b\Perv_{mf}(Y)$
admits a left adjoint, it suffices to show that the inclusion $\alpha:D^b_X
\Perv_{mf}(Y)\fl D^b\Perv_{mf}(Y)$ admits a left adjoint. 
Let $j:Y-X\fl Y$ be the inclusion. Then we have an exact triangle
$j_!j^*\fl\id\fl i_*i^*\flnom{+1}$ of endofunctors of $D^b_m(Y)$, and we can
make sense of the first two terms in $D^b\Perv_{mf}(Y)$, so we will try to
construct the left adjoint of $\alpha$ as their cone.

More precisely, 
let $(U_i)_{i\in I}$ be a finite open affine cover of $U:=Y-X$. For every
$J\subset I$, we denote by $j_J:\bigcap_{i\in J}U_i\fl X$ the inclusion.
As $X$ is separated, all the finite intersections of $U_i$'s are affine,
so the morphism $j_J$ is affine for every $J\subset I$.
If $K\in\Ob\Perv_{mf}(X)$, we denote by $D^\bullet(K)$ the complex
of $\Perv_{mf}(X)$ defined by
$D^{-r}(K)=\bigoplus_{|J|=r}j_{J!}j_J^*K$ if $r\geq 1$, $D^0(K)=K$ and $D^r(K)=0$
if $r\geq 1$, where the maps 
$D^{-r-1}(K)\fl D^{-r}(K)$, $r\geq 0$,
are alternating sums of adjunction morphisms. Note that we have a morphism
of complexes $K\fl D^\bullet(K)$, where $K$ is in degree $0$. 
Also, there is a canonical morphism $D^{-1}(K)\fl j_!j^*K$, which induces an
isomorphism $D^{\leq -1}(K)\iso j_!j^*K[1]$ in $D^b\Perv_{mf}(Y)$, so we get
a quasi-isomorphism $R_Y(D^\bullet(K))\iso i_*i^* K$ in $D^b_m(Y)$.
In particular, $D^\bullet(K)$ is in $D^b_X\Perv_{mf}(Y)$.
Note that the construction of $D^\bullet(K)$ is functorial in $K$, so
we can define a functor $\beta:D^b\Perv_{mf}(Y)\fl D^b_X\Perv_{mf}(Y)$ by sending
a complex $K$ to the total complex of the double complex $D^\bullet(K)$.

Let's show that $\beta$ is left adjoint to $\alpha$. 
For every complex $K$ of objects of $\Perv_{mf}(Y)$,
the morphism of double complexes $K\fl D^\bullet(K)$ induces a morphism
$\varepsilon_K:
K\fl\alpha\beta(K)$ in $D^b\Perv_{mf}(Y)$. If moreover $K$ is in $D^b_X\Perv_{mf}(X)$,
then $K\fl D^\bullet(K)$ is a quasi-isomorphism, so we get an
isomorphism $\eta_K:\beta\alpha(K)\iso K$. Moreover, the morphism
$\alpha(K)\flnom{\varepsilon_K \alpha} \alpha\beta\alpha(K)\flnom{\alpha\eta_K}\alpha(K)$
is clearly the
identity of $\alpha(K)$. So we have constructed the unit and counit of the
adjunction, and shown that the counit is an isomorphism.

To construct the isomorphism $\theta_i$, we use the isomorphism
$R_Y\circ V\iso i_*i^*\circ R_X$ constructed above and the last statement
of Corollary \ref{cor_support}. The last statement is also easy
to check.
\end{proof}

\section{Construction of the stable homotopic $2$-functor $H_{mf}$}

\subsection{Direct images}
\label{direct_images}

If $f:X\fl Y$ is a morphism of $k$-schemes, we write ${}^0 f_*$ for
$\Hp^0 f_*$. Remember that if $f$ is affine, then $f_*$ is right
t-exact for the perverse t-structure by \cite{BBD} Theorem 4.1.1.

In this section, we want to prove the following result.

\begin{subprop} 
There exists a $2$-functor $H_{mf,*}:\Sch/k\fl\TR$
with $H_{mf,*}(X)=D^b\Perv_{mf}(X)$ for every $X\in\Ob(\Sch/k)$ and
a natural transformation $R:H_{mf,*}\fl H_{m,*}$ (with the notation
of Example \ref{example_crossed_functor}) such that :
\begin{itemize}
\item[(a)] for every $X\in\Ob(\Sch/k)$, the functor $R_X:\Db\Perv_{mf}(X)\fl
\Dbm(X)$ is the composition of the obvious functor
$\Db\Perv_{mf}(X)\fl\Db\Perv_m(X)$ and of the realization functor
$\Db\Perv_m(X)\fl\Dbm(X)$ (see
Section \ref{def_mixed});
\item[(b)] for every morphism
$f:X\fl Y$, the natural transformation 
$\gamma_f:R_Y\circ H_{mf,*}(f)\fl
H_{m,*}(f)\circ R_X$ is an isomorphism.
\end{itemize}

\label{prop_direct_images}
\end{subprop}

The proof of the proposition will occupy most of this section.
The main ingredients are :
\begin{itemize}
\item Beilinson's version of the ``basic lemma'' that provides $f_*$-acyclic perverse
sheaves for $f:X\ra Y$ an affine morphism.
(See Theorem \ref{basic_lemma}.)

\item The fact that the functors $\Hp^k f_*$ (and in particular ${}^0 f_*$)
preserve the categories $\Perv_{mf}$.
(See Corollary \ref{cor_f_*_preserves_Perv_mf}.)

\item {\v C}ech resolutions for finite open affine coverings.

\end{itemize}

We first review Beilinson's basic lemma.

\begin{subthm} Let $f:X\fl Y$ be a morphism of $k$-schemes, let $(a_i:U_i\ra X)_{i\in I}$ be a finite
family of open affine subschemes of $X$ and let $K\in\Ob\Perv(X)$. Then there exists 
a finite family $(b_j:V_j\ra X)_{j\in J}$ of open affine subschemes of $X$ such that:
\begin{itemize}
\item[(i)] The canonical morphism $\bigoplus_{j\in J}b_{j!}b_j^*K\ra K$ is surjective.
\item[(ii)] For every $i\in I$, the object $\bigoplus_{j\in J}a_i^*b_{j!}b_j^*K$ is $(f\circ a_i)_*$-acyclic.

\end{itemize}

\label{basic_lemma}
\end{subthm}

The theorem follows from the proof of Lemma 3.3 of \cite{Be1}. 

Note also that, as the categories $\Perv_h$, $\Perv_m$ and $\Perv_{mf}$ are stable by
the functors $j^*$ and $j_!$ for $j$ an open affine embedding (by Corollary \ref{cor_f_*_preserves_Perv_mf} for $\Perv_{mf}$), if $K$ is
in one of these categories, so is $\bigoplus_{j\in J}b_{j!}b_j^*K$.

We now turn to the proof of Proposition~\ref{prop_direct_images}.

Following Section 3.4 of \cite{Be1}, we explain how to reconstruct the functor $f_*$
from ${}^0f_*$. More precisely, we want to apply Proposition~\ref{prop_der_fil} to
$\Df=\Dbm(X)$ and $\Df'=\Dbm(Y)$ with the perverse t-structures, $\DF=\DFbm(X,E)$ and $\DF'=\DFbm(Y,E)$
(see Section~\ref{def_mixed}) with the unique t-structures compatible with the perverse t-structures
(see Proposition~\ref{prop_t_DF}),
and $T=f_*$. 

To check condition~(a) of that proposition, remember that $\DFbm(X,E)$ is by definition a full subcategory of
the triangulated category $\DFbh(X,E)$ constructed in Section~\ref{section_def_hor}; we have
$\DFbh(X,E)=2-\varinjlim_{(A,\X)\in\Ob\Uf X}\DFbc(\X,E)$, and each $\DFbc(\X,E)$ is a full
subcategory of $\DFb(\X_\proet,E)$. Of course, we have similar statements for $\DFbm(Y,E)$.
If we choose $(A,\X)\in\Ob\Uf X$ and $(A,\Y)\in\Ob\Uf Y$ such that $f$ extends to an $A$-morphism
$f:\X\ra\Y$, then 
Proposition~\ref{prop_fil_der} and Section~\ref{f-horizontal}
give an f-lifting $f_{F,*}:\DFb(\X_\proet,E)\ra\DFb(\Y_\proet,E)$ of
$f_*:\Db(\X_\proet,E)\ra\Db(\Y_\proet,E)$. Taking the limit of these f-liftings, we get an f-lifting
$\DFbh(X,E)\ra\DFbh(Y,E)$ of $f_*:\Dbh(X,E)\ra\Dbh(Y,E)$, and it is easy to check that it sends
$\DFbm(X,E)$ to $\DFbm(Y,E)$.

We check condition~(b). Let $\If_m(f)$ be the full subcategory of $f_*$-acyclic objects in $\Perv_m(X)$.
Let $\Uf=(a_i:U_i\ra X)_{i\in I}$ be a finite covering of $X$ by open affine subschemes. We denote by
$\If_m(f,\Uf)$ the full subcategory of $\Perv_m(X)$ whose objects are mixed perverse sheaves that are
$(f\circ a_i)_*a_i^*$-acyclic for every $i\in I$; equivalently, $\If_m(f,\Uf)$ is the full subcategory of
$\bigoplus_{i\in I}(f\circ a_i)_* a_i^*$-acyclic objects. As all the functors $a_{i*}a_i^*$ are t-exact,
we have $\If_m(f)\subset\If_m(f,\Uf)$. By Remark~\ref{rmk_der_fil}(2) (applied to
the triangulated categories $\Kb(\Perv_m(X))\supset\Kb(\If_m(f,\Uf))\supset\Kb(\If_m(f))$), to check condition~(b) of
Proposition~\ref{prop_der_fil}, it suffices to prove the following two statements:
\begin{itemize}
\item[(1)] For every $A^\bullet\in\Ob\Cb(\Perv_m(X))$, there exists a quasi-isomorphism $B^\bullet\ra A^\bullet$ with
$B^\bullet$ in $\Cb(\If_m(f,\Uf))$.

\item[(2)] For every $B^\bullet\in\Ob\Cb(\If_m(f,\Uf))$, there exists a quasi-isomorphism $B^\bullet\ra C^\bullet$ with
$C^\bullet$ in $\Cb(\If_m(f))$.

\end{itemize}

We prove (1).
Let $A^\bullet\in\Ob\Cb(\Perv_m(X))$, and let $N\in\Nat$ such that $A^r=0$ for $r\not\in[-N,N]$.
We choose a finite family $(b_j:V_j\ra X)_{j\in J}$ of open affine subschemes of $X$ as in Theorem~\ref{basic_lemma},
for the fixed $f$, the family $(U_{i_0}\cap\ldots\cap U_{i_p})_{p\in\Nat,\ i_0,\ldots,i_p\in I}$ of open affine subschemes of $X$ and the perverse sheaf
$K=\bigoplus_{-N\leq r\leq N}A^r$. Let $A_0^\bullet\in\Ob\Cb(\Perv_m(X))$ be the complex obtained by applying the exact functor $\bigoplus_{j\in J}b_{j!}b_j^*$ to
$A^\bullet$. Then $A_0^\bullet$ is a complex of objects of $\If_m(f,\Uf)$ such that $A_0^r=0$ for $r\not\in[-N,N]$, and we have a morphism of complexes $A_0^\bullet\ra A^\bullet$
that is surjective in each degree. By iterating this procedure, we get an exact sequence $\ldots\ra A_2^\bullet\ra A_1^\bullet\ra A_0^\bullet\ra A^\bullet\ra 0$ in
$\Cb(\Perv_m(X))$ such that each $A_i^\bullet$ is a complex of objects of $\If_m(f,\Uf)$ concentrated in degrees $[-N,N]$. Moreover, by Lemma~\ref{lemma_der_fil}(v) (whose
hypothesis is satisfied thanks to 4.2.3 and 4.2.4 of~\cite{BBD}), we may without losing these properties assume that $A_i^\bullet=0$ for $i$ big enough.
We can take for $B^\bullet$ the total complex of the double complex $A_\bullet^\bullet$.

We prove (2). For every
$I'\subset I$, we denote by $a_{I'}:U_{I'}:=\bigcap_{i\in I'}U_i\fl X$ the inclusion; if $I'=\{i_0,\ldots,i_r\}$, we also write
$U_{I'}=U_{i_0,\ldots,i_r}$ and $a_{I'}=a_{i_0,\ldots,i_r}$.
As $X$ is separated, all the finite intersections of $U_i$'s are affine.
We first define the functor $\Cech_\Uf^\bullet:\Perv_m(X)\ra\Cb(\Perv_m(X))$ sending a 
mixed perverse sheaf to its {\v C}ech resolution. If $K\in\Ob\Perv_m(X)$, we set
\[\Cech_\Uf^r(K)=\bigoplus_{i_0,\ldots,i_r\in I}a_{i_0,\ldots,i_r*}a_{i_0,\ldots,i_r}^*K.\]
The differential $d^r:\Cech_\Uf^r(\F)\ra\Cech_\Uf^{r+1}(\F)$ is an alternating sum of adjunction morphisms. 
More precisely, for $i_0,\ldots,i_r\in I$, the restriction to $a_{i_0,\ldots,i_r*}a_{i_0,\ldots,i_r}^*K$ of $d^r$ 
is the sum over $i\in I$ and $0\leq s\leq r+1$ of $(-1)^s$ times the adjunction morphism $a_{i_0,\ldots,i_r*}a_{i_0,\ldots,i_r}^*K\ra 
a_{i_0,\ldots,i_{s-1},i,i_s,\ldots,i_r*}a_{i_0,\ldots,i_{s-1},i,i_s,\ldots,,i_r}^*K$
coming from the inclusion of $U_{i_0,\ldots,i_r,i}$ into $U_{i_0,\ldots,i_r}$.
This defines an exact functor $\Cech_\Uf^\bullet:\Perv_m(X)\ra\Cpl(\Perv_m(X))$. We also have a natural morphism
$K\ra\Cech_\Uf^\bullet(K)$, for every $K\in\Ob\Perv_m(K)$, given by the sum of the adjunction morphisms.

\label{alt_Cech}
We need a variant that takes its values in the category of bounded complexes,
called the \emph{alternating {\v C}ech complex} (see \cite[\href{https://stacks.math.columbia.edu/tag/01FG}{Section 01FG}]{stacks-project}).
Let $K\in\Ob\Perv_m(X)$.
For every $r\in\Nat$, we have an action of the symmetric group $\Sgoth_{r+1}$, seen as the group of permutations of $\{0,1,\ldots,r\}$,
on the perverse sheaf $\Cech^r_\Uf(K)$: if $\sigma\in\Sgoth_{r+1}$ and $i_0,\ldots,i_r\in I$,
then this action sends the component $a_{i_0,\ldots,i_r*}a_{i_0,\ldots,i_r}^*K$ to $a_{i_{\sigma^{-1}(0)},\ldots,i_{\sigma^{-1}(r)}*}a_{i_{\sigma^{-1}(0)},\ldots,i_{\sigma^{-1}(r)}}^*K=
a_{i_0,\ldots,i_r*}a_{i_0,\ldots,i_r}^*K$,
and it acts by $(-1)^\sigma\id$. We denote by $\Calt{r}{\Uf}(K)\subset\Cech_\Uf^r(K)$ the invariants of this action. 
This defines a subcomplex $\Calt{\bullet}{\Uf}(K)$ of $\Cech^\bullet(K)$. Also, as a perverse sheaf on $X$,
$\Calt{r}{\Uf}(K)$ is isomorphic to $\bigoplus_{I'\subset I,\ |I'|=r+1}a_{I'*}a_{I'}^*K$; in particular, we have $\Calt{1}{\Uf}(K)=\Cech^1_\Uf(K)$, so the morphism
$K\ra\Cech^\bullet_\Uf(K)$ factors through a morphism $K\ra\Calt{1}{\Uf}(K)$.
Another consequence of this observation is that the subfunctor $\Calt{\bullet}{\Uf}$ of $\Cech_\Uf^\bullet$ is
still exact, and that it takes its values in the full subcategory $\Cb(\Perv_m(X))$ of $\Cpl(\Perv_m(X))$. 

\begin{sublemma}
Let $K\in\Ob\Perv_m(X)$.
\begin{itemize}
\item[(i)] The morphism of complexes $K\ra\Cech^\bullet_\Uf(K)$ is a quasi-isomorphism.

\item[(ii)] The inclusion $\Calt{\bullet}{\Uf}(K)\ra\Cech^\bullet_\Uf(K)$ is a homotopy equivalence.

\item[(iii)] The morphism of complexes $K\ra\Calt{\bullet}{\Uf}(K)$ is a quasi-isomorphism.

\end{itemize}
\label{lemma_Cech}
\end{sublemma}

\begin{proof}
Point (iii) follows from (i) and (ii), and (ii) is proved as \cite[\href{https://stacks.math.columbia.edu/tag/01FM}{Lemma 01FM}]{stacks-project}.
To prove (i), we can restrict the complexes to every open subset of an open covering of $X$, for example the covering $\Uf$. So we may assume that
$X$ is equal to one of the $U_i$. In that case the complex $K\ra\Cech^\bullet_\Uf(K)$ (with $K$ in degree $-1$) is homotopy equivalent to $0$, see
for example the proof of \cite[\href{https://stacks.math.columbia.edu/tag/0G6S}{Lemma 0G6S}]{stacks-project}.
\end{proof}

Now let $B^\bullet\in\Ob\Cb(\If_m(f,\Uf))$. The double complex $\Calt{\bullet}{\Uf}(B^\bullet)$ is bounded, and we denote its total complex by $C^\bullet$. By the
definition of $\If_m(f,\Uf)$, this is a complex of $f_*$-acyclic objects, and by Lemma~\ref{lemma_Cech}(iii), the canonical morphism $B^\bullet\ra C^\bullet$ is a
quasi-isomorphism. This gives the conclusion of (2).

Proposition~\ref{prop_der_fil} now says that, if $T:\Db\Perv_m(X)\ra\Db\Perv_m(Y)$ is the composition of a quasi-inverse of the equivalence
$\Kb(\If_m(f))/\Nb(\If_m(f))$ (where, for every additive subcategory $\Cf$ of $\Perv(X)$, we denote by $\Nb(\Cf)$ the full category of exact complexes in
$\Kb(\Cf)$) and of the functor $\Kb(\If_m(f))/\Nb(\If_m(f))\flnom{\Kb({}^0 f_*)}\Db\Perv(Y)$, then $T$ is well-defined and we have $\real\circ T\simeq f_*\circ\real$.

Note also that the proofs of statements (1) and (2), and the construction and properties of the {\v C}ech complex and alternating {\v C}ech complex, would work just as well
for the categories $\Perv_{mf}$, because we know that, for $g$ a morphism of $k$-schemes, the functors $\Hp^k g_*$ preserve the subcategories
$\Perv_{mf}$ (Corollary~\ref{cor_f_*_preserves_Perv_mf}). Denote by $\If_{mf}(f)$ the full subcategory of $f_*$-acyclic objects in $\Perv_{mf}(X)$. We get in particular
that the obvious functor $\Kb(\If_{mf}(f))/\Nb(\If_{mf}(f))\ra\Db\Perv_{mf}(X)$ is an equivalence of categories and that $\Kb({}^0f_*)$ sends
objects of $\Nb(\If_{mf}(X))$ to exact complexes, so we can make the following definition.

\begin{subdef} Let $f:X\fl Y$. We define the functor
\[H_{mf,*}(f):\Db\Perv_{mf}(X)\fl\Db_{mf}\Perv(Y)\]
to be the composition
\[\Db\Perv_{mf}(X)\ra\Kb(\If_{mf}(f))/\Nb(\If_{mf}(f))\flnom{\Kb({}^0f_*)}\Db\Perv_{mf}(Y),\]
%\[H_{mf,*}(f)=\varprojlim_{\Uf\in\Cov(X)}\varepsilon_{*}f_{\Uf\bullet*}\varepsilon_\Uf^*.\]
where the first functor is right adjoint to the obvious functor $\Kb(\If_{mf}(f))/\Nb(\If_{mf}(f))\ra\Db\Perv_{mf}(X)$.
\footnote{If a functor $F$ is an equivalence of categories, then it admits a right adjoint, and any right adjoint is a quasi-inverse. We use
a right adjoint instead of an arbitrary quasi-inverse, because a right adjoint is unique up to unique isomorphism.}

\label{def_H_mf}
\end{subdef}

It remains to show that Definition \ref{def_H_mf}
does give a $2$-functor from $\Sch/k$ to $\TR$, that is, to construct connection morphisms.
So suppose that we have two morphisms of $k$-schemes $f:X\fl Y$ and $g:Y\fl Z$. We want
to lift the connection isomorphism $g_*\circ f_*\simeq(g\circ f)_*:\Dbm(X)\fl\Dbm(Z)$ to a
natural isomorphism $H_{mf,*}(g)\circ H_{mf,*}(f)\iso H_{mf,*}(g\circ f)$.
Let us write $\If_1=\If_{mf}(f)$, $\If_2=\If_{mf}(g\circ f)$, $\If_3=\If_1\cap\If_2=\If_{mf}(f\times(g\circ f))$ and $\Jf=\If_{mf}(g)$.
We consider the following diagram:
\[\xymatrix{\Kb(\If_2)/\Nb(\If_2)\ar@/_5pc/[ddd]_{\Phi_2}\ar@/^3pc/[rrddd]^-{\Kb({}^0(g\circ f)_*)} &{}\ar@{}[dd]|(.6){(*)} & \\
\Kb(\If_3)/\Nb(\If_3)\ar[u]^{\Phi_{32}}\ar[d]_{\Phi_{31}}\ar[rd]^-{\Kb({}^0 f_*)} & & \\
\Kb(\If_1)/\Nb(\If_1)\ar[d]_{\Phi_1}\ar[rd]_-{\Kb({}^0 f_*)} & \Kb(\Jf)/\Nb(\Jf)\ar[d]_\Phi\ar[rd]^-{\Kb({}^0 g_*)} & \\
\Db\Perv_{mf}(X) & \Db\Perv_{mf}(Y) & \Db\Perv_{mf}(Z)}\]
All the vertical maps in it are equivalences, and the diagram commutes, except for triangle (*) that only commutes up to a natural isomorphism
coming from the connection isomorphism $(g\circ f)_*\simeq g_*\circ f_*$.

We also set $\Phi_3=\Phi_1\circ\Phi_{31}=\Phi_2\circ\Phi_{32}$. For every arrow $\Phi_?$, we denote by $\Psi_?$ a right adjoint of $\Phi_?$.
Then we have natural transformations
\begin{align*}
H_{mf,*}(g\circ f)=\Kb({}^0(g\circ f)_*)\circ\Psi_2 & \leftarrow \Kb({}^0(g\circ f)_*)\circ\Phi_{32}\circ\Psi_{32}\circ\Psi_2\\
& \simeq\Kb({}^0(g\circ f)_*)\circ\Phi_{32}\circ\Psi_3\\
& \stackrel{(*)}{\simeq} \Kb({}^0g_*)\circ\Kb({}^0f_*)\circ\Psi_3\\
& \rightarrow\Kb({}^0g_*)\circ\Psi\circ\Phi\circ\Kb({}^0f_*)\circ\Psi_3\\
& = H_{mf,*}(g)\circ\Kb({}^0 f_*)\circ\Phi_{31}\circ\Psi_3\\
& \simeq H_{mf,*}(g)\circ\Kb({}^0 f_*)\circ\Phi_{31}\circ\Psi_{31}\circ\Psi_1\\
& \rightarrow H_{mf,*}(g)\circ\Kb({}^0 f_*)\circ\Psi_1\\
& = H_{mf,*}(g)\circ H_{mf,*}(f).
\end{align*}
Moreover, all these natural transformations are isomorphisms: the ones marked $\leftarrow$ and $\rightarrow$ are units or counits of
adjunctions between equivalences of categories, isomorphism (*) comes from the connection isomorphism $(g\circ f)_*\simeq g_*\circ f_*$, and
the two remaining isomorphisms come from the uniqueness of right adjoints.
This gives the desired connection isomorphism, and we check the cocycle condition in a similar way.
This finishes the proof of Proposition~\ref{prop_direct_images}.

\begin{subrmk}
Suppose that $f:X\ra Y$ is an affine morphism of $k$-schemes. By Theorem~\ref{basic_lemma} (and Corollary~\ref{cor_f_*_preserves_Perv_mf})),
the category $\If_{mf}(f_*)$ of $f_*$-acyclic objects in $\Perv_{mf}(X)$ is cogenerating. So, by Remark~\ref{rmk_der_fil}(1), the functor
${}^0f_*:\Perv_{mf}(X)\ra\Perv_{mf}(Y)$ has a left derived functor, and the functor
$H_{mf,*}(f)$ is the restriction of this derived functor to $\Db\Perv_{mf}(X)$. In particular, if $f$ is quasi-finite and affine,
then $f_*:\Perv_{mf}(X)\ra\Perv_{mf}(Y)$ is exact, and $H_{mf,*}(f)$ is its obvious extension to the bounded derived categories.

\label{rmk_direct_images}
\end{subrmk}

From now on, for $f:X\ra Y$ a morphism of $k$-schemes,
we will also denote the functor $H_{mf,*}(f)$ by $f_*$ if there is no risk of confusion.

\begin{subprop} The functor $\boxtimes$ from Proposition \ref{prop_obvious2}
induces a natural isomorphism between the $2$-functors $H_{mf,*}\times H_{mf,*}:
\Sch_k\times\Sch_k\fl\TR$ and $\Sch_k\times\Sch_k\flnom{\times}
\Sch_k \flnom{H_{mf,*}}\TR$ (where the first arrow sends $(X,Y)$ to $X\times Y$).

In other words, if $f_1:X_1\fl Y_1$ and $f_2:X_2\fl Y_2$ are morphisms
in $\Sch/k$ and $K_1\in D^b\Perv_{mf}(X_1)$, $K_2\in D^b\Perv_{mf}(X_2)$, then
we have an isomorphism
\[(f_1\times f_2)_*(K_1\boxtimes K_2)\iso (f_{1*}K_1)\boxtimes(f_{2*}K_2)\]
functorial in $K_1$ and $K_2$ and compatible with the composition of
arrows in $\Sch/k$.

\label{prop_boxtimes_direct_images}
\end{subprop}

\begin{proof} On the categories $\Dbc$, we have canonical isomorphisms
(see SGA 5 III 1.6). These induce similar isomorphisms in the categories $\Dbh$ and
$\Dbm$.

By the construction of $f_*$ (see Definition \ref{def_H_mf}),
we only need to show the statement of the proposition for
the functors ${}^0 f_*$ between the categories $\Perv_{mf}(X)$.
But then it is an immediate consequence of the similar result for the
categories $\Perv(X)$, which follows from the result recalled at the beginning of the proof and from
the t-exactness of the external tensor
product (see Proposition \ref{prop_exactness}).
\end{proof}

\begin{subprop} Let $j:U\fl X$ be an open embedding. 
Denote by $j^*:D^b\Perv_{mf}(X)\fl D^b\Perv_{mf}(U)$ the derived functor
of the exact functor $j^*:\Perv_{mf}(X)\fl\Perv_{mf}(U)$.

Then this functor $j^*$ is left adjoint to the functor
$j_*:D^b\Perv_{mf}(U)\fl D^b\Perv_{mf}(X)$, and
the counit map
$j^*j_*\fl\id$ is an isomorphism.

\label{prop_adjunction_open}
\end{subprop}

\begin{proof} 
Let $\Uf=(U_i)_{i\in I}$ be a finite open affine cover of $U$. For every
$J\subset I$, we denote by $a_J:\bigcap_{i\in J}U_i\fl U$ the inclusion.
As $U$ is separated, all the finite intersections of $U_i$'s are affine,
so the morphisms $a_J$ and $j\circ a_J$ are affine for every $J\subset I$.
If $K\in\Ob\Perv_{mf}(U)$, we denote by $\Calt{\bullet}{\Uf}(K)$ the alternating {\v C}ech complex
of $K$ associated to the covering $(U_i)_{i\in I}$ (defined on page \ref{alt_Cech}), so that
$\Calt{r}{\Uf}(K)=\bigoplus_{|J|=r+1}a_{J*}a_J^*K$.
%and the maps $C^r(K)\fl C^{r+1}(K)$ are alternating sums of adjunction morphisms. 
The canonical morphism
$K\fl\Calt{\Uf}{\bullet}(K)$ is a quasi-isomorphism (Lemma~\ref{lemma_Cech}(iii)), and all the $\Calt{r}{\Uf}(K)$ are
$j_*$-acyclic (indeed, as $j\circ a_J$ is affine for every $J\subset I$, the complex
$j_*(a_{J*}a_J^*K)$ is perverse, and so $j_* \Calt{r}{\Uf}(K)$ is perverse).

As the alternating {\v C}ech complex is a functor $\Perv_{mf}(U)\ra\Cb(\Perv_{mf}(U))$,
we have by Definition \ref{def_H_mf} an isomorphism of functors between $j_*$ and the functor
sending $K$ to the the total complex of the double complex
$j_*\Calt{\bullet}{\Uf}(K)$. So $j^* j_*$ is isomorphic to the functor sending $K$
to the total complex of $\Calt{\bullet}{\Uf}(K)$. Let
$\alpha:\id\ra j^*j_*$ the map corresponding to the natural transformation
$K\ra\Tot\Calt{\bullet}{\Uf}(K)$. This is an isomorphism of endofunctors
of $\Db\Perv_{mf}(U)$ by Lemma~\ref{lemma_Cech}(iii)), and we set
$\varepsilon=\alpha^{-1}$; this is the counit of the adjunction.

Now we construct the unit $\eta:\id\ra j_*j^*$.
As $j^*$ is exact, the functor $j^*j_*$ is isomorphic to the functor sending
$L\in\Ob\Db\Perv_{mf}(X)$ to the total complex of the double complex $j_*\Calt{\bullet}{\Uf}(j^*L)$.
If $L$ is a finite complex of objects of $\Perv_{mf}(X)$, then $j_*\Calt{0}{\Uf}(j^*L)=\bigoplus_{i\in I}(j\circ a_i)_*(j\circ a_i)^*L$;
taking the sum of the unit morphisms of the adjunctions $((j\circ a_i)^*,(j\circ a_i)_*)$,
we get a morphism $L\ra j_*\Calt{0}{\Uf}(j^*L)$ in $\Cb(\Perv_m(X))$, hence in its full subcategory $\Cb(\Perv_{mf}(X))$.
Also, it is easy to see that the composition $L\ra j_*\Calt{0}{\Uf}(j^*L)\ra j_*\Calt{1}{\Uf}(j^*L)$ is $0$ in
$\Cb(\Perv_{mf}(X))$. So we get a morphism from $L$ to the total complex of the double complex $j_*\Calt{\bullet}{\Uf}(j^*L)$, which
induces $\eta:\id\ra j_*j^*$.

%Moreover, by Corollary
%\ref{cor_f_*_preserves_Perv_mf}, if $K$ is in $\Perv_{mf}(U)$, then
%$C^\bullet(K)$ is a complex of objects of $\Perv_{mf}(U)$, and
%$j_*C^\bullet(K)$ is a complex of objects of $\Perv_{mf}(X)$, which is
%quasi-isomorphic to $j_* K$ by definition of the functor $j_*$. Note
%also that this construction is functorial in $K$.

%Now we want to define a unit map $\varepsilon:\id\fl j_*j^*$ and a counit
%map $j^*j_*\fl\id$. Let $K$ be a finite complex of objects of $\Perv_{mf}(U)$.
%As the construction of the alternating {\v C}ech complex is functorial, we have by Definition \ref{def_H_mf}
%an isomorphism between $j_*K$ and the total complex of the double complex
%$j_*\Calt{\bullet}{\Uf}(K)$, so
%$j^* j_* K$ is isomorphic to the total complex of $\Calt{\bullet}{\Uf}(K)$, and we can take
%for $\eta$ the inverse of the canonical
%quasi-isomorphism $K\fl C^\bullet(K)$. (Note in particular that $\eta$ is
%an isomorphism, which gives the last statement of the proposition.)
%If $L$ is a complex of objects
%of $\Perv_{mf}(X)$, then $j_*j^*L$ is quasi-isomorphic to the total complex
%of the double complex $j_* C^\bullet(j^* L)$. But we have a canonical
%morphism $L\fl j_* C^0(j^*L)$ (because it exists in $C^b\Perv_m(X)$, and
%$C^b\Perv_{mf}(X)$ is a full subcategory of $C^b\Perv_m(X)$), and it
%is easy to see that this induces a morphism $L\fl j_* C^\bullet(j^* L)$, which
%is the desired morphism $\varepsilon$. 

To finish the proof, it suffices to
show that, for $K\in\Ob D^b\Perv_{mf}(U)$ and $L\in\Ob D^b\Perv_{mf}(X)$,
the composition
\[j_* K\flnom{\eta j_*}j_*j^* j_* K\flnom{j_*\varepsilon}j_*K\]
is the identity of $j_*K$ and the composition
\[j^* L\flnom{j^*\eta}j^*j_*j^*L\flnom{\varepsilon j^*}j^*L\]
is the identity of $j^*L$.

%The first statement is clear from the explicit descriptions
%of $j_*$, $\varepsilon$ and $\eta$, and the second statement follows from
%the conservativity of the functor $R_U$.
Let $K\in\Cb(\Perv_{mf}(U))$. We denote by $s$ the canonical morphism from $K$ to the total complex of $\Calt{\bullet}{\Uf}(K)$. 
Using the canonical isomorphism $j^*\Tot j_*\Calt{\bullet}{\Uf}(K)=\Tot\Calt{\bullet}{\Uf}(K)$, we
get a commutative diagram
\[\xymatrix@C=70pt{j_* K\ar[r]^-{\eta j_*}\ar[d]^\wr & j_*j^*j_* K \ar[r]^-{j_*\varepsilon}\ar[d]^\wr& j_* K\ar[d]^\wr \\
\Tot j_*\Calt{\bullet}{\Uf}(K)\ar[r]_-{\Tot j_*\Calt{\bullet}{\Uf}(s)} & \Tot j_*\Calt{\bullet}{\Uf}(\Tot\Calt{\bullet}{\Uf}(K)) &
\Tot j_*\Calt{\bullet}{\Uf}(K)\ar[l]^-{\Tot j_*\Calt{\bullet}{\Uf}(s)}
}\]
This gives the first statement.

Let $L\in\Cb(\Perv_{mf}(X))$. We denote by $s$ the canonical morphism $j^*L\ra \Calt{\bullet}{\Uf}(j^*L)$, and by
$s'$ the morphism from $L$ to $\Tot j_*\Calt{\bullet}{\Uf}(j^*L)$ from the construction of $\eta$. We have a canonical isomorphism
$j^*\Tot j_*\Calt{\bullet}{\Uf}(j^*L)=\Tot\Calt{\bullet}{\Uf}(j^*L)$, and this identifies $j^*s'$ and $s$. So we get a commutative
diagram
\[\xymatrix{j^*L \ar[r]^-{j^*\eta}\ar[rd]_-s & j^*j_*j^* L \ar[r]^-{\varepsilon j^*}\ar[d]^\wr & j^* L\ar[ld]^-s \\
& \Tot\Calt{\bullet}{\Uf}(j^*L) &
}\]
This gives the second statement.
\end{proof}

\subsection{Inverse images}
\label{inverse_images}

In this section, we construct the inverse images functors as the left
adjoints of the direct image functors of Proposition \ref{prop_direct_images}.

First we treat a particular case. For every smooth equidimensional
$k$-scheme $X$, we denote by $\ungras_X$ the constant sheaf on $X$,
seen as an object of $\Perv_{mf}(X)[-\dim X]$.

\begin{subprop} Let $X,Y\in\Ob(\Sch/k)$, and suppose that $X$ is smooth
equidimensional. 
Let $p:X\times Y\fl Y$ be the second projection. 

Then the functor $p_*:D^b\Perv_{mf}(X\times Y)\fl D^b\Perv_{mf}(Y)$ 
admits a left adjoint $p^*$, which is given
by $K\fle\ungras_X\boxtimes K$.

\label{prop_inverse_image_projection}
\end{subprop}

In particular, we get a natural isomorphism $\theta_p:p^*\circ R_Y\iso
R_{X\times Y}\circ p^*$.

\begin{proof} Let $p^*$ be as in the statement.
It suffices to construct
natural morphisms $\eta:\id\fl p_*p^*$
and $\varepsilon:p^* p_*\fl\id$ whose images by $R$ are the unit and counit of
the adjonction in the categories $D^b\Perv_m$, and such that
$p_*\flnom{\eta p_*}p_*p^*p_*\flnom{p_*\varepsilon}p_*$ is the identity.
(As $R_{X\times Y}$ is conservative, we'll automatically get the fact
that $p^*\flnom{\varepsilon p^*}p^*p_*p^*\flnom{p^*\eta}p^*$ is an
isomorphism.)

Let $a_X:X\fl\Spec k$ be the structural map.
Note that, as $a_{X*}\ungras_X\in\Dpos\Perv_{mf}(\Spec k)$, we have
\[\begin{array}{rcl}
\Hom_{\Db\Perv_{mf}(\Spec k)}(\ungras_{\Spec k},a_{X*}\ungras_X)&=&
\Hom_{\Perv_{mf}(\Spec k)}(\ungras_{\Spec k},\H^0 a_{X*}\ungras_X)\\
&=&
\Hom_{\Perv(\Spec k)}(E,\H^0 a_{X*}\underline{E}_X).
\end{array}\]
So the canonical morphism
$E\fl\H^0 a_{X*}\underline{E}_X$ (coming from the unit of the
adjunction $(a_X^*,a_{X*})$ gives a morphism $u_X:\ungras_{\Spec k}\fl
a_{X*}\ungras_X$ in $\Db\Perv_{mf}(\Spec k)$.

If $K\in\Ob(\Db\Perv_{mf}(Y))$, then we have a morphism
\[K=\ungras_{\Spec k}\boxtimes K\flnom{u_X\boxtimes\id}(a_{X*}\ungras_X)\boxtimes K
\simeq p_*(\ungras_X\boxtimes K)=p_*p^* K,\]
where the third arrow is the isomorphism of Proposition
\ref{prop_boxtimes_direct_images}. This morphism is an isomorphism because
its image by $R_Y$ is an isomorphism, and we denote it
by $\eta$.

Now we want to construct $\varepsilon$. Consider the commutative diagram
\[\xymatrix{X\times Y\ar[d]_p & X\times X\times Y\ar[l]_-{q_2}
\ar[d]^{q_1} & X\times Y\ar[l]_-{i} \\
Y & X\times Y\ar[l]^-p &}\]
where $q_1=\id_X\times p$, $q_2=a_X\times\id_{X\times Y}$ and $i$ is the
product of the diagonal embedding of $X$ and of $\id_Y$. Note that
$q_1 i=q_2 i=p$.
Using Proposition \ref{prop_boxtimes_direct_images}, we get an isomorphism
\[p^*p_*K=\ungras_X\boxtimes(p_* K)\simeq q_{1*}(\ungras_X\boxtimes K)=
q_{1*}q_2^*K.\]
As $i$ is a closed immersion, we know (by Corollary \ref{cor_i^*}) that
the functor $i_*$ has a left adjoint $i^*$. This and the functoriality of
$H_{mf,*}$ gives a morphism
\[q_{1*}q_2^*K\fl q_{1*}i_*i^*q_2^*K\simeq q_{2*}i_*i^*q_2^*K.\]
Note also that using 
the unit of $(i^*,i_*)$ and the analogue of the natural transformation
$\eta$ for
$q_2$ instead of $p$, we get a morphism
\[K\iso q_{2*}q_2^*K\fl q_{2*}i_*i^*q_2^*K,\]
which is an isomorphism because its image by $R_{X\times Y}$ is an isomorphism.
Putting all these together gives
\[\varepsilon:p^*p_*K\iso q_{1*}q_2^*K\fl q_{2*}i_*i^*q_2^*K\simeq K.\]

It is clear from the construction
that the images of $\eta$ and $\varepsilon$ by $R$ are the unit
and the counit of the adjonction $(p^*,p_*)$ in $\Dbm$.
So we just need to show that $p_*\flnom{\eta p_*}p_*p^*p_*
\flnom{p_*\varepsilon}p_*$ is the identity. 
This follows from the fact that
we get this composition by following the
outside of the commutative diagram below in the clockwise direction (where the
two arrows marked ``$\adj$'' come from the unit of the adjunction
$(i^*,i_*)$) :
\[\xymatrix{&&& p_*(\ungras_X\boxtimes(p_*K))\ar@{-}[d]^\wr\ar@{=}[r]  & 
p_*p^*p_*K\\
&\ungras_{\Spec k}\boxtimes (p_* K)\ar[r]^-{u_X}
 & (a_{X*}\ungras_X)\boxtimes(p_*K)\ar@{-}[r]^-\sim &
p_*q_{1*}(\ungras_X\boxtimes K)\ar@{=}[r]\ar@{-}[ddl]_-\sim
 & p_*q_{1*}q_2^*K\ar[d]^\adj \\
p_*K\ar@{-}[ur]^-\sim\ar@{-}[r]_-\sim &
p_*(\ungras_{\Spec k}\boxtimes K)\ar[d]_{u_X}
&&&p_*q_{1*}i_*i^*q_2^* K\ar[d]_\wr \\
&p_*((a_{X*}\ungras_X)\boxtimes K)\ar[r]^-\sim & 
p_*q_{2*}(\ungras_X\boxtimes K)\ar@{=}[r] &
p_*q_{2*}q_2^*K\ar[r]_-\adj &
p_*q_{2*}i_*i^*q_2^*K}\]
\end{proof}

Having at our disposal
the constant sheaf on $X$ was very important when constructing
the inverse image of the second projection $X\times Y\fl Y$.
Now, in order to generalize this construction to the case when $X$
is not necessarily smooth,
we want to construct (and characterize) the analogue
in $\Db\Perv_{mf}(X)$ of the constant sheaf $\underline{E}_X$.
Note that this is not totally obvious
in this context because, if $X$ is not smooth,
then the constant sheaf is not perverse (or shifted perverse) in general.

For every $k$-scheme $X$, we denote by $a_X:X\fl\Spec k$ the structural
morphism. We also denote by $\ungras_{\Spec k}$ the constant sheaf with
value $E$ on $\Spec k$, seen as an object of $\Perv_{mf}(\Spec k)$.

\begin{subcor} For every $k$-scheme $X$, the functor $\Db\Perv_{mf}(X)\fl
\Sets$ (where $\Sets$ is the category of sets),
$K\fle\Hom_{\Db\Perv_{mf}(\Spec k)}(\ungras_{\Spec k},a_{X*} K)$, is representable.

Moreover, if $(\ungras_X,u_X:\ungras_{\Spec k}\fl
a_{X*}\ungras_X)$ represents this functor, then there is an isomorphism
$R_X(\ungras_X)\simeq\underline{E}_X$ that makes the following
diagram commute :
\[\xymatrix{R_{\Spec k}(\ungras_{\Spec k})\ar[d]_{R_{\Spec k}(u_X)}\ar@{=}[r] & 
\underline{E}_{\Spec k}\ar[r]^-{\adj}&
a_{X*}\underline{E}_X\ar[d]^\wr & \\
R_{\Spec k}(a_{X*}\ungras_X)\ar[rr]_-{\gamma_{a_X}} && a_{X*}R_X(\ungras_X)}\]
where the arrow marked ``$\adj$'' is the unit of the adjunction $(a_X^*,a_{X*})$.

\label{cor_constant_sheaf}
\end{subcor}

Note that the couple $(\ungras_X,u_X)$ is unique up
to unique isomorphism if it exists.

\begin{proof} First note that, thanks to Corollary
\ref{cor_i^*} and Proposition
\ref{prop_adjunction_open}, if $h:Z\fl X$ is an open embedding or
a closed embedding and
the result is true for $X$, then it is also true for $Z$, and moreover we
have a canonical isomorphism $\ungras_Z\simeq h^*\ungras_X$.
Moreover, if $X$ is smooth, then the result follows immediately
from Proposition \ref{prop_inverse_image_projection}. In particular,
we get the result for $X$ affine, because in that case $X$ is a closed
subscheme of some $\Aff^n$.

For a general $k$-scheme $X$, we chose a finite open cover $X=\bigcup_{i=1}^n
U_i$ such that the result is known for every $U_i$. (For example, we can
take a finite affine open cover.)
We want to show that this implies the result for
$X$. We reduce to the case $n=2$ by an easy induction on $n$. Let $j_1:U_1\fl
X$, $j_2:U_2\fl X$ and $j_{12}:U_1\cap U_2\fl X$ be the inclusions. By the
uniqueness statement of the corollary,
we have canonical isomorphisms $\ungras_{U_i\vert U_1
\cap U_2}\simeq\ungras_{U_1\cap U_2}$ for $i=1,2$ that identify $u_{U_i}$ and
$u_{U_1\cap U_2}$,
so, using Proposition
\ref{prop_adjunction_open}, we get morphisms $v_i:j_{i*}\ungras_{U_i}\fl
j_{12,*}\ungras_{U_1\cap U_2}$, $i=1,2$. Complete $v:=v_1\oplus(-v_2)$ into
an exact triangle
\[(*)\qquad K\fl j_{1*}\ungras_{U_1}\oplus j_{2*}\ungras_{U_2}\flnom{v}j_{12*}\ungras_{U_1\cap
U_2}\flnom{+1}\]
Applying $a_{X*}$, we get a triangle
\[(**)\qquad a_{X*}K\fl a_{U_1,*}\ungras_{U_1}\oplus a_{U_2,*}\ungras_{U_2}\flnom{a_{X*}v} 
a_{U_1\cap U_2,*}\ungras_{U_1\cap U_2}\flnom{+1}\]
Consider the morphism $u_{U_1}\oplus u_{U_2}:\ungras_{\Spec k}\fl
a_{U_1,*}\ungras_{U_1}\oplus a_{U_2,*}\ungras_{U_2}$. Composing it by $a_{X*}v$ gives
$0$, by definition of $v$, so it comes from a map $u_X:\ungras_{\Spec k}\fl
a_{X*}K$. Also, as $a_{U_1\cap U_2,*}\ungras_{U_1\cap U_2}$ is concentrated in degree
$\geq 0$, we have $\Hom_{D^b\Perv_{mf}(\Spec k)}(\ungras_{\Spec k},a_{U_1\cap U_2,*}
\ungras_{U_1\cap U_2}[-1])=0$, and so the map $u_X$ is uniquely determined.

Now we show that $(K,u_X)$ represents the functor of the statement.
For every $L\in\Ob (D^b\Perv_{mf}(X))$, the map $u_X:\ungras_{\Spec k}\fl a_{X*}K$
induces a morphism
\[\Hom_{D^b\Perv_{mf}(X)}(K,L)\rightarrow\Hom_{D^b\Perv_{mf}(\Spec k)}(
a_{X*}K,a_{X*}L)\rightarrow\Hom_{D^b\Perv_{mf}(\Spec k)}(\ungras_{\Spec k},a_{X*}L),\]
and we must
show that this is an isomorphism. Suppose that we can prove this
if one of the adjunction maps
$L\fl j_{1*}j_1^*L$, $L\fl j_{2*}j_2^*L$ or $L\fl j_{12*}j_{12}^*L$ is an
isomorphism, then we are done. Indeed,
for a general $L$, we have an exact triangle
$L\fl j_{1*}j_1^*L\oplus j_{2*}j_2^*L\fl j_{12*}j_{12}^*L\flnom{+1}$, and we use
the five lemma.

Suppose that the adjunction map $L\fl j_{1*}j_1^*L$ is an isomorphism. 
Applying $j_1^*$ to the triangle $(*)$ and noting that $j_1^*j_{2*}\ungras_{U_2}\fl
j_1^*j_{12*}\ungras_{U_1\cap U_2}$ is an isomorphism, we get an isomorphism
$j_1^*K\iso\ungras_{U_1}$.
We denote by $c$ the
base change morphism $a_{X*}\fl a_{U_1*}j_1^*$. Applying $c$ to the entries of the triangle
$(**)$, we get a commutative diagram
\[\xymatrix{a_{U_1*}j_1^*K\ar[r] &  a_{U_1,*}j_1^*j_{1*}\ungras_{U_1}\oplus a_{U_1,*}j_1^*j_{2*}
\ungras_{U_2} \ar[r]&
a_{U_1*}j_1^* j_{12*}\ungras_{U_1\cap U_2}\ar[r]^-{+1} & \\
a_{X*}K\ar[r]\ar[u] &  a_{U_1,*}\ungras_{U_1}\oplus a_{U_2,*}\ungras_{U_2}\ar[r]\ar[u] &  
a_{U_1\cap U_2,*}\ungras_{U_1\cap U_2}\ar[r]^-{+1}\ar[u] &
}\]
The morphism $a_{U_1,*}j_1^*j_{2*}
\ungras_{U_2} \fl
a_{U_1*}j_1^* j_{12*}\ungras_{U_1\cap U_2}$ in the first row of this diagram is an isomorphism,
so we get an isomorphism
$a_{U_1*}j_1^*K\fl  a_{U_1,*}j_1^*j_{1*}\ungras_{U_1}\simeq a_{U_1*}\ungras_{U_1}$ (which is
just the image by $a_{U_1*}$ of the isomorphism $j_1^*K\iso\ungras_{U_1}$ of the beginning of
this paragraph.
By this isomorphism, the map
$c_K\circ u_X:\ungras_{\Spec k}\fl a_{U_1*}j_1^*K$ corresponds to the composition
of
\[(c_{j_{1*}\ungras_{U_1}}\oplus c_{j_{2*}\ungras_{U_2}})\circ (u_{U_1}\oplus u_{U_2}):
\ungras_{\Spec k}\fl 
a_{U_1,*}j_1^*j_{1*}\ungras_{U_1}\oplus a_{U_1,*}j_1^*j_{2*}
\ungras_{U_2} \]
and of the first projection. In other words, we get a commutative diagram :
\[\xymatrix{\ungras_{\Spec k}\ar[r]^-{u_X}\ar[d]_{u_{U_1}} & a_{X*}K\ar[d]^{c_K} \\
a_{U_1*}\ungras_{U_1} & a_{U_1*}j_1^*K\ar[l]^-\sim}\]
Consider the following diagram (where all the $\Hom$ groups are taken in the appropriate
$\Db\Perv_{mf}$ category) :
\[\xymatrix{
\Hom(K,j_{1*}j_1^*L) \ar[r]^-{a_{X*}}\ar[d]^{j_1^*}_*[@]{\sim}
& \Hom(a_{X*}K,a_{X*}j_{1*}j_1^*L)\ar[r]^-{(-)\circ u_X}
 & \Hom(\ungras_{\Spec k},
a_{U_1*}j_1^*L)\ar@{=}[dd] \\
\Hom(j_1^*,j_1^*L)\ar[r]^-{a_{U_1*}} 
& \Hom(a_{U_1*}j_1^*K,a_{U_1*}j_1^*L)\ar[u]_{(-)\circ c_K} & \\
\Hom(\ungras_{U_1},j_1^*L)\ar[r]_-{a_{U_1*}}\ar[u]^*[@]{\sim}
& \Hom(a_{U_1*}\ungras_{U_1},a_{U_1*}j_1^*L)\ar[u]^*[@]{\sim}\ar[r]_-{(-)\circ u_{U_1}} &
\Hom(\ungras_{\Spec k},a_{U_1*}j_1^*L)
}\]
We have just seen that the right rectangle of this diagram is commutative. It is also easy to see
that the two squares on the left are commutative, so the whole diagram commutes. As
the composition of the two bottom horizontal
arrows is an isomorphism by assumption, the composition of the
two top horizontal arrows is also an isomorphism, which is what we wanted to prove.

The case where $L\fl j_{2*}j_2^*L$ (resp. $L\fl j_{12*}j_{12}^*L$) is an
isomorphism is similar. This finishes the proof of the first statement of the corollary.
The second statement of the corollary follows easily from the explicit definition of $u_X$.
\end{proof}

Now that we have the object $\ungras_X$, the proof of the following
corollary is exactly the same as the proof of Proposition
\ref{prop_inverse_image_projection}.

\begin{subcor} Let $X,Y\in\Ob(\Sch/k)$, and let                                 
$p:X\times Y\fl Y$ be the second projection. 

Then the functor 
$p_*:\Db\Perv_{mf}(X\times Y)\fl\Db\Perv_{mf}(Y)$ 
admits a left adjoint $p^*$, which is given
by $K\fle\ungras_X\boxtimes K$.

\label{cor_inverse_image_projection}
\end{subcor}

\begin{subcor} The $2$-functor $H_{mf,*}:\Sch/k\fl\TR$ of Proposition
\ref{prop_direct_images} admits a global left adjoint in the sense
of Definition 1.1.18 of \cite{Ay1}. 

In particular, we get a uniquely determined
$2$-functor $H_{mf}^*:\Sch/k\fl\TR$ such, for every morphism
$f:X\fl Y$ in $\Sch/k$, the functor 
$H_{mf}^*(f):\Db\Perv_{mf}(Y)\fl
\Db\Perv_{mf}(X)$ is a left adjoint of 
$H_{mf,*}(f):\Db\Perv_{mf}(X)\fl
\Db\Perv_{mf}(Y)$.

Moreover, for every morphism of $k$-schemes $f:X\fl Y$, we have an
invertible natural transformation $\theta_f:H_m^*(f)\circ R_Y\iso R_X\circ
H_{mf}^*(f)$, and this is compatible with the composition of morphisms
in $\Sch/k$.

\label{cor_inverse_images}
\end{subcor}

\begin{proof} 
By Proposition 1.1.17 of \cite{Ay1}, to show the first statement,
it suffices to show
that, for every $f:X\fl Y$ in $\Sch/k$, the functor
$H_{mf,*}(f):\Db\Perv_{mf}(X)\fl
\Db\Perv_{mf}(Y)$ admits a left adjoint. 
We factor $f$ as $X\flnom{i} X\times Y\flnom{p}Y$,
where $i=\id_X\times f$ and $p$ is the second projection.
The first map is a closed embedding, so it admits a left adjoint by
Corollary \ref{cor_i^*}, and the second map admits a left adjoint by
Corollary \ref{cor_inverse_image_projection}.
The natural transformation $\theta_i$ and $\theta_p$ are also constructed in
these corollaries, and we take $\theta_f$ equal to :
\[R_X\circ f^*=R_X\circ i^*p^*\flnom{\theta_i}i^*\circ R_{X\times Y}\circ
p^*\flnom{\theta_p}i^*p^*\circ R_Y\simeq f^*\circ R_Y.\]
By a slight abuse of notation, we will write that $\theta_f=\theta_p\circ
\theta_i$.

Suppose that we are given a second morphism $g:Y\fl Z$, and that we are
trying to prove the compatibility between $\theta_f$, $\theta_g$ and
$\theta_{gf}$. Consider the commutative diagram :
\[\xymatrix{&&& X\times Z\ar[dddl]^-{p'''} \\
X\ar[rrru]^-{i'''}\ar[r]_-{i}\ar[dr]_-{f} & X\times Y\ar[r]_-{i''}\ar[d]^p
 & X\times Y\times Z\ar[ru]_-q\ar[d]_{p''} & \\
& Y\ar[r]_-{i'}\ar[dr]_-g & Y\times Z\ar[d]_{p'} & \\
&& Z &}\]
where $i'=\id_Y\times g$, $i''(x,y)=(x,y,g(y))$, $i'''=\id_X\times(gf)$ and
$p',p'',p''',q$ are the obvious projections.
Then $\theta_g=\theta_{p'}\circ\theta_{i'}$ and $\theta_{gf}=\theta_{p'''}\circ
\theta_{i'''}$. So it suffices to prove that :
\begin{itemize}
\item[(a)] $\theta_q\circ\theta_{i''i}=\theta_{i'''}$;
\item[(b)] $\theta_{i''i}=\theta_{i''}\theta_{i}$;
\item[(c)] $\theta_{p''}\circ\theta_q=\theta_{p'}\circ\theta_{p''}$;
\item[(d)] $\theta_{p''}\circ\theta_{i''}=\theta_{i'}\circ\theta_p$.

\end{itemize}

Point (b) follows from Corollary \ref{cor_i^*} and point (c) from the
explicit formula for the inverse image of a projection in
Corollary \ref{cor_inverse_image_projection}.
The other two compatibilities can easily be proved directly.
\end{proof}

Finally, we have :

\begin{subprop} The functor $\boxtimes$ of Proposition \ref{prop_obvious2}
induces a natural isomorphism between the $2$-functors $H_{mf}^*\times H_{mf}^*:
\Sch_k\times\Sch_k\fl\TR$ and $\Sch_k\times\Sch_k\flnom{\times}
\Sch_k \flnom{H_{mf}^*}\TR$, where the first arrow sends $(X,Y)$ to $X\times Y$.

In other words, if $f_1:X_1\fl Y_1$ and $f_2:X_2\fl Y_2$ are morphisms
in $\Sch/k$ and $L_1\in \Db\Perv_{mf}(Y_1)$, $L_2\in \Db\Perv_{mf}(Y_2)$, then
we have an isomorphism
\[(f_1\times f_2)^*(L_1\boxtimes L_2)\iso (f_{1}^*L_1)\boxtimes(f_{2}^*L_2)\]
functorial in $L_1$ and $L_2$ and compatible with the composition of
arrows in $\Sch/k$.

\label{prop_boxtimes_inverse_images}
\end{subprop}

\begin{proof} By the construction of the functors $f^*$ above,
we only need to show the statement when $f_1$ and $f_2$ are both closed
immersions, or when they are both projections. If $f_1$ and $f_2$ are both
projections, the result is obvious. If they are both closed
immersions, the result follows from the construction in the proof of
Corollary \ref{cor_i^*} and from Proposition \ref{prop_boxtimes_direct_images}.
\end{proof}

\subsection{Poincar\'e-Verdier duality}

Just as in sections \ref{direct_images} and \ref{inverse_images},
we can prove the
following result.

\begin{subprop} 
There exists a $2$-functor $H_{mf,!}:\Sch/k\fl\TR$
with $H_{mf,!}(X)=\Db\Perv_{mf}(X)$ for every $X\in\Ob(\Sch/k)$ and
a natural transformation $R:H_{mf,!}\fl H_{m,!}$ (with the notation
of Example \ref{example_crossed_functor}) such that :
\begin{itemize}
\item[(a)] for every $X\in\Ob(\Sch/k)$, the functor $R_X:\Db\Perv_{mf}(X)\fl
\Dbm(X)$ is the functor of Theorem \ref{thm_main1};
\item[(b)] for every morphism
$f:X\fl Y$, the natural transformation 
$\rho_f:R_Y\circ H_{mf,!}(f)\fl
H_{m,!}(f)\circ R_X$ is an isomorphism.
\end{itemize}

This functor satisfies the same compatibility with $\boxtimes$ as in
Proposition \ref{prop_boxtimes_direct_images}, and it admits a global
right adjoint $H_{mf}^!$.

\label{prop_!}
\end{subprop}

Moreover, by
Proposition \ref{prop_obvious2}, we have an exact contravariant endofunctor
$D_X$ of $\Db\Perv_{mf}(X)$ together with an isomorphism $D_X^2\simeq\id$,
for every $X\in\Ob(\Sch/k)$.

\begin{subprop}
Let $f:X\fl Y$ be a morphism of $k$-schemes.
\begin{itemize}
\item[(i)] We have a natural isomorphism $\alpha_f:f_*\iso D_Y\circ f_!\circ D_X$
such that, if $g:Y\fl Z$ is another morphism of $k$-schemes, then the
isomorphism $\alpha_{gf}:(gf)_*\simeq D_Z(gf)_! D_X$ is equal to the isomorphism
\[(gf)_*\simeq
g_* f_*\flnom{\alpha_g\alpha_f}
D_Z g_!D_YD_Yf_!D_X\simeq D_Z g_!f_!D_X\simeq
D_Z(gf)_!D_X\]
where the first and fourth arrows are given by the composition isomorphisms
of the $2$-functors $H_{mf,*}$ and $H_{mf,!}$, and the third arrow is given
by the isomorphism $D_Y^2\simeq\id$.

\item[(ii)]  We have a natural isomorphism $\beta_f:f^*\iso D_X\circ f^!
\circ D_Y$, where $f^!=H_{mf}^!(f)$,
such that, if $g:Y\fl Z$ is another morphism of $k$-schemes, then the
isomorphism $\beta_{gf}(gf)^*\simeq D_X(gf)^! D_Z$ is equal to the isomorphism
\[(gf)^*\simeq
f^*g^* \flnom{\beta_f\beta_g}
D_X f^!D_YD_Yg^!D_Z\simeq D_X f^!g^! D_Z\simeq
D_X(gf)^!D_Z\]
where the first and fourth arrows are given by the composition isomorphisms
of the $2$-functors $H_{mf}^*$ and $H_{mf}^!$, and the third arrow is given
by the isomorphism $D_Y^2\simeq\id$.

\item[(iii)] If $f$ is smooth and purely of relative dimension $d$, then
we have an natural isomorphism $f^![-d]\simeq f^*[d](d)$ of functors
$\Db\Perv_{mf}(Y)\fl \Db\Perv_{mf}(X)$.

\item[(iv)] If $f$ is smooth and purely of relative dimension $d$, then
the functor $f^*:\Db\Perv_{mf}(Y)\fl \Db\Perv_{mf}(X)$ admits a left
adjoint $f_\sharp$.

\item[(v)] If $i:X\fl Y$ is a closed immersion, then we have a natural
isomorphism $i_!\iso i_*$.

\end{itemize}

\label{prop_duality}
\end{subprop}

\begin{proof} 
Point (iii) follows from the fact that both functors are t-exact and that
such an isomorphism exists in the category of functors
$\Perv_{mf}(Y)\fl\Perv_{mf}(X)$ (because it does for mixed perverse sheaves
and the categories $\Perv_{mf}$ are full subcategories of the categories of
mixed perverse sheaves).

Point (iv) follows from point (iii) : take $f_\sharp=f_![2d](d)$.

Point (v) is proved like point (iii) : both functors are
$t$-exact, and the natural isomorphism exists when we see $i_!$ and $i_*$ as
functors from $\Perv_m(X)$ to $\Perv_m(Y)$.

Let's prove (i).
By the construction of $f_*$ in section \ref{direct_images}
(and point (iii) applied to inverse images by open immersions)
it suffices to prove the analogous result
for the functors $\Hp^0 f_*$ and $\Hp^0 f_!$ if $f$ is affine. But then this
follows from the case of the categories $\Dbc(X)$.

Point (ii) now follows from (i) and from the uniqueness of adjoint functors.
\end{proof}

\section{Tensor products and internal Homs}
\label{tensor}

\begin{defi}  Let $X$ be a $k$-scheme. We denote by $\Delta_X:X\fl X\times X$
the diagonal embedding. We define a functor $\otimes_X:(\Db\Perv_{mf}(X))^2
\fl \Db\Perv_{mf}(X)$ by $K\otimes_X L=\Delta_X^*(K\boxtimes L)$. 

\end{defi}

Note that it follows from Proposition \ref{prop_boxtimes_inverse_images}
that, for every morphism of $k$-schemes $f:X\fl Y$ and
all $K,L\in\Ob \Db\Perv_{mf}(Y)$, we have a canonical isomorphism
\[f^*(K\otimes_Y L)\simeq (f^* K)\otimes_X(f^*L).\]

\begin{prop} The operation $\otimes_X$ defined above makes
 $\Db\Perv_{mf}(X)$ into a symmetric mono\"idal
triangulated category. Also, the object $\ungras_X$ constructed
in Corollary \ref{cor_constant_sheaf} is a unit for $\otimes_X$,
and the functor $R_X:\Db\Perv_{mf}(X)\fl \Dbm(X)$
is symmetric mono\"idal unitary.

\end{prop}

\begin{proof} The first statement follows easily from the commutativity
and associativity of $\boxtimes$ (which in turn follows from the similar
statement in $\Dbm(X)$, as $\boxtimes$ is exact).
Moreover,
for every $K\in\Ob \Db\Perv_{mf}(X)$, if $p:X\times X\fl X$ is the
second projection, then :
\[K\otimes_X\ungras_X=\Delta_X^*(K\boxtimes\ungras_X)=\Delta_X^* p^*K\simeq
(p\Delta_X)^*K\simeq K\]
because $p\Delta_X=\id_X$. 
This proves the second statement. Finally, the fact that $R_X$ is
mono\"idal follows from the fact that it preserves $\boxtimes$, and
the last statement of Corollary \ref{cor_constant_sheaf} 
(i.e., the isomorphism $R_X(\ungras_X)\simeq\underline{E}_X$)
implies that $R_X$ is
unitary.
\end{proof}

The main result of this section in the following :

\begin{prop} For every $k$-scheme $X$ and every $K\in\Ob \Db\Perv_{mf}(X)$,
the endofunctor $K\otimes_X\cdot$ of $\Db\Perv_{mf}(X)$ has a right adjoint
$\Homf(K,\cdot)$,
given by $L\fle D_X(K\otimes_X D_X(L))$.
Moreover, for all $K,L,M\in\Ob\Db\Perv_{mf}(X)$, we have a commutative
diagram
\[\xymatrix{R\Hom_{\Db\Perv_{mf}(X)}(K\Ltimes_X L,M)\ar[r]^-\sim\ar[d]_{R_X}& 
R\Hom_{\Db\Perv_{mf}(X)}
(L,D_X(K\Ltimes_X D_X M))\ar[d]^{R_X}\\
R\Hom_{\Dbm(X)}(R_X(K)\Ltimes_X R_X(L),R_X(M))\ar[r]_-\sim& R\Hom_{\Dbm(X)}
(R_X(L),D_X(R_X(K)\Ltimes_X D_X R_X(M)))
}\]
where the horizontal arrows are the adjunction isomorphisms
(see Lemma \ref{lemma_tens1} for the identification
$D_X(R_X(K)\Ltimes_X D_X R_X(M))=\Homf_X(R_X(K),R_X(M))$).

\label{prop_internal_Hom}
\end{prop}

In the lemmas that follow, we will denote the structural morphism
$X\fl\Spec k$ by $a$.
Remember that we write $K_X=a^!\underline{E}_{\Spec k}$ for the dualizing
complex in $\Dbh(X)$. This is an object of $\Dbm(X)$ (because
$\underline{E}_{\Spec k}$ clearly is a mixed complex, and $a^!$
preserves mixed complexes).

\begin{lemma} In the category $\Dbh(X)$,
we have a canonical isomorphism, functorial in $K$ and $L$ :
\[\Homf_X(K\Ltimes_X L,K_X)\simeq
\Homf_X(K,D_X(L)).\]
Moreover, these complexes are concentrated in perverse degree $\geq 0$
if $K$ and $L$ are perverse.

\label{lemma_tens1}
\end{lemma}

In particular, if we replace $L$ by $D_X(L)$, we get a natural
isomorphism
\[\Homf_X(K,L)\simeq D_X(K\Ltimes_X D_X(L)),\]
which explains the definition of the internal Hom given in
Proposition \ref{prop_internal_Hom}.

\begin{proof} By Theorem 6.3(ii) of \cite{Ek} (see also the
remark following Definition 1.2 of \cite{Hu} for the extension of this to the
category $\Dbh(X)$),
we have a natural isomorphism
\[\Homf_X(K\Ltimes_X L,M)=\Homf_X(K,\Homf_X(L,M))\]
for all $K,L,M\in\Ob\Dbh(X)$.
Applying this to $M=K_X$
gives the desired isomorphism.

If $K$ and $L$ are perverse, then the complex $K\Ltimes_X L$ is
concentrated in perverse degree $\leq 0$ (because it is equal
by definition to $\Delta_X^*(K\boxtimes L)$, where $\Delta_X:X\fl X\times X$
is the diagonal morphism, and $\Delta_X^*$ is right t-exact), so
its dual $\Homf_X(K\Ltimes_X L,K_X)$ is concentrated in perverse
degree $\geq 0$.
\end{proof}

\begin{lemma} If $K,L\in\Ob\Perv_h(X)$, then
the complex $a_!(K\Ltimes_X L)\in\Dbh(\Spec k)$ is concentrated in
degree $\leq 0$, and so the adjunction $(a_!,a^!)$ gives a canonical
isomorphism
\[\Hom_{\Dbh(X)}(K\Ltimes_X L,K_X)\simeq\Hom_{\Perv_h(\Spec k)}(\H^0(a_{!}
(K\Ltimes_X L)),\underline{E}_{\Spec k})\]
and equalities
\[\Ext^i_{\Dbh(X)}(K\Ltimes_X L,K_X)=0\]
for every $i<0$.

\label{lemma_tens2}
\end{lemma}

We could also have deduced the vanishing of $\Ext^i_{\Dbh(X)}(K
\Ltimes_X L,K_X)$ for $i<0$ from the adjunction isomorphism
$\Ext^i_{\Dbh(X)}(K\Ltimes_X L,K_X)=\Ext^i(K,D_X(L))$. (But we
won't be able to do this in the next lemma, which is the analogous
statement in $\Db\Perv_{mf}(X)$.)

\begin{proof} We have
\[a_!(K\Ltimes_X L)\simeq D_{\Spec k}(a_* D_X(K\Ltimes_X L))\simeq
D_{\Spec k}(a_* \Homf_X(K,D_X(L))),\]
where the second isomorphism comes from Lemma \ref{lemma_tens1}.
So it suffices to show that $a_* \Homf_X(K,D_X(L)))
=R\Hom_{\Dbh(X)}(K,D_X(L))$ is concentrated in degree $\geq 0$.
As $K$ and $D_X(L)$ are perverse, this just follows from the definition
of a t-structure.

Now, using the adjunction $(a_!,a^!)$ and the fact that $K_X=a^!\underline{E}
_{\Spec k}$, we get a canonical isomorphism
\[\Hom_{\Dbh(X)}(K\Ltimes_X L,K_X)=\Hom_{\Dbh(\Spec k)}(a_!(K\Ltimes
_X L),\underline{E}_{\Spec k}).\]
The second statement follows from this and from the fact that
$a_!(K\Ltimes_X L)$ is concentrated in degree $\leq 0$.
\end{proof}

\begin{lemma} If $K,L\in\Ob\Perv_{mf}(X)$, then
the complex $a_!(K\Ltimes_X L)\in\Db\Perv_{mf}(\Spec k)$ is concentrated in
degree $\leq 0$, and so the adjunction $(a_!,a^!)$ gives a canonical
isomorphism
\[\Hom_{\\Db\Perv_{mf}(X)}(K\Ltimes_X L,a^!\ungras_{\Spec k})\simeq
\Hom_{\Perv_{mf}(\Spec k)}(\H^0(a_{!}
(K\Ltimes_X L)),\ungras_{\Spec k})\]
and equalities
\[\Ext^i_{\Db\Perv_{mf}(X)}(K\Ltimes_X L,a^!\ungras_X)=0\]
for every $i<0$.

\label{lemma_tens3}
\end{lemma}

\begin{proof} We have
\[R_X(a_!(K\Ltimes_X L))\simeq a_!(R_X(K)\Ltimes_X R_X(L)),\]
so $R_X(a_!(K\Ltimes_X L))$ is concentrated in degree
$\leq 0$ by Lemma \ref{lemma_tens2}. The first statement follows from
the conservativity of $R_X$. The second statement is proved exactly
as the second statement of Lemma \ref{lemma_tens2}, using the
adjunction $(a_!,a^!)$ in the categories $\Db\Perv_{mf}$.
\end{proof}

\begin{lemma} Let $K,L\in\Ob\Perv_{mf}(X)$, write $K'=R_X(K)$,
$L'=R_X(L)$. Then 
the morphism
\[R_X:\Hom_{\Db\Perv_{mf}(X)}(K\Ltimes_X L,D_X(\ungras_X))\fl
\Hom_{\Dbh(X)}(K'\Ltimes_X L',K_X)\]
is an isomorphism. In particular,
there exists a unique isomorphism
$\alpha_{K,L}:\Hom_{\Db\Perv_{mf}(X)}(K\Ltimes_X L,D_X(\ungras_X))\iso
\Hom_{\Perv_{mf}(X)}(K,D_X(L))$, making the following diagram
commute
\[\xymatrix{\Hom_{\Db\Perv_{mf}(X)}(K\Ltimes_X L,D_X(\ungras_X))\ar[r]^-{
\alpha_{K,L}}
\ar[d]_{R_X} &
\Hom_{\Perv_{mf}(X)}(K,D_X(L))\ar[d]^{R_X} \\
\Hom_{\Dbh(X)}(K'\Ltimes_X L',K_X)\ar[r]_-{\sim} &
\Hom_{\Perv_h(X)}(K',D_X(L'))
}\]
where the bottom isomorphism comes from applying the functor $\H^0(X,.)$ to the
isomorphism of Lemma \ref{lemma_tens1}.

\label{lemma_tens4}
\end{lemma}

\begin{proof} As $\Perv_{mf}(X)$ is a full subcategory of
$\Perv_h(X)$, the morphism $R_X:\Hom_{\Perv_{mf}(X)}(K,D_X(L))\fl
\Hom_{\Perv_h(X)}(K',D_X(L'))$ is an isomorphism. So we just need
to show that
\[R_X:\Hom_{\Db\Perv_{mf}(X)}(K\Ltimes_X L,D_X(\ungras_X))\fl
\Hom_{\Dbh(X)}(K'\Ltimes_X L',K_X)\]
is an isomorphism. 
By Lemmas \ref{lemma_tens2} and \ref{lemma_tens3}, we have a commutative
diagram
\[\xymatrix{\Hom_{\Db\Perv_{mf}(X)}(K\Ltimes_X L,a^!\ungras_{\Spec k})
\ar[r]^-\sim\ar[d]_{R_X}&\Hom_{\Perv_{mf}(\Spec k)}(\H^0(a_{!}
(K\Ltimes_X L)),\ungras_{\Spec k})\ar[d]^{R_{\Spec k}}\\
\Hom_{\Dbh(X)}(K'\Ltimes_X L',K_X)\ar[r]_-\sim &
\Hom_{\Perv_{h}(\Spec k)}(\H^0(a_{!}
(K'\Ltimes_X L')),\underline{E}_{\Spec k})
}\]
The right vertical map in this diagram is an isomorphism because $\Perv_{mf}(
\Spec k)$ is a full subcategory of $\Perv_h(\Spec k)$, so the left vertical
map is also an isomorphism.
\end{proof}

We use the formalism of filtered derived categories, which is recalled at the beginning
of Section~\ref{section_f_cat}.
Let $\Af$ be an abelian category, and let $\DFb(\Af)$ be its bounded filtered
derived category. Let us recall the spectral sequence of \cite{BBD}
(3.1.3.4) : If $K$ and $L$ are two objects of $\DFb(A)$, then we have a spectral sequence
\[E_1^{pq}=\bigoplus_{j-i=p}\Ext^{p+q}_{\D(\Af)}
(\Gr_F^i K,\Gr_F^j L)\Longrightarrow
\Ext^{p+q}_{\D(\Af)}(\omega(K),\omega(L)).\]
Remember that $\omega:\DFb(\Af)\fl\Db(\Af)$ is the functor that forgets the
filtration.

\begin{lemma} 
Let $K_1^\bullet,K_2^\bullet$ be bounded complexes
of objects of $\Perv_h(X)$, and let $K_1,K_2$ be their images by
$\real:\Db\Perv_h(X)\fl\Dbh(X)$. Then, for every object $L$ of $\Dbh(X)$,
we have a spectral sequence
\[E_1^{pq}=\bigoplus_{a-b=-p}\Ext^q_{\Dbh(X)}(K_1^a\Ltimes_X D_X(K_2^b),L)
\Longrightarrow\Ext^{p+q}_{\Dbh(X)}(K_1\Ltimes_X D_X(K_2),L).\]

\label{lemma_tens5}
\end{lemma}

\begin{proof} By definition of the category $\Dbh(X)$, it suffices to
prove the statement in $\Dbc(\X)$, where $(A,\X,u)$ is an object of
$\Uf X$ such that all the $K^i_r$ (resp. $L$) extend to shifts of objects of
$\Perv(\X)$ (resp. $\Dbc(\X)$), that we will denote by the same letters.

Remember the construction of the
realization functor $\Db\Perv(\X)\fl\Dbc(\X)$ in Section~\ref{section_f_real}:
We consider the full subcategory
$\Dbete(\X)$ of objects $A$ of $\DFb(\X_\proet)$ such that $\Gr_F^iA[i]$ is
in $\Perv(\X)$ for every $i\in\Z$ and $0$ for $|i|$ big enough.
We have an equivalence $G:\Dbete(\X)\fl C^b(\Perv(\X))$ (see
\cite{BBD} 3.1.7 or Theorem~\ref{thm_real}), and $\real$ is induced by
$\omega\circ G^{-1}:C^b(\Perv(\X))\fl\Dbc(\X)$.

Let $\Delta:X\fl X\times X$ be the diagonal morphism. As
$K_1\Ltimes_X D_X(K_2)=\Delta^*(K_1\boxtimes D_X(K_2))$, we have a canonical
isomorphism
\[R\Hom_{\Dbc(\X)}(K_1\Ltimes_X D_X(K_2),L)=
R\Hom_{\Dbc(\X\times\X)}(K_1\boxtimes D_X(K_2),\Delta_* L).\]

Let $M=G^{-1}(K_1^\bullet\boxtimes D_X(K_2^\bullet))\in\Ob\Dbete(\X\times
\X)$. 
We can also see $\Delta_* L$ as an
object of $\DF((\X\times\X)_\proet)$ (because, for any abelian
category $\Af$, the category $\Db(\Af)$ is canonically
equivalent to the full subcategory of $A\in\Ob\DFb(\Af))$ such that
$\Gr_F^i A=0$ for $i\not=0$). Using the spectral sequence recalled
before the statement of the lemma (and (iii) of Proposition
\ref{prop_comp_real1}), we get a spectral sequence
\[E_1^{pq}=\bigoplus_{-i=p}\Ext^{p+q}_{\Dbc(\X\times\X)}(\Gr_F^i(M),
\Delta_* L)\Longrightarrow\Ext^{p+q}_{\Dbc(\X)}(K_1\Ltimes_X D_X(K_2),L).\]
For every $i\in\Z$, we have
\[\Gr_F^i M=\bigoplus_{a+b=i}K_1^a[-a]\boxtimes D_X(K_2^{-b})[-b]=
\bigoplus_{a-b=i}(K^a\boxtimes D_X(K^b))[-i].\]
So
\[E_1^{pq}=\bigoplus_{a-b=-p}\Ext^{p+q}_{\Dbc(\X\times\X)}(K_1^a\boxtimes
D_X(K_2^b),\Delta_*[-p])=\bigoplus_{a-b=-p}\Ext^p_{\Dbc(\X)}
(K^a\Ltimes_X D_X(K^b),L).\]
The statement of the lemma now follows by taking the limit over $A'$, with $A\subset A'\in\Uf$.
\end{proof}

\begin{lemma} Let $K_1^\bullet,K_2^\bullet$ be bounded complexes
of objects of $\Perv_{mf}(X)$, and let $K_1,K_2$ be their images by the canonical
functor
$\Cb\Perv_{mf}(X)\fl\Db\Perv_{mf}(X)$. Then, for every object $L$ of
$\Db\Perv_{mf}(X)$,
we have a spectral sequence
\[E_1^{pq}=\bigoplus_{a-b=-p}\Ext^q_{\Db\Perv_{mf}(X)}(K_1^a\Ltimes_X D_X(K_2^b),L)
\Longrightarrow\Ext^{p+q}_{\Db\Perv_{mf}(X)}(K_1\Ltimes_X D_X(K_2),L).\]
Moreover, the functor $R_X$ induces a morphism of spectral sequences
from this spectral sequence to the one of Lemma \ref{lemma_tens5}.

\label{lemma_tens6}
\end{lemma}

\begin{proof} The proof is exactly the same as for Lemma \ref{lemma_tens5},
except that we work in the filtered derived category
$\DFb(\Perv_{mf}(X\times X))$. The last statement is obvious.
\end{proof}

\begin{nota} Let $K\in\Ob\Dbh(X)$. We denote
by $\iota_K$ the evaluation
morphism 
\[K\Ltimes_X D_X(K)=K\Ltimes_X \Homf_X(K,K_X)\fl K_X.\]

\label{nota_iota_K}
\end{nota}

This morphism satisfies the following naturality property: If $u:K\ra L$ is a morphism in
$\Dbh(X)$, then we get a commutative square
\[\xymatrix@C=45pt{K\otimes_X D_X(L)\ar[r]^-{\id\otimes_X D_X(u)}\ar[d]_{u\otimes\id} & K\otimes_X D_X(K)\ar[d]^{\iota_K} \\
L\otimes_X D_X(L)\ar[r]_-{\iota_L} & K_X}\]

\begin{lemma} 
Let $K_1^\bullet,K_2^\bullet$ be bounded complexes
of objects of $\Perv_h(X)$, and let $K_1,K_2$ be their images by the functor
$\real:\Db\Perv_h(X)\fl\Dbh(X)$. Let
\[E_1^{pq}\Longrightarrow\Ext^{p+q}_{\Dbh(X)}(K_1\Ltimes_X
D_X(K_2),K_X)\]
be the spectral sequence of Lemma \ref{lemma_tens5} for $L=K_X$.

Then $E_1^{pq}=0$ if $q<0$. Moreover, if $K_1^\bullet=K_2^\bullet=K^\bullet$ and we write $K=K_1$,
then the element $\sum_{a\in\Z}\iota_{K^a}
$ of $E_1^{00}$ is in $\Ker(E_1^{00}\fl E_1^{10})=E_2^{00}$,
and the element $\iota_K$
of $\Hom_{\Dbh(X)}(K\Ltimes_X D_X(K),K_X)\supset E_\infty^{00}$ is the image
of $\sum_{a\in\Z}\iota_{K^a}$
by the map $E_2^{00}\fl E_\infty^{00}$.

\label{lemma_tens7}
\end{lemma}

\begin{proof} 
We have
\[E_1^{pq}=\bigoplus_{a-b=-p}\Ext^q_{\Dbh(X)}(K^a_1\Ltimes_X D_X(K^b_2),
K_X).\]
As all the $K^a_1$ and $D_X(K^b_2)$ are perverse, this is $0$ for
$q<0$ by Lemma \ref{lemma_tens2}.
This implies that $E_2^{00}=\Ker(E_1^{00}\fl E_1^{10})$ and
that $E_r^{pq}=0$ for any $r\geq 1$ and any $q<0$, so
$E_\infty^{pq}=0$ for $q<0$. In particular, we get that
$E_{\infty}^{00}$ is a quotient of $E_2^{00}$ and that
$E_\infty^{00}\subset \Hom_{\Dbh(X)}(K_1\Ltimes_X D_X(K_1),K_X)$.
The last statement now follows from the construction of the spectral
sequence (and (iii) of Proposition \ref{prop_comp_real1}).
\end{proof}

\begin{lemma} Let $K^\bullet$ be a bounded complex of
objects of $\Perv_{mf}(X)$, and let $K\in\Ob\Db\Perv_{mf}(X)$ be its
image by the obvious functor $C^b\Perv_{mf}(X)\fl\Db\Perv_{mf}(X)$.
Then there exists a unique morphism $\iota_{K^\bullet}:K\Ltimes_X D_X(K)\fl
a^!\ungras_{\Spec k}$ satisfying the following
conditions :
\begin{itemize}
\item[(a)] The image of $\iota_{K^\bullet}$ by $R_X$ is the morphism
$\iota_{R_X(K)}$ of \ref{nota_iota_K}.
\item[(b)] The analogue of Lemma \ref{lemma_tens7} holds if we use
the spectral sequence of Lemma \ref{lemma_tens6}.

\end{itemize}

Moreover, if $u^\bullet:K^\bullet\ra L^\bullet$ is a morphism in $\Cb\Perv_{mf}(X)$ and if
$u:K\ra L$ is its image in $\Db\Perv_{mf}(X)$, then
the following square commutes:
\begin{equation*}\tag{*}\xymatrix@C=45pt{K\otimes_X D_X(L)\ar[r]^-{\id\otimes_X D_X(u)}\ar[d]_{u\otimes\id} & K\otimes_X D_X(K)\ar[d]^{\iota_{K^\bullet}} \\
L\otimes_X D_X(L)\ar[r]_-{\iota_{L^\bullet}} & K_X}\end{equation*}

\label{lemma_tens8}
\end{lemma}

\begin{proof} 
Let ${K'}^\bullet=R_X(K^\bullet)$ and $K'=R_X(K)$.
If $K^\bullet$ is concentrated in degree $0$, then, by
Lemma~\ref{lemma_tens4}, the morphism
\[R_X:\Hom_{\Db\Perv_{mf}(X)}(K\Ltimes_X D_X(K),D_X(\ungras_X))\fl
\Hom_{\Dbh(X)}(K'\Ltimes_X D_X(K'),K_X)\]
is an isomorphism. So condition (a) forces us to take $\iota_K=
R_X^{-1}(\iota_{K'})$, and condition (b) is trivial in this case.

We now construct $\iota_{K^\bullet}$ in the general case. The spectral sequence of Lemma
\ref{lemma_tens6} for $L=a^!\ungras_X$ is
\[E_1^{pq}=\Ext^q_{\Db\Perv_{mf}(X)}(K^a\Ltimes_X D_X(K^b),a^!\ungras_X)
\Longrightarrow\Ext^{p+q}_{\Db\Perv_{mf}(X)}(K\Ltimes_X D_X(K),a^!\ungras_X).\]
We have $E_1^{pq}=0$ for $q<0$ by Lemma \ref{lemma_tens3}.
As in the proof of Lemma \ref{lemma_tens7}, this implies that
$E_2^{00}=\Ker(E_1^{00}\fl E_1^{10})$ surjects to $E_\infty^{00}$
and that $E_\infty^{00}\subset\Hom_{\Db\Perv_{mf}(X)}(K\Ltimes_X D_X(K),a^!\ungras_X)$.
By condition (b), the element $\iota_K\in\Hom_{\Db\Perv_{mf}(X)}(K\Ltimes_X D_X(K),
a^!\ungras_X)$ that we want to construct must be the image of
$\sum_{a\in\Z}\iota_{K^a}\in E_1^{00}$. As $\iota_{K^a}$ exists and is
uniquely determined by the first case, it suffices to show
that $\sum_{a\in\Z}\iota_{K^a}\in\Ker(E_1^{00}\fl E_1^{10})$. Indeed,
condition (a) will then follow from the fact that $R_X$ induces a morphism
between the spectral sequences of Lemmas \ref{lemma_tens5}
and \ref{lemma_tens6} (and from Lemma \ref{lemma_tens7}).
We denote by
$E_1^{pq}(K')$ the spectral sequence of Lemma \ref{lemma_tens5} for ${K'}^\bullet$.
Then we have a commutative diagram
\[\xymatrix{E_1^{00}\ar[r]\ar[d]_{R_X} &E_1^{10}\ar[d]^{R_X} \\
E_1^{00}(K')\ar[r] & E_1^{10}(K')}\]
By Lemma \ref{lemma_tens4}, the vertical maps in this diagram are
isomorphisms. By Lemma \ref{lemma_tens7}, the image
by $R_X$ of $\sum_{a\in\Z}\iota_{K^a}\in E_1^{00}$, which is
$\sum_{a\in\Z}\iota_{{K'}^a}$ by construction of the $\iota_{K^a}$, is
in $\Ker(E_1^{00}(K')\fl E_1^{10}(K'))$. So
$\sum_{a\in\Z}\iota_{K^a}$ is in $\Ker(E_1^{00}\fl E_1^{10})$, and we are
done.

We prove the last statement. If $K^\bullet$ and $L^\bullet$ are both concentrated in degree $0$, then $K$ and $L$
are perverse, so the commutativity of (*) follows from the commutativity of the square we get applying $R_X$ and
from Lemma~\ref{lemma_tens4}. We treat the general case. For $K_1^\bullet,K_2^\bullet\in\Ob\Cb\Perv_{mf}(X)$, we denote
by ${}_{K_1^\bullet,K_2^\bullet}E_\bullet^{\bullet,\bullet}$ the spectral sequence of Lemma~\ref{lemma_tens6} for $K_1^\bullet$, $K_2^\bullet$ and $K_X$.
The morphisms $u^\bullet$ and $D_X(u^\bullet)$ induce morphisms of spectral sequence
${}_{L^\bullet,L^\bullet}E_\bullet^{\bullet,\bullet}\ra{}_{K^\bullet,L^\bullet}E_\bullet^{\bullet,\bullet}\longleftarrow{}_{K^\bullet,K^\bullet}E_\bullet^{\bullet,\bullet}$
that are compatible with the morphisms $\Ext^\bullet_{\Db\Perv_{mf}(X)}(L,L)\flnom{u^*}\Ext^\bullet_{\Db\Perv_{mf}(X)}(K,L)\stackrel{u_*}\longleftarrow\Ext^\bullet_{\Db\Perv_{mf}(X)}(K,K)$.
So we get the commutative diagram (**), where $\H(\cdot)$ means $\Hom_{\Db\Perv_{mf}(X)}(\cdot,K_X)$.
We have $\iota_{K^\bullet}\in\H(K\otimes_X D_X(K))$, and its image in $\H(K\otimes_X D_X(L))$ by the arrow in the last row of (**) is $\iota_{K^\bullet}\circ(\id\otimes_X D_X(u))$. Similarly,
we have $\iota_{L^\bullet}\in\H(L\otimes_X D_X(L))$, and its image in $\H(K\otimes_X D_X(L))$ by the arrow in the last row of (**) is $\iota_{L^\bullet}\circ(u\otimes_X\id)$.
On the other hand, during the construction of $\iota_{K^\bullet}$, we proved that the element $\sum_{a\in\Z}\iota_{K^a}$ of ${}_{K^\bullet,K^\bullet}E_1^{00}$ comes from a (unique)
element $e_{K^\bullet}$ of ${}_{K^\bullet,K^\bullet}E_2^{00}$, and that $\iota_{K^\bullet}$ is the image of $e_{K^\bullet}$ in $H(K\otimes_X D_X(K))$; we have a similar statement for $\iota_{L^\bullet}$.
By the commutativity of (**), it suffices to prove that the images of $\sum_{a\in\Z}\iota_{K^a}\in{}_{K^\bullet,K^\bullet}E_1^{00}$ and $\sum_{a\in\Z}\iota_{L^a}\in{}_{L^\bullet,L^\bullet}E_1^{00}$ in
${}_{K^\bullet,L^\bullet}E_1^{00}$ by the maps of the second row of (**) are equal. But the image of the first element is $\sum_{a\in\Z}\iota_{K^a}\circ(\id\otimes_X D_X(u^a))$,
and the image of the second element is $\sum_{a\in\Z}\iota_{L^a}\circ(u^a\otimes_X\id)$. So the equality of these images follows from the case where $K^\bullet$ and $L^\bullet$ are
concentrated in degree $0$, which we already treated.

\begin{equation*}\tag{**}\xymatrix@C=10pt{
\bigoplus\limits_{a\in\Z}\H(L^a\otimes_X D_X(L^a))\ar[r] & \bigoplus\limits_{a\in\Z}\H(K^a\otimes_X D_X(L^a)) & \bigoplus\limits_{a\in\Z}\H(K^a\otimes_X D_X(K^a))\ar[l] \\
{}_{L^\bullet,L^\bullet}E_1^{00}\ar[r]\ar@{=}[u] & {}_{K^\bullet,L^\bullet}E_1^{00}\ar@{=}[u] & {}_{K^\bullet,K^\bullet}E_1^{00}\ar[l]\ar@{=}[u] \\
{}_{L^\bullet,L^\bullet}E_2^{00}\ar[r]\ar@{^{(}->}[u]\ar@{->>}[d] & {}_{K^\bullet,L^\bullet}E_2^{00}\ar@{^{(}->}[u]\ar@{->>}[d]   & {}_{K^\bullet,K^\bullet}E_2^{00}\ar[l]\ar@{^{(}->}[u]\ar@{->>}[d]   \\
{}_{L^\bullet,L^\bullet}E_\infty^{00}\ar[r]\ar@{^{(}->}[d]& {}_{K^\bullet,L^\bullet}E_\infty^{00}\ar@{^{(}->}[d] & {}_{K^\bullet,K^\bullet}E_\infty^{00}\ar[l]\ar@{^{(}->}[d] \\
\H(L\otimes_X D_X(L))\ar[r] & \H(K\otimes_X D_X(L)) & \H(K\otimes_X D_X(K))\ar[l]
}\end{equation*}
\end{proof}

The last statement of Lemma~\ref{lemma_tens8} implies in particular that, for
$K\in\Ob\Db\Perv_{mf}(X)$, the morphism $\iota_{K^\bullet}:K\otimes_X D_X(K)\ra K_X$ does not
depend on the complex $K^\bullet\in\Ob\Cb\Perv_{mf}(X)$ representing $K$. We denote this
morphism by $\iota_K$. 

\begin{lemma} For $K,L\in\Ob\Db\Perv_{mf}(X)$, we define
a morphism
\[u_{K,L}:R\Hom_{\Db\Perv_{mf}(X)}(L,D_X(K))\fl R\Hom_{\Db\Perv_{mf}(X)}
(K\Ltimes_X L,a^!\ungras_{\Spec k})\]
as the composition of
\[K\Ltimes_X (.):R\Hom_{\Db\Perv_{mf}(X)}(L,D_X(K))\fl
R\Hom_{\Db\Perv_{mf}(X)}
(K\Ltimes_X L,K\Ltimes_X D_X(K))\]
and of
\[\iota_{K*}:R\Hom_{\Db\Perv_{mf}(X)}
(K\Ltimes_X L,K\Ltimes_X D_X(K))\fl
R\Hom_{\Db\Perv_{mf}(X)}
(K\Ltimes_X L,a^!\ungras_{\Spec k}).\]

Then this morphism is natural in $K$ and $L$, its image
by $R_X$ is the adjunction morphism
\[R\Hom_{\Dbh(X)}(R_X(L),D_X(R_X(K)))=R\Hom_{\Dbh(X)}
(R_X(K)\Ltimes_X R_X(L),K_X),\]
and it is an isomorphism.

\label{lemma_tens9}
\end{lemma}

\begin{proof} The naturality in $L$ is obvious and the second statement
follows from property (a) of Lemma \ref{lemma_tens8}.

We prove the naturality in $K$. Let $u:K\ra K'$ be a morphism in $\Db\Perv_{mf}(X)$.
We consider the diagram, where we write $\H$ for $R\Hom_{\Db\Perv_{mf}\Perv(X)}$
\[\xymatrix@C=30pt{
\H(L,D_X(K))\ar[r]^-{K\otimes_X(.)}\ar@{}[rd]_(.35){(1)} & \H(K\otimes_X L,K\otimes_X D_X(K))\ar[r]^-{\iota_{K*}} & \H(K\otimes_X L,K_X)\\
&  \H(K\otimes_X L,K\otimes_X D_X(K'))\ar[u]_{(\id\otimes_X D_X(u))_*}\ar[d]^{(u\otimes_X\id)_*}\ar@{}[ru]_(.6){(3)} & \\
& \H(K\otimes_X L,K'\otimes_X D_X(K'))\ar@/_2pc/[ruu]_-{\iota_{K'*}}\ar@{}[rd]^(.6){(4)} & \\
\H(L,D_X(K'))\ar[uuu]^-{D_X(u)_*}\ar[r]_-{K'\otimes_X(.)}\ar@/^2pc/[ruu]^-{K\otimes_X(.)}\ar@{}[ru]^(.35){(2)} 
& \H(K'\otimes_X L,K'\otimes_X D_X(K'))\ar[r]_-{\iota_{K'*}}\ar[u]_{(u\otimes_X\id)^*} & \H(K'\otimes_X L,K_X)\ar[uuu]_-{(u\otimes_X\id)^*}
}\]
Squares (1), (2) and (4) in this diagram clearly commute, and square (3) commutes
by the last statement of Lemma~\ref{lemma_tens8}. So the exterior rectangle commutes, which is what we needed to prove.

We turn to the proof of the third statement, i.e. that $u_{K,L}$ is an isomorphism.
We first prove it in the
case where $X$ is smooth and connected and $K=\Lf$, $L=\Mf$ are
lisse sheaves on $X$.
Let $d=\dim(X)$. Then we have $a^!\ungras_{\Spec k}=\ungras_X[2d](d)$ 
(by Proposition \ref{prop_1.3}(i))
and
$D_X(\Lf)=\Lf^*[2d](d)$, where $\Lf^*=\Homf(\Lf,\underline{E}_X)$ is the
dual locally constant sheaf (by the calculation at the end of
section \ref{l-adic_complexes} and Proposition
\ref{prop_eta_fully_faithful}).
So $u_{\Lf,\Mf}$ is a morphism
\[R\Hom_{\Db\Perv_{mf}(X)}(\Mf,\Lf^*)\fl R\Hom_{\Db\Perv_{mf}(X)}(\Lf\Ltimes_X
\Mf,\ungras_X),\]
and the morphism $\iota_\Lf:\Lf\Ltimes D_X(\Lf)\fl a^!\ungras_{\Spec k}$ of 
Lemma \ref{lemma_tens8} is
just the the canonical morphism $\Lf\Ltimes_X\Lf^*\fl\ungras_X$, shifted by $2d$
and twisted by $d$ (we see this easily from conditions (a) and (b)
of Lemma \ref{lemma_tens8}, as $\Lf$ is perverse up
to a shift). 
We will use the Yoneda description of the $\Ext^k$ groups, as in
section 3.2 of chapter III of Verdier's book \cite{Ve}.
The definition of $u_{\Lf,\Mf}$ gives the following
formula for the image of a class $c$ in
\[\Ext^i_{\Db\Perv_{mf}(X)}(\Mf,\Lf^*)=
\Ext^i_{\Db\Perv_{mf}(X)}(\Mf[d],\Lf^*[d])\ :\]
Choose an 
exact sequence in $\Perv_{mf}(X)$ representing $c$, say :
\[0\fl\Lf^*[d]\fl K_{i-1}\fl\dots\fl K_0\fl\Mf[d]\fl 0.\]
Tensoring this sequence by $\Lf$, we still get an exact sequence in $\Perv_{mf}(X)$ :
\[0\fl\Lf\Ltimes_X\Lf^*[d]\fl \Lf\Ltimes_X K_{i-1}\fl\dots\fl 
\Lf\Ltimes_X K_0\fl
\Lf\Ltimes_X\Mf[d]\fl 0.\]
Then $u_{\Lf,\Mf}(c)$ is represented by the exact sequence
\[0\fl\ungras_X[d]\fl K'_{i-1}\fl\Lf\Ltimes_X K_{i-2}\fl
\dots\fl \Lf\Ltimes_X K_0\fl
\Lf\Ltimes_X\Mf[d]\fl 0,\]
where $K'_{i-1}$ is the amalgamated sum
\[\ungras_X[d]\oplus_{\Lf\Ltimes_X\Lf^*[d]}(\Lf\Ltimes_X K_{i-1})\]
with the morphism $\Lf\Ltimes_X\Lf^*[d]\fl\ungras_X[d]$ 
being the shift of the obvious one.
We want to show that $u_{\Lf,\Mf}$ is bijective, so it suffices to construct its
inverse. Suppose that $c'$ is an element of
\[\Ext^i_{\Db\Perv_{mf}(X)}(\Lf\Ltimes_X\Mf,\ungras_X)=
\Ext^i_{\Db\Perv_{mf}(X)}(\Lf\Ltimes_X\Mf[d],\ungras_X[d]),\]
and choose an exact sequence in $\Perv_{mf}(X)$ representing $c'$, say :
\[0\fl\ungras_X[d]\fl L_{i-1}\fl\dots\fl L_0\fl\Lf\Ltimes_X\Mf[d]\fl 0.\]
Tensoring this sequence by $\Lf^*$, we still get an exact sequence in 
$\Perv_{mf}(X)$ :
\[0\fl\Lf^*[d]\fl \Lf^*\Ltimes_X L_{i-1}\fl\dots\fl \Lf^*\Ltimes_X L_0\fl
\Lf^*\Ltimes_X\Lf\Ltimes_X\Mf[d]\fl 0.\]
We send $c'$ to the element of $\Ext^i_{\Perv_{mf}(X)}(\Mf[d],\Lf^*[d])$
represented by the exact sequence
\[0\fl\Lf^*[d]\fl \Lf^*\Ltimes_X L_{i-1}\fl\dots\fl
\Lf^*\Ltimes_X L_1\fl L'_0\fl
\Mf[d]\fl 0,\]
where $L'_{0}$ is the fiber product
\[(\Lf^*\Ltimes_X L_0)\times_{\Mf[d]}(\Lf^*\Ltimes_X\Lf\Ltimes_X\Mf[d])\]
with the morphism $\Lf^*\Ltimes_X\Lf\Ltimes_X\Mf[d]\fl\Mf[d]$ 
coming from $\Lf^*\Ltimes_X\Lf\fl\ungras_X$ by tensoring by $\Mf[d]$. This is
clearly the inverse of $u_{\Lf,\Mf}$.

Now we show that the morphism $u_{K,L}$ is an
isomorphism for all $K,L\in\Ob\Db\Perv_{mf}(X)$.
Note the following two reductions : First,
using the fact that all the functors are triangulated and the five lemma,
we see that if we have an
exact triangle
\[K'\fl K\fl K''\flnom{+1}\]
such that the result is true for $(K',L)$ and $(K'',L)$, then the result
if true for $(K,L)$. There is a similar statement for the second variable $L$.
So it suffices to prove the result for $K$ and $L$ concentrated in 
perverse degree $0$,
and we may also assume that $K$ and $L$ are simple perverse sheaves.
Second, suppose that we have a closed immersion $i:Y\fl X$, and let
$j:U:=X-Y\fl X$ be the complementary open immersion. Then we have
a commutative diagram whose columns are distinguished triangles (all the
$R\Hom$s are taken in the appropriate category $\Db\Perv_{mf}(Z)$, with $Z\in\{X,U,Y\}$) :
\[\xymatrix{R\Hom(i^*L,i^!D_X K)\ar[d]\ar[r]^-{i^* u_{K,L}} & 
R\Hom(i^*(K\Ltimes_X L),i^!a^!\ungras_{\Spec k})\ar[d]
\\
 R\Hom(L,D_X K)\ar[d]\ar[r]^-{u_{K,L}}  & 
R\Hom(K\Ltimes_X L,a^!\ungras_{\Spec k})\ar[d] 
\\
R\Hom(j^*L,j^* D_X K)\ar[d]_{+1}\ar[r]^-{j^*u_{K,L}}
& R\Hom(j^*(K\Ltimes_X L),j^* a^!\ungras_{\Spec k})\ar[d]_{+1} \\
&
}\]
Moreover, using the compatibility of $\Ltimes_X$ with inverse images
and point (ii) of Proposition \ref{prop_duality}, we get isomorphisms :
\[R\Hom(i^*(K\Ltimes_X L),i^! a^!\ungras_{\Spec k})\simeq
R\Hom((i^*K)\Ltimes_Y (i^* L),a_Y^!\ungras_{\Spec k}) \]
and
\[R\Hom(j^*(K\Ltimes_X L),j^* a^!\ungras_{\Spec k})
\simeq R\Hom((j^* K)\Ltimes_U(j^*L),a_U^!\ungras_{\Spec k}),\]
where $a_Y=a\circ i$ and $a_U=a\circ j$.
It is easy to see that these isomorphisms identify $i^* u_{K,L}$ 
(resp. $j^* u_{K,L}$) with $u_{i^*K,i^*L}$ (resp. $u_{j^*K,j^*L}$).
So the result for $X$ follows from the result for $Y$ and $U$.

Using the two reductions above and Noetherian induction on $X$, we can reduce
to the case where $X$ is smooth and $K$ and $L$ are both shifts of locally
constant sheaves on $X$. But this case has already been treated in the first
part of the proof.
\end{proof}

\begin{proof}[Proof of Proposition \ref{prop_internal_Hom}]
 We have to construct an isomorphism
\[R\Hom_{\Db\Perv_{mf}(X)}(K\Ltimes_X L,M)\iso R\Hom_{\Db\Perv_{mf}(X)}
(L,D_X(K\Ltimes_X D_X M))\]
functorial in $K,L,M\in\Ob \Db\Perv_{mf}(X)$ and compatible (via $R_X$)
with the adjunction morphism in $\Dbh(X)$.
But such an isomorphism is
given by
\[u_{L,K\Ltimes_X D_X M}\circ u_{K\Ltimes_X L,D_X M}^{-1},\]
where $u_{.,.}$ is constructed in Lemma \ref{lemma_tens9}.
\end{proof}

\section{Weight filtration on complexes}

The goal of this section is to generalize the results of section 3
of \cite{M}, and in particular the formula for the intermediate extension
of a pure perverse sheaf, to the categories $\Perv_{mf}(X)$ and their derived
categories.
This was the original
motivation for considering the categories $\Db\Perv_{mf}(X)$.

\begin{defi} Let $X$ be a $k$-scheme. For every $a\in\Z\cup\{\pm\infty\}$,
we denote by $\DP^{\leq a}(X)$ (resp. $\DP^{\geq a}(X)$) the full subcategory
of $\Db\Perv_{mf}(X)$ whose objects are the complexes $K$ such that,
for every $i\in\Z$, $\H^iK\in\Perv_{mf}(X)$ is of weight $\leq a$ (resp.
$\geq a$).

\end{defi}

Note that $\DP^{\leq a}(X)$ and $\DP^{\geq a}(X)$ are triangulated subcategories
of $\Db\Perv_{mf}(X)$.

\begin{prop} Let $K,L\in\Ob\Perv_{mf}(X)$. Suppose that there exists
$a\in\Z$ such that $K$ is of weight $\leq a$ and $L$ is of weight $\geq a+1$.
Then we have, for every $i\in\Z$,
\[\Ext^i_{\Perv_{mf}(X)}(K,L)=0.\]

\label{prop_nullite}
\end{prop}

For categories like that of mixed Hodge modules,
this result follows from Lemma 6.9 of \cite{S}, but M. Saito assumes (and uses)
the fact that pure objects are semisimple, which is false in our case.

\begin{proof} We obviously have $\Ext^i_{\Perv_{mf}(X)}(K,L)=0$ if $i<0$,
and $\Hom_{\Perv_{mf}(X)}(K,L)=0$ because the weights of $K$ and $L$ are disjoint.
We denote by $W$ the weight filtration on objects of $\Perv_{mf}(X)$. 
For every $b\in\Z$, we get an endofunctor $W_b$ of $\Perv_{mf}(X)$, which
is exact because weight filtrations are
strictly compatible with morphisms in $\Perv_{mf}(X)$ (by Lemma 3.8 of
\cite{Hu}).

As in the proof of Proposition \ref{prop_internal_Hom},
we will use the Yoneda description of the $\Ext^k$ groups (see
section 3.2 of chapter III of Verdier's book \cite{Ve}).
Let $i\geq 1$ and
let $\alpha\in\Ext^i_{\Perv_{mf}(X)}(K,L)$. Choose an exact sequence
\[0\fl L\flnom{u_i}M_{i-1}\flnom{u_{i-1}}\ldots\flnom{u_1}M_0\flnom{u_0}K\fl 0\]
in $\Perv_{mf}(X)$ that represents $\alpha$. Applying $W_a$ to this exact
sequence and using the fact that $W_a K=K$ and $W_a L=0$,
we get a morphism of exact sequences
\[\xymatrix{
0\ar[r]& L\ar[r]^-{u_i}& M_{i-1}\ar[r]^-{u_{i-1}} & \ldots \ar[r]^-{u_1}
& M_0\ar[r]^-{u_0} & K\ar[r]& 0 \\
0\ar[r]& L\ar[r]^-{\id_L+0}\ar@{=}[u]^{u+\can}& L\oplus W_a M_{i-1}
\ar[r]^-{0+u_{i-1}}\ar[u]^{\can} & \ldots \ar[r]^-{u_1}
& W_a M_0\ar[r]^-{u_0}\ar[u]^{\can} & K\ar@{=}[u]\ar[r]& 0
}\]
where $\can:W_a\fl\id$ is the canonical inclusion.
So the class $\alpha$ is also represented by the second row of this
diagram, hence it is trivial.
\end{proof}

\begin{cor} For every $a\in\Z\cup\{\pm\infty\}$, the pair
$(\DP^{\leq a},\DP^{\geq a+1})$ is a t-structure on $\\Db\Perv_{mf}(X)$.

\end{cor}

We denote by $w_{\leq a}$ and $w_{\geq a+1}$ the truncation functors for this
t-structure. They extend the exact functors $K\fle W_a K$ and $K\fle
K/W_aK$ on $\Perv_{mf}(X)$.

\begin{proof}
Once we have the vanishing result of Proposition \ref{prop_nullite}, the proofs
of Lemmas 3.2.1 and 3.2.2 of \cite{M} apply without modification.
\end{proof}

\begin{cor} The results of sections 3 and 5.1 of \cite{M} are still
true in our situation. In particular, if $j:U\fl X$ is an open immersion
of $k$-schemes and $K\in\Ob\Perv_{mf}(U)$ is pure of weight $a$, then the
canonical morphisms
\[w_{\geq a}j_!K\fl j_{!*}K\fl w_{\leq a}j_*K\]
are isomorphisms.

\end{cor}

\begin{proof} The proofs of \cite{M} apply without modification.
\end{proof}

\appendix
\section{Filtered derived categories and f-categories}
\label{section_fil_der}

Let $\Df$ be a triangulated category and let $(\Df^{\leq 0}, \Df^{\geq 0})$ be a t-structure on $\Df$ with heart $\Cf=\Df^{\leq 0}\cap\Df^{\geq 0}$.
Under appropriate hypotheses, the inclusion $\Cf\ra\Df$ extends to a triangulated functor $\real:\Db(\Cf)\ra\Df$ called the \emph{realization functor}; the first goal of this appendix
is to explain how to construct this extension (see Theorem~\ref{thm_real}). 

Now suppose that $\Df'$ is another triangulated category, that $({\Df'}^{\leq 0},{\Df'}^{\geq 0})$ is a t-structure with heart
$\Cf'$ and that $T:\Df\ra\Df'$ is a triangulated functor. Under appropriate hypotheses
(see Theorem~\ref{thm_real}), we will have realization
functors $\real:\Db(\Cf)\ra\Df$ and $\real:\Db(\Cf')\ra\Df'$. The second goal of this appendix is, supposing that there are ``enough'' $T$-acyclic objects in $\Cf$,
to extend $\H^0 T:\Cf\ra\Cf'$ to a triangulated functor $DT:\Db(\Cf)\ra\Db(\Cf')$ such that $\real\circ DT\simeq T\circ\real$; the functor $DT$ should be the restriction 
to $\Db(\Cf)$ of the derived functor of $\H^0T$ if $T$ is left or right t-exact. See Proposition~\ref{prop_der_fil} and Remark~\ref{rmk_der_fil} for precise statements.

The technical tool that we need for the two goals above is the formalism of f-categories over triangulated categories (introduced by Beilinson in Appendix~A of~\cite{Be1}), 
which abstracts the properties of filtered derived categories. So we start with a review of this formalism.

\subsection{Review of f-categories and f-functors}
\label{section_f_cat}

\subsubsection*{Filtered derived categories}

We recall the definition of filtered derived categories, as they are the motivation for the definition of f-categories, and also
the main example of this formalism.

Let $\Af$ be an abelian category.
We denote by $\Fil(\Af)$ category of filtered objects of $\Af$, where
filtrations are assumed to be decreasing; it has a full subcategory $\Filf(\Af)$, whose objects are
filtered objects of $\Af$ whose filtration is finite, which means in particular that it is separated and exhaustive.
Both $\Fil(\Af)$ and $\Filf(\Af)$ are quasi-abelian categories.
We denote by $\D(\Filf(\Af))$ (resp. $\Dpl(\Filf(\Af))$, $\D^-(\Filf(\Af))$, $\Db(\Filf(\Af))$) the unbounded (resp. bounded above, bounded below, bounded)
derived category of $\Filf(\Af)$ (see for example Definition~2.3 of Schapira and
Schneiders's \cite{Scha}, or \cite[\href{https://stacks.math.columbia.edu/tag/05S2}{Definition 05S2}]{stacks-project} and
\cite[\href{https://stacks.math.columbia.edu/tag/05S4}{Definition 05S4}]{stacks-project}; note that the definitions of the variously bounded derived
categories in \cite{Scha} and in the Stacks Project are a priori different, but they define the same subcategories by \cite[\href{https://stacks.math.columbia.edu/tag/05S5}{Lemma 05S5}]{stacks-project}).
Following the convention of Section~3.1 of~\cite{BBD} and Chapter~V of Illusie's~\cite{CT1}, we define the
the \emph{filtered derived category} $\DF(\Af)$ (resp. the \emph{bounded above filtered derived category} $\DFpl(\Af)$, etc) of $\Af$ to be the full subcategory of $\D(\Filf(\Af))$ (resp.
$\Dpl(\Filf(\Af))$ etc) of complexes whose filtration is finite, and not just finite in each degree.

As the filtration on a finite complex of objects of $\Filf(\Af)$ is always finite, we have $\DFb(\Af)=\Db(\Filf(\Af))$, so the definitions of
\cite[\href{https://stacks.math.columbia.edu/tag/05RX}{Section 05RX}]{stacks-project}, \cite{BBD} and \cite{CT1} coincide for the bounded filtered
derived category.

We denote objects of $\Filf(\Af)$ and $\D(\Filf(\Af))$ by $(K,F^\bullet)$, or often just $K$.

We have additive functors:
\begin{itemize}
\item $s:\Filf(\Af)\ra\Filf(\Af)$, $(K,F^\bullet)\mapsto(K,F^{\bullet-1})$;
\item $\Fil^r:\Filf(\Af)\fl\Filf(\Af)$ sending $(K,F^\bullet)$ to $F^r$ with the filtration induced by $F^\bullet$;
\item $(.)/\Fil^r:\Filf(\Af)\fl\Filf(\Af)$ sending $(K,F^\bullet)$ to $K/F^r K$ with the filtration induced by $F^\bullet$;
\item $\Gr^r:\Filf(\Af)\fl\Af$, $(K,\Fil^\bullet)\fle\Fil^r K/\Fil^{r-1}K$;
\item $\Gr=\bigoplus_{r\in\Z}\Gr^r:\Filf(\Af)\ra\Af$;
\item $\omega:\Fil(\Af)\fl\Af$, $(K,F^\bullet)\fle K$;
\item $i:\Af\ra\Filf(\Af)$ sending $K\in\Ob(\Af)$ to $K$ with the trivial filtration $F^\bullet$ ($F^r=K$ for $r\leq 0$ and $F^r=0$ for $r\geq 1$).

\end{itemize}

These induce functors on the categories of complexes, and, as in
\cite[\href{https://stacks.math.columbia.edu/tag/05S3}{Lemma 05S3}]{stacks-project}, we get triangulated functors
$s,\sigma_{\leq r},
\sigma_{\geq r+1}:\D(\Filf(\Af))\fl\D(\Filf(\Af))$, 
$\Gr^r,\Gr:\D(\Filf(\Af))\fl\D(\Af)$
$\omega:\D(\Filf(\Af))\fl\D(\Af)$ and $i:\D(\Af)\ra\D(\Filf(\Af))$. For $?\in\{+,-,b\}$, these functors send to $\D^?(\Filf(\Af))$ (resp. $\D^?(\Af)$) to $\D^?(\Filf(\Af))$ (resp. $\D^?(\Af)$), and
they also respect the various filtered derived categories (when it makes sense, for example $s$ sends $\DF^?(\Af)$ to $\DF?(\Af)$, and $i$ sends $\D^?(\Af)$ to $\DF^?(\Af)$).
For every $r\in\Z$, we have
a natural transformation $\Gr^r\fl\Gr^{r+1}[1]$ (coming from the triangle
$\Gr^{r+1}K\fl F^rK/F^{r+2}K\fl\Gr^r K$, if $(K,F^\bullet)$ is an object
of $\D(\Filf(\Af))$). We also have a natural transformation $\alpha:\id_{\DF(\Af)}\ra s$ coming from the inclusions $F^rK\subset F^{r-1}K$, for $(K,F^\bullet)\in\Ob\Filf(\Af)$.

For every $n\in\Z$, let
\[\DF(\leq n)=\{K\in\Ob\DF(\Af)\mid\forall r\geq n+1,\ \Gr^r K=0\}\]
and
\[\DF(\geq n)=\{K\in\Ob\DF(\Af)\mid\forall r\leq n-1,\ \Gr^r K=0\};\]
these are full subtriangulated categories of $\DF(\Af)$, stable by isomorphisms. 
Thanks to our convention on filtered derived categories, we have $\DF(\Af)=\bigcup_{n\in\Z}\DF(\leq n)=\bigcup_{n\in\Z}\DF(\geq n)$.
Similarly, if $?\in\{+,-,b\}$ and $n\in\Z$, we set
$\DF^?(\leq n)=\DF(\leq n)\cap\DF^?(\Af)$ and $\DF^?(\geq n)=\DF(\geq n)\cap\DF^?(\Af)$.
For $?\in\{\varnothing,+,-,b\}$, the functor $i$ induces an equivalence of categories $\D^?(\Af)\ra\DF^?(\leq 0)\cap\DF^?(\geq 0)$.

\subsubsection*{f-categories}

In Appendix~A of~\cite{Be1}, Beilinson introduced f-categories over triangulated categories, that have all the abstract properties of filtered
derived categories, and generalized the properties of filtered derived categories to this more general setting. We review his definition and results.
Note that Section~6 of Schn\"urer's paper \cite{Schnur} gives more detailed proofs of many of the results of Appendix~A of~\cite{Be1}.

\begin{subdef}[Definition~A.1 of~\cite{Be1}]
\begin{itemize}
\item[(1)] A \emph{filtered triangulated category}, or \emph{f-category}, is the data of:
\begin{itemize}
\item a triangulated category $\DF$;
\item two full triangulated subcategories $\DF(\leq 0)$, $\DF(\geq 0)$ of $\DF$ that are stable by isomorphisms;
\item a triangulated self-equivalence $s:\DF\ra\DF$ (called \emph{shift of filtration});
\item a morphism of functors $\alpha:\id_{\DF}\ra s$;

\end{itemize}
satisfying the following conditions, where, for every $n\in\Z$, we set
\[\DF(\leq n)=s^n\DF(\leq 0)\mbox{ and }\DF(\geq n)=s^n\DF(\geq 0):\]
\begin{itemize}
\item[(i)] We have $\DF(\geq 1)\subset\DF(\geq 0)$, $\DF(\leq 1)\supset\DF(\leq 0)$ and
\[\DF=\bigcup_{n\in\Z}\DF(\leq n)=\bigcup_{n\in\Z}\DF(\geq n).\]
\item[(ii)] For any $X\in\Ob\DF$, we have $\alpha_X=s(\alpha_{s^{-1}(X)})$.
\item[(iii)] For any $X\in\Ob\DF(\geq 1)$ and $Y\in\Ob\DF(\leq 0)$, we have $\Hom(X,Y)=0$, and the maps $\Hom(s(Y),X)\ra\Hom(Y,X)\ra\Hom(Y,s^{-1}(X))$ induced
by $\alpha_Y$ and $\alpha_{s^{-1}(X)}$ are bijective.
\item[(iv)] For every $X\in\Ob\DF$, there exists a distinguished triangle $A\ra X\ra B\flnom{+1}$ with $A$ in $\DF(\geq 1)$ and $B$ in $\DF(\leq 0)$.
\end{itemize}

\item[(2)] If $\DF$ and $\DF'$ are f-categories, an \emph{f-functor} from $\DF$ to $\DF'$ is the data of a triangulated functor $T:\DF\ra\DF'$ and a natural isomorphism
$s'\circ T\iso T\circ s$ 
such that $T(\DF(\leq 0))\subset\DF'(\leq 0)$,
$T(\DF(\geq 0))\subset\DF(\geq 0)$ and that, for every $X\in\Ob\DF$, the following triangle commutes:
\[\xymatrix{T(X)\ar[r]^-{\alpha'_{T(X)}}\ar[rd]_-{T(\alpha(X))} & s'(T(X))\ar[d]^\wr \\
& T(s(X))}\]

\item[(3)] 
Let $\Df$ be a triangulated category. An \emph{f-category over $\Df$} is an f-category $\DF$ together with an equivalence $i:\Df\ra\DF(\leq 0)\cap\DF(\geq 0)$.
If $\Df'$ is another triangulated category, $\DF'$ is an $f$-category over $\Df'$ and $T:\Df\ra\Df'$ is a triangulated functor, an \emph{f-lifting} of $T$ is an f-functor
$FT:\DF\ra\DF'$ and a natural isomorphism $i'\circ T\simeq TF\circ i$.

\end{itemize}
\label{def_f_category}
\end{subdef}

\begin{subex}
Let $\Af$ be an abelian category. For $?\in\{\varnothing,+,-,b\}$, the category $\DF^?(\Af)$ with the subcategories $\DF^?(\leq 0)$ and $\DF^?(\geq 0)$, the functors
$s$ and $i$ and the natural transformation $\alpha$, is an f-category over the triangulated category $\D^?(\Af)$. Note that this is not be true for
$\D^?(\Filf(\Af))$ instead of $\DF^?(\Af)$ (except of course if $?=b$), because the third statement of condition (i) of Definition~\ref{def_f_category} does not hold.
\label{ex_fil_der}
\end{subex}

\begin{subprop}[Proposition~A.3 of~\cite{Be1}]
Let $\DF$ be an f-category.
\begin{itemize}
\item[(i)] For every $n\in\Z$, the inclusion $\DF(\leq n)\subset\DF$ admits a left adjoint $\sigma_{\leq n}$, and the inclusion $\DF(\geq n)\subset\DF$ admits a right adjoint $\sigma_{\geq n}$.
The functors $\sigma_{\leq n}$, $\sigma_{\geq n}$ are triangulated and preserve the subcategories $\DF(\leq m)$, $\DF(\geq m)$ for every $m\in\Z$.
\item[(ii)] For $a,b\in\Z$, there exists a unique isomorphism of functors $\sigma_{\leq a}\sigma_{\geq b}\simeq\sigma_{\geq b}\sigma_{\leq a}$ that makes the following diagram commute:
\[\xymatrix@C=10pt{\sigma_{\geq b}\ar[rr]\ar[rd] & & \id_{\DF}\ar[rr] & & \sigma_{\leq a} \\
 & \sigma_{\leq a}\sigma_{\geq b}\ar[rr]_-\sim && \sigma_{\geq b}\sigma_{\leq a}\ar[ru] &
}\]

\item[(iii)] Let $X\in\Ob\DF$. Then there exists a unique morphism $\delta:\sigma_{\leq 0}X\ra\sigma_{\geq 1}[1]$ making the triangle $\sigma_{\leq 1}X\ra X\ra\sigma_{\leq 0}X\flnom{\delta}\sigma_{\geq 1}X[1]$
distinguished.
Any distinguished triangle $A\ra X\ra B\flnom{+1}$ with $A\in\Ob\DF(\geq 1)$ and $B\in\Ob\DF(\leq 0)$ admits a unique isomorphism to the triangle of the previous sentence.

\item[(iv)] We have canonical isomorphisms $\sigma_{\leq n}\circ s=s\circ\sigma_{\leq n-1}$ and $\sigma_{\geq n}\circ s=s\circ\sigma_{\geq n-1}$.

\end{itemize}

\label{prop_f_category1}
\end{subprop}

Point (iv) is not stated in Proposition~A.3 of~\cite{Be1} but follows immediately from the fact that $s(\DF(\leq n -1))=\DF(\leq n)$ (resp. $s(D(\geq n-1))=D(\geq n)$) and the uniqueness
of adjoints.

\begin{subdef}
Let $\Df$ be a triangulated category and $\DF$ be an f-category over $\Df$. For every $n\in\Z$, we define a functor $\Gr^n:\DF\ra\Df$ by
$\Gr^n=i^{-1}\circ s^{-n}\circ\sigma_{\leq n}\sigma_{\geq n}$.

\end{subdef}

\begin{subprop}
Let $\Df$ be a triangulated category and $\DF$ be an f-category over $\Df$. 
\begin{itemize}
\item[(i)] For every $r\in\Z$, we have a natural isomorphism $\Gr^r\circ s=\Gr^{r-1}$. 

\item[(ii)] Let $r\in\Z$. Then $\Gr^r\circ i=0$ if $r\ne 0$ and $\Gr^r\circ i\simeq\id_\Df$ if $r=0$.

\item[(iii)] Let $r,n\in\Z$. We have
\[\Gr^r\circ\sigma_{\leq n}=\left\{\begin{array}{ll}\Gr^r& \mbox{if }r\leq  n\\ 0 & \mbox{otherwise}\end{array}\right.
\qquad\mbox{and}\quad\Gr^r\circ\sigma_{\geq n}=\left\{\begin{array}{ll}\Gr^r& \mbox{if }r\geq  n\\ 0 & \mbox{otherwise}\end{array}\right..\]

\end{itemize}
\label{prop_Gr}
\end{subprop}

\begin{proof}
Point (i) follows from Proposition~\ref{prop_f_category1}(iv), point (ii) from the fact that the image of $i$ is contained in $\DF(\leq 0)\cap\DF(\geq 0)$, and point (iii)
from the definition of $\Gr^r$.
\end{proof}

\begin{subprop}[Proposition~A.3 of~\cite{Be1}]
Let $\Df$ be a triangulated category and $\DF$ be an f-category over $\Df$.
Then there exists a triangulated functor $\omega:\DF\ra\Df$ such that:\footnote{Note that there is a typo in Proposition~A.3 of~\cite{Be1}: the left and right
adjoints are switched; see the correction in Proposition~6.6 of~\cite{Schnur}.}
\begin{itemize}
\item[(a)] $\omega_{\mid\DF(\leq 0)}:\DF(\leq 0)\ra\Df$ is left adjoint to $\Df\flnom{i}\DF(\leq 0)\cap\DF(\geq 0)\subset\DF(\leq 0)$;
\item[(b)] $\omega_{\mid\DF(\geq 0)}:\DF(\geq 0)\ra\Df$ is right adjoint to $\Df\flnom{i}\DF(\leq 0)\cap\DF(\geq 0)\subset\DF(\geq 0)$;
\item[(c)] for any $X\in\Ob\DF$, the map $\omega(\alpha_X):\omega(X)\ra\omega(s(X))$ is an isomorphism;
\item[(d)] if $A\in\Ob\DF(\leq 0)$ and $B\in\DF(\geq 0)$, then $\omega:\Hom(A,B)\ra\Hom(\omega(A),\omega(B))$ is bijective.

\end{itemize}
Moreover, $\omega$ is determined up to unique isomorphism by properties (a) and (c) (resp. (b) and (c)).

\label{prop_f_category2}
\end{subprop}

\begin{subrmk}
If $\Af$ is an abelian category, $?\in\{\varnothing,+,-,n\}$, $\DF=\DF^?(\Af)$ and $\Df=\D^?(\Af)$, then the functors $\sigma_{\leq n}$, $\sigma_{\geq n}$, $\Gr^n$ and
$\omega$ are isomorphic to the ones defined in the first part of this section.

\end{subrmk}

The following proposition follows easily from the definitions.

\begin{subprop}
Let $\Df,\Df'$ be triangulated categories, and let $\DF$ (resp. $\DF'$) be an f-category over $\Df$ (resp. $\Df'$).
Let $T:\Df\fl\Df'$ be a triangulated functor, and let $FT:\DF\ra\DF'$ be an f-lifting of $T$.
Then the following squares commute up to natural isomorphism:
\[\xymatrix{\DF\ar[r]^-{FT}\ar[d]_\omega&\DF'\ar[d]^\omega \\
\Df\ar[r]_-T& \Df'}\quad
\xymatrix{\DF\ar[r]^-{FT}
\ar[d]_{\Gr^n}&\DF'\ar[d]^{\Gr^n}\\ \Df\ar[r]_-T&\Df'}\quad
\xymatrix{\DF\ar[r]^-{FT}
\ar[d]_{\sigma_{\leq n}}&\DF'\ar[d]^{\sigma_{\leq n}}\\ \DF\ar[r]_-{FT}&\DF'}\]

\label{prop_f_lifting}
\end{subprop}

\subsubsection*{Construction of f-categories and f-liftings}

\begin{subprop}
Let $\Df$ be a triangulated category, $\DF$ be an f-category over $\Df$, and $\Df'$ be a full triangulated subcategory of $\Df$ that is stable by isomorphisms. We define a full
subcategory $\DF'$ of $\DF$ by
\[\Ob\DF'=\{K\in\Ob\Df\mid\forall r\in\Z,\ \Gr^rK\in\Ob\Df'\}.\]
Then $\DF'$ is a triangulated subcategory of $\DF$, it is stable by isomorphisms, we have $s(\DF')\subset\DF'$ and $i(\Df')\subset\DF'$. The data of
$\DF'$, $\DF'\cap\DF(\leq 0)$, $\DF'\cap\DF(\geq 0)$, $s_{\mid\DF'}:\DF'\ra\DF'$, $\alpha$ and $i_{\mid\Df'}:\Df'\ra\DF'$ defines an f-category over $\Df'$.

\label{prop_f_category3}
\end{subprop}

\begin{proof}
As the functors $\Gr^r$ are triangulated, $\DF'$ is a triangulated subcategory of $\DF$; it is clearly stable by isomorphisms, and it stable by $s$ thanks to the isomorphisms
$\Gr^r\circ s=\Gr^{r-1}$ (Proposition~\ref{prop_Gr}(i)). 
If $X\in\Ob\Df'$, then $\Gr^r(i(X))=0$ if $r\ne 0$ and $\Gr^0(i(X))\simeq X$ (Proposition~\ref{prop_Gr}(ii)), so $i(X)\in\Ob\DF'$. To prove the last assertion, we check the conditions of
Definition~\ref{def_f_category}. Conditions (i)-(iii) are clear. To check condition~(iv), it suffices by Proposition~\ref{prop_f_category1}(iii) to prove that the functors
$\sigma_{\leq n}$, $\sigma_{\geq n}$ preserve $\DF'$; but this follows immediately from Proposition~\ref{prop_Gr}(iii).
\end{proof}

The next proposition, which follows easily from the definitions, is used to construct an f-category over the triangulated category of horizontal constructible complexes
in Section~\ref{section_def_hor}.

\begin{subprop}
Let $\Uf$ be a small category, let $(\DF_x)_{x\in\Ob\Uf}$ be an inductive system of f-categories where all transition functors are f-functors, let $(\Df_x)_{x\in\Ob\Uf}$ be
an inductive system of triangulated categories where all transition functors are triangulated functors. Suppose that, for every $x\in\Ob\Uf$, we have a functor
$i_x:\Df_x\ra\DF_x$ making $\DF_x$ an f-lifting of $\Df_x$ and that, for every morphism $x\ra y$ of $\Uf$, the diagram
\[\xymatrix{\DF_x\ar[r] & \DF_y \\
\Df_x\ar[u]^{i_x}\ar[r] &\Df_y\ar[u]_{i_y}}\]
commutes up to isomorphism; we also suppose that these isomorphisms are compatible with the composition of morphisms of $\Uf$.

Then $\DF:=2-\varinjlim_{x\in\Ob\Uf}\DF_x$ is an f-category over $\Df:=2-\varinjlim_{x\in\Uf}\Df_x$ and, for every $x\in\Ob\Uf$, the canonical functor $\DF_x\ra\DF$ is an f-lifting
of the canonical functor $\Df_x\ra\Df$.
\label{prop_f_colimit}
\end{subprop}

Our main source of f-liftings will be filtered derived functors, so we recall how they are constructed.
Let $\Af,\Af'$ be abelian categories, and let $T:\Af\ra\Af'$ be a left exact functor. 
If $(A,F^\bullet)$ is an object of $\Filf(\Af)$, we define a filtration $F^\bullet$ on $T(A)$ by setting $F^i(T(A))=\Im(T(F^iA)\ra T(A))$ for every $i\in\Z$; in
particular, we have canonical morphisms $T(F^iA)\ra F^i(T(A))$, hence also $T(\Gr^i(A))\ra\Gr^i(T(A))$ (as $T$ is left exact, we have
$T(\Gr^i(A))\simeq T(F^i(A))/T(F^{i+1}(A))$).
This defines an additive functor $\Filf(\Af)\ra\Filf(\Af')$ (the action on morphisms is the restriction of that of $T$), that we denote by $\Filf(T)$. 

If $\Filf(T)$ admits a right derived functor $R\Filf(T)$ in the sense of Definition~1.3.1 of Schneiders's book~\cite{Schn}, we see that $R\Filf(T)$ is the
\emph{filtered right derived functor} of $F$. 
\footnote{The definition of a right derived functor in this setting is analogous to the usual one:
a right derived functor of $\Filf(T)$ is a left Kan extension of $\Ko^+(\Filf(\Af))\flnom{\Ko^+(\Filf(T))}\Ko^+(\Filf(\Af'))\ra\D^+(\Filf(\Af'))$
along the canonical functor $\Ko^+(\Filf(\Af))\ra\D^+(\Fil(\Af))$.
}

\begin{subprop}
Let $T:\Af\ra\Af'$ be a left exact functor between abelian categories.
If $\Af$ has a $T$-injective subcategory in the
sense of Definition 13.3.4 of \cite{KS1}, then $\Filf(T):\Filf(\Af)\ra\Filf(\Af')$ has a right derived functor, 
and the squares
\[\xymatrix@C=40pt{\Dpl(\Filf(\Af))\ar[r]^-{R\Filf(T)}\ar[d]_\omega&\Dpl(\Filf(\Af'))\ar[d]^\omega \\
\Dpl(\Af)\ar[r]_-{RF}& \Dpl(\Af')}\quad
\xymatrix@C=40pt{\Dpl(\Filf(\Af))\ar[r]^-{R\Filf(T)}
\ar[d]_{\Gr}&\Dpl(\Filf(\Af'))\ar[d]^{\Gr}\\ \Dpl(\Af)\ar[r]_-{RT}&\Dpl(\Af')}\]
\[\xymatrix@C=40pt{\Dpl(\Filf(\Af))\ar[r]^-{R\Filf(T)}
\ar[d]_{\sigma_{\leq i}}&\Dpl(\Filf(\Af'))\ar[d]^{\sigma_{\leq i}}\\ \DFpl(\Af)\ar[r]_-{RT}&\DFpl(\Af')}
\]
commute up to natural isomorphism.

\label{prop_fil_der}
\end{subprop}

In particular, if $RT$ sends $\Db(\Af)$ to $\Db(\Af')$, then $R\Filf(T)$ sends $\DFb(\Af)$ to $\DFb(\Af')$, because $\DFb(\Af')$ is the full subcategory of $\Dpl(\Filf(\Af'))$ whose
objects are the $K$ such that $\Gr(K)$ is in $\Db(\Af')$. In that case, $R\Filf(T)$ will be an f-lifting of $RT$.

Before we prove the proposition, we introduce some definitions.
We assume that $\Af$ has a $T$-injective subcategory $\If$. By Proposition~13.3.5 of \cite{KS1}, this implies that $T$ has a right derived functor $RT:\D^+(\Af)\ra\D^+(\Af')$;
we say that an object $A$ of $\Af$ is \emph{$T$-acyclic} if $R^rT(A)=0$ for every $r\geq 1$. Then every object of $\If$ is $T$-acyclic (Remark~13.3.6(ii) of~\cite{KS1}), so we may assume
that $\If$ is the full subcategory of $T$-acyclic objects. The hypothesis on $\Af$ implies in particular that every object of $\Af$ injects into a $T$-acyclic object.

We say that an object $X$ of $\Filf(\Af)$ is \emph{filtered $T$-acyclic} if $\Gr^i(X)$ is $T$-acyclic for every $i\in\Z$; we denote the full subcategory of filtered $T$-acyclic objects
by $\Filf(\If)$. Let $X\in\Filf(\Af)$ be filtered $T$-acyclic; as $T$-acyclic objects are stable by extensions and by taking
cokernels of monomorphisms, and as the filtration on $X$ is finite, the object $\Fil^i(X)/\Fil^j(X)$ is $T$-acyclic for all $i\leq j$ in $\Z$, and in particular all $\Fil^i(X)$ and $X$ itself are
$T$-acyclic. 

\begin{sublemma}
\begin{itemize}
\item[(i)] For every $X\in\Ob\Filf(\Af)$, there exists a strict monomorphism $X\ra Y$ with $Y$ filtered $T$-acyclic.

\item[(ii)] Let $0\ra X\ra Y\ra Z\ra 0$ be a strictly exact sequence in $\Filf(\Af)$. If $X$ and $Y$ are filtered $T$-acyclic, then so is $Z$.

\item[(iii)] Let $X\in\Ob(\Filf(\Af))$ be filtered $T$-acyclic. Then, for every
$i\in\Z$, the canonical morphisms $T(\Fil^i(X))\ra \Fil^i(T(X))$ and $T(\Gr^i(X))\ra\Gr^i(T(X))$ are isomorphisms.

\item[(iv)] Let $0\ra X\ra Y\ra Z\ra 0$ be a strictly exact sequence in $\Filf(\Af)$. If $X$, $Y$ and $Z$ are filtered $T$-acyclic, then the sequence
$0\ra\Filf(T)(X)\ra\Filf(T)(Y)\ra\Filf(T)(Z)\ra 0$ is strictly exact.

\end{itemize}
\label{lemma_fil_der}
\end{sublemma}

\begin{proof}
Point (i) is proved exactly like \cite[\href{https://stacks.math.columbia.edu/tag/05TS}{Lemma 05TS}]{stacks-project}.

Let $0\ra X\ra Y\ra Z\ra 0$ be an exact sequence in $\Filf(\Af)$. Then it is strictly exact if and only if the sequence $0\ra\Gr^i(X)\ra\Gr^i(Y)\ra\Gr^i(Z)\ra 0$ is exact for every $i\in\Z$; as
$T$-acyclic objects are stable by taking cokernels of monomorphisms, this gives (ii).
Point (iv) also follows from this observation and from point (iii), as $T$ sends short exact sequences of
$T$-acyclic objects to short exact sequences.

It remains to prove (iii). So suppose that $X\in\Ob(\Filf(\Af))$ is filtered $T$-acyclic. By the paragraph before the statement of the lemmas, this implies
that all the $\Fil^i(X)$, $X/\Fil^i(X)$ and $\Gr^i(X)$ (and in particular $X$ itself) are $T$-acyclic.
In particular, if we apply $T$ to the short exact sequence of $T$-acyclic objects
\[0\ra\Fil^i(X)\ra X\ra X/\Fil^i(X)\ra 0,\]
we get a short exact sequence; this implies that $T(F^i(X))\ra T(X)$ is injective, whence the first assertion of (iii). Now we apply $T$ to the short
exact sequence of $T$-acyclic objects
\[0\ra\Fil^{i+1}(X)\ra\Fil^i(X)\ra\Gr^i(X)\ra 0;\]
this gives again a short exact sequence, so we get that
\[T(\Gr^i(X))=\Coker(T(\Fil^{i+1}(X))\ra T(\Fil^i(X))).\]
This, together with the first assertion of (iii), gives the second assertion of (iii).
\end{proof}

\begin{proof}[Proof of~Proposition~\ref{prop_fil_der}]
In order to prove that $\Filf(T)$ has a right derived functor, it suffices by Proposition~1.3.4 of~\cite{Schn} to show that $\Filf(\If)$ is a $\Filf(T)$-injective subcategory of
$\Filf(\Af)$ in the sense of Definition~1.3.2 of~\cite{Schn}; but this is exactly points (i), (ii) of (iv) of Lemma~\ref{lemma_fil_der}.
This also gives a way to compute the right derived functor: for every $K\in\Dpl(\Filf(\Af))$, there exists be a
filtered quasi-isomorphism (i.e. a morphism $\alpha$ of complexes of $\Filf(\Af)$ such that $\Gr(\alpha)$ is a quasi-isomorphism) $\alpha:K\ra I$ with $I$ a complex of
filtered $T$-acyclic objects, and we have $R\Filf(T)(K)=\Filf(T)(I)$.
As the complexes $\omega(I)$, $\Gr(I)$ and $\sigma_{\leq i}(I)=\Fil^i(I)$ are bounded below complexes of $T$-acyclic objects of $\Af$, 
and as $\omega(\alpha)$, $\Gr(\alpha)$ and $\sigma_{\leq i}(\alpha)$ are quasi-isomorphisms by \cite[\href{https://stacks.math.columbia.edu/tag/05S3}{Lemma 05S3}]{stacks-project}, we also
have $RF(\omega(I))=F(\omega(I))=\omega(F(I))$, $RF(\Gr(I))=F(\Gr(I))\simeq\Gr(F(I))$ and $RF(\sigma_{\leq i}(I))=F(\sigma_{\leq i}(I))\simeq\sigma_{\leq i}(F(I))$
(the last isomorphisms are given by point (iii) of Lemma~\ref{lemma_fil_der}). This gives the 
commutativity of the three squares. 
\end{proof}

\subsection{t-structures and the realization functor}
\label{section_f_real}

We review the construction of the realization functor from \cite{BBD} 3.1 and Appendix~A of~\cite{Be1}.

\begin{subdef}
Let $\Df$ be a triangulated category and $\DF$ be an f-category over $\Df$. Suppose that we are given
a t-structure $(\Df^{\leq 0},\Df^{\geq 0})$ on $\Df$ and a t-structure $(\DF^{\leq 0},\DF^{\geq 0})$ on $\DF$.
We say that these t-structures are \emph{compatible} if $i:\Df\ra\DF$ is t-exact and if $s(\DF^{\leq 0})=\DF^{\leq -1}$.

\end{subdef}

\begin{subprop}[Proposition~A.5 of~\cite{Be1}]
Let $\Df$ be a triangulated category and $\DF$ be an f-category over $\Df$. Suppose that we are given
a t-structure $(\Df^{\leq 0},\Df^{\geq 0})$ on $\Df$. Then there exists a unique t-structure on $\DF$ compatible
with $(\Df^{\leq 0},\Df^{\geq 0})$, and it is given by
\[\Ob\DF^{\leq 0}=\{X\in\Ob\DF\mid \forall i\in\Z,\ \Gr^i X[i]\in\Ob\Df^{\leq 0}\}\]
\[\Ob\DF^{\geq 0}=\{X\in\Ob\DF\mid \forall i\in\Z,\ \Gr^i X[i]\in\Ob\Df^{\geq 0}\}.\]
\label{prop_t_DF}
\end{subprop}

\begin{subthm}[Proposition~A.5 and~A.6 of~\cite{Be1}]
Let $\Df$ be a triangulated category and $\DF$ be an f-category over $\Df$. Suppose that we are given compatible
t-structures $(\Df^{\leq 0},\Df^{\geq 0})$ on $\Df$ and $(\DF^{\leq 0},\DF^{\geq 0})$ on $\DF$. Denote by
$H:\Df\ra\Cf:=\Df^{\leq 0}\cap\Df^{\geq 0}$ the cohomology functor. We define a functor
$H_F:\DF\ra\Cb(\Cf)$ in the following way: if $X\in\Ob\DF$, we set $H_F(X)^i=H^i\Gr^i(X)$, and we
take as differential $H_F(X)^i\ra H_F(X)^{i+1}$ the map induced from the connection morphism in the 
distinguished triangle $\omega(\sigma_{\leq i+1}\sigma_{\geq i+1}(X))\ra\omega(\sigma_{\leq i+1}\sigma_{\geq i}(X))\ra
\omega(\sigma_{\leq i}\sigma_{\geq i}(X))\flnom{+1}$.

\begin{itemize}
\item[(i)] The functor $H_F$ is well-defined, its restriction to the heart $\Cf_F$ of $(\DF^{\leq 0},\DF^{\geq 0})$ is an equivalence
of categories $G:\Cf_F\iso\Cb(\Cf)$, and $G^{-1}\circ H_F:\DF\ra\Cf_F$ is the cohomology functor of the t-structure
$(\DF^{\leq 0},\DF^{\geq 0})$.

\item[(ii)] The functor $\omega\circ G^{-1}:\Cb(\Cf)\ra\Df$ factors through $\Db(\Cf)$.

\end{itemize}
\label{thm_real}
\end{subthm}

\begin{subdef}
In the situation of Theorem~\ref{thm_real}, we call the functor $\Db(\Cf)\ra\Df$ induced by $\omega\circ G^{-1}$ the
\emph{realization functor} and denote it by $\real$.

\label{def_real}
\end{subdef}

\begin{subex}
Let $\Af$ be an abelian category,
let $\Df$ be a full
triangulated subcategory of $\Db(\Af)$, let $(\Df^{\leq 0},\Df^{\geq 0})$
be a t-structure on $\Df$, and denote its heart by $\Cf$. 
Let $\DF$ be the full subcategory of $K$ in $\DFb(\Af)$ such that
$\Gr^i K\in\Ob\Df$ for every $i\in\Z$. This is an f-category over $\Df$ by
Proposition~\ref{prop_f_category3}.
By Proposition~\ref{prop_t_DF}, the t-structure of $\Df$
also lifts to a compatible t-structure $(\DF^{\leq 0},\DF^{\geq 0})$ on $\DF$.
The heart of this t-structure is the abelian category 
with objects
\[\{K\in\Ob\DFb(\Af)|\forall i\in\Z,\ \Gr^i K[i]\in\Ob\Cf\}.\]
It is the category called ``$\Df F_{\text{b\^ete}}$'' in
\cite{BBD} 3.1.7. 
If $(K,F^\bullet)$ is an object of this category, then the sequence
\[\ldots\fl\Gr^i K[i]\fl\Gr^{i+1}K[i+1]\fl\Gr^{i+2}K[i+2]\fl\ldots\]
is a bounded complex of objects of $\Cf$, which is the image of $(K,F^\bullet)$ by the
functor $G$ of Theorem~\ref{thm_real}.

\end{subex}

\subsection{The realization functor and f-liftings}

Now we come to the second goal of this subsection. We start with some preliminaries.
Let $\Df,\Df'$ be triangulated categories and $T:\Df\ra\Df'$ be a triangulated functor. Suppose that we are given a t-structure $(\Df^{\leq 0},\Df^{\geq 0})$ (resp. $({\Df'}^{\leq 0},{\Df'}^{\geq 0})$)
on $\Df$ (resp. $\Df'$), and denote the cohomology functors of this t-structure by $\H^i$ and its heart by $\Cf$ (resp. $\Cf'$).
We say that an object $X$ of $\Cf$ is \emph{$T$-acyclic} if $T(X)\in\Ob\Cf'$. If $X$ is $T$-acyclic, then we have $\H^n T(X)=0$ for every $n\in\Z\setminus\{0\}$; 
the converse if true if the t-structure on $\Df'$ is nondegenerate.

\begin{sublemma}
\begin{itemize}
\item[(i)] The full subcategory of $T$-acyclic objects of $\Cf$ is stable by extensions.

\item[(ii)] Let $0\ra X\ra Y\ra Z\ra 0$ be an exact sequence in $\Cf$. If $X,Y,Z$ are $T$-acyclic, then $0\ra T(X)\ra T(Y)\ra T(Z)\ra 0$
is an exact sequence in $\Cf'$.

\item[(iii)] Let $(X^\bullet,d^\bullet)$ be a complex of objects of $\Cf$ and let $k\in\Z$. If $X^k$ is $T$-acyclic and $\H^{k+1}(X^\bullet,d^\bullet)=0$, then we have
$\H^r T(\Ker d^{k+1})\simeq\H^{r+1}T(\Ker d^k)$ for every $r\in\Z\setminus\{-1,0\}$.

\item[(iv)] Suppose that the t-structure on $\Df'$ is non-degenerate.
Let $(X^\bullet,d^\bullet)$ be an exact complex of $T$-acyclic objects. Suppose that at least one of the following conditions hold:
\begin{itemize}
\item[(a)] The complex $(X^\bullet,d^\bullet)$ is bounded.
\item[(b)] There exists $N\in\Nat$ such that $T(X)\in{\Df'}^{[-N,N]}$ for every $X\in\Ob\Cf$.

\end{itemize}
Then the complex $T(X^\bullet)$ of objects of $\Cf'$ is exact.

\item[(v)] Suppose that the t-structure on $\Df'$ is non-degenerate and
that there exists $N\in\Nat$ such that $T(X)\in{\Df'}^{[-N,N]}$ for every $X\in\Ob\Cf$.
Let $(X^\bullet,d^\bullet)$ be a complex of $T$-acyclic objects of $\Cf$. If $X^\bullet$ is quasi-isomorphic to a bounded complex, then, for $n\in\Nat$ big enough,
the complexes $\tau_{\leq n}X^\bullet$, $\tau_{\geq -n}X^\bullet$ and $\tau_{\leq n}\tau_{\geq -n}X^\bullet$ are complexes of $T$-acyclic objects, and all the maps in the square
\[\xymatrix{\tau_{\leq n}X^\bullet\ar[r]\ar[d] & \tau_{\leq n}\tau_{\geq -n}X^\bullet\ar[d]\\
X^\bullet\ar[r] & \tau_{\geq -n}X^\bullet}\]
are quasi-isomorphisms. In particular, $X^\bullet$ is quasi-isomorphic to a bounded
complex of $T$-acyclic objects.

\end{itemize}
\label{lemma_der_fil}
\end{sublemma}

\begin{proof}
We repeatedly use the fact that a complex $X\ra Y\ra Z$ in $\Cf$ is a short exact sequence if and only if can completed to a distinguished triangle of $\Df$
(see Theorem~1.3.6 of~\cite{BBD}).

Let $0\ra X\ra Y\ra Z\ra 0$ be an exact sequence in $\Cf$. If $X,Z$ are $T$-acyclic, then we have an exact triangle $T(X)\ra T(Y)\ra T(Z)\flnom{+1}$ with
$T(X),T(Z)$ in $\Cf'$, so $T(Y)$ is in $\Cf'$ and the sequence $0\ra T(X)\ra T(Y)\ra T(Z)\ra 0$ is exact in $\Cf'$. This proves (i) and (ii).

In the situation of (iii), we have an exact sequence
\[0\ra\Ker d^k\ra X^k\flnom{d^k} \Im(d^k)=\Ker(d^{k+1})\ra 0,\]
hence a distinguished triangle $T(\Ker d^k)\ra T(X^k)\ra T(\Ker d^{k+1})\flnom{+1}$.
As $\H^r T(X^k)=0$ for $r\ne 0$, the conclusion of (iii) follows from the long exact cohomology sequence of this triangle.

Suppose that we are in the situation of (iv). Let $k\in\Z$, and let $r$ be a positive integer. By (iii), we have isomorphisms
$\H^r T(\Ker d^k)\simeq\H^{r+l}T(\Ker d^{k-l})$ and $\H^{-r}T(\Ker d^k)\simeq\H^{-r-l}T(\Ker d^{k+l})$ for every $l\in\Nat$. Also, we have $\H^{r+l}T(\Ker d^{k-l})=0$ and $\H^{-r-l}T(\Ker d^{k+l})=0$
for $l$ big enough; indeed, if (a) holds, this is true because $X^{k+l}=0$ and $X^{k-l}=0$ for $l$ big enough, and if (b) holds, this is true as soon as $l\geq N$. We deduce that $\H^rT(\Ker d^k)=0$
and $\H^{-r}(\Ker d^k)=0$ for every $k\in\Z$ and every positive integer $r$, hence that all $\Ker d^k$ are $T$-acyclic. The conclusion of (iv) then follows by applying (ii) to the short exact
sequences $0\ra\Ker d^k\ra X^k\ra\Ker d^{k+1}\ra 0$.

Finally, suppose that we are in the situation of (v). As $X^\bullet$ is quasi-isomorphic to a bounded complex, there exists $M\in\Nat$ such that $\H^r(X^\bullet)=0$ for $r\not\in[-M,M]$.
Let $k\in\Nat$. If $k\geq M$ and $r$ is a positive integer, then we have by (iii):
\[\H^{-r}T(\Ker d^k)\simeq\H^{-r-N}T(\Ker d^{k+N})=0\]
and
\[H^r T(\Ker d^{-k})\simeq\H^{r+N}\Ker(d^{-k-N})=0.\]
Similarly, if $k\geq N+M$ and $r$ is a positive integer, then we have by (iii):
\[\H^rT(\Ker d^k)\simeq\H^{r+N}T(\Ker d^{k-N})=0\]
and
\[H^{-r} T(\Ker d^{-k})\simeq\H^{-r-N}\Ker(d^{-k+N})=0.\]
We conclude that $\Ker(d^k)$ is $T$-acyclic for $k\geq N+M$ or $k\leq -N-M$. Also, if $n\leq -N-2$, then $\H^n(X^\bullet)=0$ and $\H^{n+1}(X^\bullet)=0$,
hence $\Coker(d^{n-1})\simeq\Ker(d^{n+1})$. So the two statements of (v) hold for $n\geq N+M+2$.
\end{proof}

The following proposition is essentially proved in Section~A.7 of \cite{Be1} .

\begin{subprop} 
Let $\Df,\Df'$ be triangulated categories, and let $\DF$ (resp. $\DF'$) be an f-category over $\Df$ (resp. $\Df'$).
Suppose that we are given compatible t-structures $(\Df^{\leq 0},\Df^{\geq 0})$ and $(\DF^{\leq 0},\DF^{\geq 0})$ (resp. $({\Df'}^{\leq 0},{\Df'}^{\geq 0})$
and $({\DF'}^{\leq 0},{\DF'}^{\geq 0})$) on $\Df$ and $\DF$ (resp. $\Df'$ and $\DF'$), and denote the hearts of this t-structures 
by $\Cf$ and $\Cf_F$ (resp. $\Cf'$ and $\Cf'_F$). Suppose also that the t-structure on $\Df'$ is non-degenerate.

Let $T:\Df\fl\Df'$ be a triangulated functor. Suppose that the following conditions are satisfied:
\begin{itemize}
\item[(a)] The functor $T$ admits an f-lifting $FT:\DF\ra\DF'$.

\item[(b)] Let $\If:=\{X\in\Cf\mid T(X)\in\Cf'\}$ be the full subcategory of $T$-acyclic objects of $\Cf$. Then the functor
$\Kb(\If)/\Nb(\If)\ra\Db(\Cf)$ is an equivalence, where $\Kb(\If)$ is the category of bounded complexes of objects of $\If$ up to
homotopy and $\Nb(\If)$ is its full subcategory of exact complexes.

\end{itemize}

Then the functor $\Kb(\If)\flnom{\Kb(T)}\Kb(\Cf')\ra\Db(\Cf')$ sends $\Nb(\If)$ to $0$, hence induces a functor
$DT:\Db(\Cf)\ra\Db(\Cf')$, and the following diagram commutes up to natural isomorphism:
\[\xymatrix{\Db(\Cf)\ar[r]^-{DT}\ar[d]_{\real} & \Db(\Cf')\ar[d]^{\real} \\
\Df\ar[r]_-{T} & \Df'}\]

\label{prop_der_fil}
\end{subprop}

\begin{proof} 
The first statement follows from point (iv) of Lemma~\ref{lemma_der_fil}.

We prove the second statement. 
In Theorem~\ref{thm_real}, we defined equivalences $G:\Cf_F\ra\Cb(\Cf)$ and $G':\Cf'_F\ra\Cb(\Cf')$.
By (ii) of the same theorem, the functor $\omega\circ G^{-1}:\Cb(\Cf)\ra\Df$ (resp. $\omega\circ {G'}^{-1}:\Cb(\Df')\ra\Df'$) sends exact complexes to $0$, hence
induces a functor $\Db(\Cf)\ra\Df$ (resp. $\Db(\Cf')\ra\Df'$), which is the realization functor $\real$.
Now let $\If_F$ be the full subcategory of $\Cf_F$ whose objects are the $X$ such that $\Gr^iX[i]\in\Ob\If$ for every $i\in\Z$, i.e. such that $G(X)$
is in $\Cb(\If)$. Proposition~\ref{prop_f_lifting} implies that $FT$ sends
$\If_F$ to $\Cf'_F$, and that the restrictions of $G'\circ FT$ and $\Cb(T)\circ G$ to $\If_F$ are isomorphic. So we get an isomorphism of functors on $\Cb(\If)$:
\[T\circ\omega\circ G^{-1}\simeq\omega\circ FT\circ G^{-1}\simeq\omega\circ{G'}^{-1}\circ\Cb(T).\]
This gives the isomorphism $T\circ\real\simeq\real\circ DT$.
\end{proof}

\begin{subrmk}
\begin{itemize}
\item[(1)] Suppose that we are in the situation of Proposition~\ref{prop_der_fil}.
If moreover $T:\Df\ra\Df'$ is left t-exact and if $\If$ is cogenerating in $\Cf$ (i.e. every object of $\Cf$ has a monomorphism into an object of $\If$), then the functor
$\H^0(T):\Cf\ra\Cf'$ admits a right derived functor $RT:\Dpl(\Cf)\ra\Dpl(\Cf')$ by Proposition~13.3.5 of~\cite{KS1}, and the construction of $RT$ in that proposition shows that
$RT$ sends that $\Db(\Cf)$ to $\Db(\Cf')$ and that $DT$ is the restriction of $RT$ to $\Db(\Cf)$. We have a similar statement if $T$ is right t-exact and $\If$ is generating in $\Cf$.

\item[(2)] By Proposition~10.2.7 of~\cite{KS1}, to check assumption (b) in the statement of Proposition~\ref{prop_der_fil}, it suffices to find triangulated subcategories
$\Df_0=\Kb(\Cf)\supset\Df_1\supset\ldots\supset\Df_r=\Kb(\If)$ of $\Kb(\Cf)$ such that, for every $i\in\{1,\ldots,r-1\}$, one of the following conditions holds:
\begin{itemize}
\item For every $X\in\Ob\Df_i$, there exists a quasi-isomorphism $X\ra Y$ with $Y\in\Ob\Df_{i+1}$.
\item For every $X\in\Ob\Df_i$, there exists a quasi-isomorphism $Y\ra X$ with $Y\in\Ob\Df_{i+1}$.

\end{itemize}

\end{itemize}
\label{rmk_der_fil}
\end{subrmk}

\subsection{Application to horizontal perverse sheaves}
\label{f-horizontal}

In this section, we explain how to construct the f-categories underlying the triangulated categories of the main text,
as well as f-liftings of the triangulated functors between these categories.

\subsubsection*{$\ell$-adic complexes}

Let $X$ be a scheme and $E$ be an algebraic extension of $\Q_\ell$.
We use the notation of Section~\ref{l-adic_complexes}.

By Example~\ref{ex_fil_der} applied to $\Af=\Sh(X_\proet,E)$ and Proposition~\ref{prop_f_category3}
applied to the bounded filtered category of $\Af$ and the full subcategory $\Dbc(X,E)$ of
$\Dpl(\Af)$, we get an f-category $\DFbc(X,E)$ over the triangulated category $\Dbc(X,E)$.

Let $f:X\ra Y$ be a morphism of finite type.
We have triangulated
functors $f_*$, $\Ltimes$ and $\Homf_X$ on $\D^+(X_\proet,E)$, $\D^-(X_\proet,E)\times\D(X_\proet,E)$ and $\D(X_\proet,E)^\circ\times\Dpl(X_\proet,E)$,
and they are all derived functors, so we can extend them
to triangulated functors on the filtered derived categories
$\DFpl(X_\proet,E)$, $\DF^-(X_\proet,E)\times\DF(X_\proet,E)$ and $\DF(X_\proet,E)^\circ\times\DFpl(X_\proet,E)$,
using Proposition~\ref{prop_fil_der}. 
\footnote{For
$\Homf_X$ and $\otimes$, 
we could also use V.2 of \cite{CT1}.}
Next, if $X$ has a dimension function, using the fact that
$\Dpl(X_\proet,E)$ is equivalent to the full subcategory
of $\DFpl(X_\proet,E)$ with objects the $K$ such that
$\Gr^i K=0$ for $i\not=0$, we can see the dualizing
complex $\widehat{K}_X$ as an
object of $\DFpl(X_\proet,E)$, and so we can define $D_X$ on
$\DF(X_\proet,E)$ by $D_X(K)=\Homf_X(K,K_X)$. Finally, we extend
the inverse image functor. The functor $f_*:\Dpl(X_\proet,E)\fl
\Dpl(Y_\proet,E)$ has a left adjoint $f^*$, given by
$f^*K=f^*_\naive K\Ltimes_{f^*_\naive E_X}E_Y$, where
$f^*_\naive$ is the regular pullback functor (see Remark
6.8.15 of \cite{BS}). The functor $f^*_\naive$ is exact and so
extends to $\DFpl(Y_\proet,E)$ by Proposition~\ref{prop_fil_der}, and 
we can see $f^*_\naive E_Y$ and $E_X$ as objects of $\DFpl(X_\proet,E)$,
so $f^*$ also extends. 
Restricting all these functors to the subcategories $\DFbc$, we get f-liftings of
the functors $f_*$, $f^*$, $\otimes$, $\Homf_X$ and $D_X$ (when $X$ has a dimension function
for the last one) on the categories $\Dbc$.

\subsubsection*{Perverse t-structure}
\label{perverse_test}

Fix a scheme $X$ as in section \ref{def_perverse}.
Applying Proposition~\ref{prop_t_DF} 
to the f-lifting $\DFc(X,E)$ of the triangulated category $\Dbc(X,E)$ 
and to the perverse t-structure on $\Dbc(X,E)$,
we get a compatible t-structure on $\DFbc(X,E)$.
Then Theorem \ref{thm_real} gives a triangulated realization functor
$\real:\Db\Perv(X,E)\fl\Dbc(X,E)$ extending the inclusion $\Perv(X,E)\subset\Dbc(X,E)$.

We can apply Proposition~\ref{prop_der_fil} and Remark~\ref{rmk_der_fil} to any functor between
categories $\Dbc(X,E)$ that is t-exact for the perverse t-structures and constructed from
the 6 operations $f_*,f^*,f_!,f^!,\otimes,\Homf_X$.

For example,
if $Y$ is a scheme satisfying the same conditions as $X$, and if 
$T=f_*$ or $T=f_!$ for $f:X\to Y$ quasi-finite affine (resp.
$T=f^*[d]$ for $f:Y\to X$ smooth of relative dimension $d$), then 
we get the
commutative diagrams of point (i) (resp. (ii)) of Proposition~\ref{prop_comp_real1}. 
Similarly, taking for $T=D_X:\Dbc(X,E)^\op\to\Dbc(X,E)$ the duality functor, we
get point (iii) of the same Proposition, and taking $T$ to be an appropriate shift of
the restriction to the generic fiber functor, we get Proposition~\ref{prop_comp_real2}).

\subsubsection*{Horizontal constructible complexes}

We use the notation of Section~\ref{section_def_hor}.
In particular, $k$ is a field of finite type over its prime field, $X$ is a separated
finite type $k$-scheme and $E$ is an algebraic extension of $\Q_\ell$.

We define an f-category $\DFbh(X,E)$ over $\Dbh(X,E)$ using Proposition~\ref{prop_f_colimit}.
First, for every $(A,\X)\in\Ob\Uf X$, we get an f-category $\DFbc(\X,E)$ over $\Dbc(\X,E)$ by
applying Proposition~\ref{prop_f_category3} to the triangulated subcategory $\Dbc(\X,E)$ of
the bounded derived category of pro\'etale sheaves of $E$-modules on $\X$.
Next, if $(A,\X)\ra (A',\X')$ is a morphism in $\Uf X$, then the functor
$\Dbc(\X,E)\ra\Dbc(\X',E)$ is the restriction of the trivial derived functor of
a functor $\Sh(\X_\proet,E)\ra\Sh(X_\proet,E)$, so, by Proposition~\ref{prop_der_fil},
it admits an f-lifting $\DFbc(\X,E)\ra\DFbc(\X',E)$.
We set
\[\DFbh(X,E)=2-\varinjlim_{(A,\X)\in\Ob\Uf X}\DFbc(\X,E);\]
this is an f-category over $\Dbh(X,E)$.

Moreover, if $\eta^*:\Dbh(X,E)\fl \Dbc(X,E)$ is the exact
functor induced
by the restriction functors $\Dbc(\X,E)\fl
\Dbc(\X\otimes_Ak,E)\flnom{u^*}
\Dbc(X,E)$, for $(A,\X,u)\in\Ob\Uf X$, then $\eta^*$ admits an f-lifting, by
Proposition~\ref{prop_der_fil}.

By Proposition~\ref{prop_t_DF}, the perverse t-structure on $\Dbh(X,E)$ of
Section~\ref{section_hor}
lifts to a compatible t-structure on $\DFbh(X,E)$, so
by Theorem \ref{thm_real},
we get a realization functor $\real:\Db\Perv_h(X,E)\fl\Dbh(X,E)$.

\subsubsection*{Mixed perverse sheaves}

We use the notation of Section~\ref{def_mixed}, so $X$ and $E$ are as before.

Applying Proposition~\ref{prop_f_category3} to the f-category $\DFbh(X,E)$ lifting
$\Dbh(X,E)$, we get an f-category $\DFbm(X,E)$ lifting $\Dbm(X,E)$.
The realization functor $\real:D^b\Perv_h(X,E)\fl
\Dbh(X,E)$ restricts to a functor
\[\real:D^b\Perv_m(X,E)\fl\Dbh(X,E),\]
whose essential image is contained in $\Dbm(X,E)$ by
definition of $\Dbm(X,E)$; this is also the realization functor that we would
get by applying Proposition~\ref{prop_t_DF} 
and Theorem \ref{thm_real} to the f-category $\DFbm(X,E)$.

\bibliographystyle{alpha}
\bibliography{sur_Q}

\section*{Acknowledgements}

The author would like to thank 
the \'Ecole Normale Sup\'erieure de Lyon
and the Universit\'e Lyon~1
for their hospitality and support during
the academic year 2017--2018. Most of this papers was written while she was a
professor at Princeton University.
Also, she has greatly benefited from conversations with Florian Ivorra
about the formalism of crossed functors and Beilinson's construction of
nearby and vanishing cycles. Last but not least, she would like to thank
an anonymous referee for numerous insightful comments and corrections.

\bigskip

\small

\noindent
{\textsc{ENS de Lyon site Monod
UMPA UMR 5669 CNRS
46, allée d’Italie
69364 Lyon Cedex 07, FRANCE.
}} \hfill\break
{\texttt{sophie.morel@ens-lyon.fr}}.

\end{document}